\documentclass[a4paper, 11pt]{amsart}

\title{Formal category theory in $\infty$-equipments I}
\author{Jaco Ruit}
\address{Mathematisch Instituut, Universiteit Utrecht, The Netherlands}
\email{j.c.ruit@uu.nl} 

\usepackage{amssymb}
\usepackage{amsfonts}
\usepackage{amsthm}

\usepackage{kpfonts} 
\usepackage[mathscr]{eucal}

\usepackage[left=3cm,right=3cm,top=3.2cm,bottom=3.2cm]{geometry} 

\usepackage{tikz-cd}{}  
\usepackage{enumitem}
\usepackage{hyperref}
\usepackage[noabbrev]{cleveref}
\usepackage{aliascnt}
\usepackage{todonotes}
\usepackage{adjustbox}
\usepackage{xstring} 
\usepackage[dvipsnames]{xcolor}
\usepackage{breakurl}
\usepackage{comment}
 
\newtheorem{theorem}{Theorem}
\newcommand{\jnewtheorem}[2]{
	\newaliascnt{#1}{theorem}
	\newtheorem{#1}[#1]{#2}
	\aliascntresetthe{#1} 
}
\numberwithin{theorem}{section}
\jnewtheorem{lemma}{Lemma}
\jnewtheorem{proposition}{Proposition}
\jnewtheorem{corollary}{Corollary}

\theoremstyle{definition}
\jnewtheorem{definition}{Definition}
\jnewtheorem{example}{Example}
\jnewtheorem{remark}{Remark}
\jnewtheorem{construction}{Construction}
\jnewtheorem{notation}{Notation}
\jnewtheorem{convention}{Convention}

\crefformat{theorem}{Theorem~#2#1#3}
\crefformat{lemma}{Lemma~#2#1#3}
\crefformat{proposition}{Proposition~#2#1#3}
\crefformat{corollary}{Corollary~#2#1#3}
\crefformat{definition}{Definition~#2#1#3}
\crefformat{example}{Example~#2#1#3}
\crefformat{remark}{Remark~#2#1#3}
\crefformat{construction}{Construction~#2#1#3}
\crefformat{notation}{Notation~#2#1#3}
\crefformat{convention}{Convention~#2#1#3}
\crefformat{section}{Section~#2#1#3}
\crefformat{subsection}{Subsection~#2#1#3}

\hypersetup{ 
	pdftitle={Formal category theory in ∞-equipments I},
	pdfauthor={Jaco Ruit},
	bookmarksnumbered=true,     
	bookmarksopen=true,         
	bookmarksopenlevel=1,       
	colorlinks=true,   
	linkcolor = {OliveGreen},
	urlcolor = {OliveGreen},
	citecolor = {OliveGreen},   
	pdfencoding=unicode,      
	pdfstartview=Fit,           
	pdfpagemode=UseOutlines,    
	pdfpagelayout=OneColumn,
	linktocpage,
	breaklinks
}

\newcommand{\cat}[1]{\textup{#1}}
\newcommand{\op}{\mathrm{op}}
\newcommand{\im}{\mathrm{im}}
\newcommand{\map}{\mathrm{Map}}
\def\colim{\qopname\relax m{colim}}
\newcommand{\fun}{\mathrm{Fun}}
\newcommand{\Cat}{\cat{Cat}}
\newcommand{\CAT}{\cat{CAT}}
\newcommand{\sSet}{\cat{sSet}}

\newcommand{\Cocart}{\mathrm{Cocart}}
\newcommand{\lax}{\mathrm{lax}}
\newcommand{\opo}{\textup{\textrm{1-op}}}
\newcommand{\opt}{\textup{\textrm{2-op}}}

\renewcommand{\S}{\mathscr{S}}
\newcommand{\C}{\mathscr{C}}
\renewcommand{\D}{\mathscr{D}}
\renewcommand{\P}{\mathscr{P}}

\newcommand{\E}{\mathscr{E}}

\renewcommand{\Vert}{\mathrm{Vert}}
\newcommand{\Hor}{\mathrm{Hor}}

\newcommand{\F}{\mathscr{F}}
\newcommand{\Seg}{\mathrm{Seg}}

\newcommand{\tp}{\mathrm{t}}
\newcommand{\DFib}{\mathrm{DFib}}
\newcommand{\RFib}{\mathrm{RFib}}

\newcommand{\DbliCat}{\cat{Dbl}\infty\cat{Cat}}
\newcommand{\DblSeg}{\cat{DblSeg}}
\newcommand{\Equip}{\cat{Equip}}
\newcommand{\ev}{\mathrm{ev}}
\newcommand{\conj}{\mathfrak{conj}}
\newcommand{\comp}{\mathfrak{comp}}

\newcommand{\PSh}{\cat{PSh}}
\newcommand{\Sq}{\cat{Sq}}	
\newcommand{\abs}[1]{\left|#1\right|}
\newcommand{\id}{\mathrm{id}}
\newcommand*{\vrectangle}{{\ooalign{\lower.3ex\hbox{$\sqcup$}\cr\raise.4ex\hbox{$\sqcap$}}}}
\newcommand{\cart}[3]{{#1}^\circledast{#2}{#3}_\circledast}

\newcommand{\vop}{\mathrm{vop}}
\newcommand{\Con}{\mathrm{Con}}
\newcommand{\SSEG}{\mathbb{S}\mathrm{eg}}
\newcommand{\CCAT}{\mathbb{C}\mathrm{at}}
\newcommand{\SSPAN}{\mathbb{S}\mathrm{pan}}
\newcommand{\hop}{\mathrm{hop}}
\newcommand{\prof}{\mathrm{Prof}}

\begin{document}

\begin{abstract}
	We generalize proarrow equipments from strict category theory to the $\infty$-cate\-go\-rical setting, 
	introducing the concept of $\infty$-equipments. These are specific  
	 double $\infty$-categories that support an internal higher category theory. This paper explores several examples of $\infty$-equip\-ments, including the prototypical example of the $\infty$-equipment 
	of $\infty$-categories and the more general $\infty$-equipments of internal $\infty$-categories.
	The ultimate objective of this article is to study the basic concepts of category theory within an arbitrary $\infty$-equipment, such as colimits and Kan extensions.
\end{abstract}

\maketitle

\setcounter{tocdepth}{1}
\tableofcontents

\vspace{16mm}

\section{Introduction} 

\textit{Formal category theory} is concerned with distilling the fundamental concepts that enables one to define a category theory internal to an ambient structure such as a 2-category.
The archetypical example is given by the ambient 2-category of (small) categories. Ideally, one would like to obtain an umbrella framework  
that encompasses internal and indexed category theory, enriched category theory, and relative category theory among others. 
In this article, we will extend the existing methods of strict formal category theory to the higher categorical context. 

A similar extension was already made by Riehl and Verity in their successful formal treatise of the foundations 
of $\infty$-category theory. Nowadays, one can find a complete introduction to $\infty$-category theory from this perspective \cite{RiehlVerity}. 
The approach that will be presented 
in this article is closely related to the one by Riehl--Verity. There is also a synthetic approaches outside traditional foundations that build on a homotopical interpretation of Martin-L\"of's dependent type theory, called \textit{homotopy type theory} \cite{hottbook}. 
In \cite{RiehlShulman}, Riehl and Shulman introduce an extension of this type theory called \textit{simplicial homotopy type theory} which yields a formal language to reason about $\infty$-categories.
A nice exposition of this formalism can be found in \cite[Section 4]{RiehlUndergrads}. 

The motivation of this article commences with the observation that there is multitude of variants on $\infty$-categories that are relevant to homotopy theory. 
For instance, in equivariant homotopy theory, the formalism of $G$-$\infty$-categories, i.e.\ $\infty$-categories with a \textit{genuine} $G$-action, 
has proven to be effective, see e.g.\ \cite{EquivThom} \cite{Hilman}. The development of 
$G$-equivariant $\infty$-category theory was initiated by Hill, and motivated by the groundbreaking solution of the Kervaire invariant one problem by Hill--Hopkins--Ravenel \cite{HHR}; see the section on \textit{Hill's program} in \cite{Expose0}.
The $G$-$\infty$-categories are examples of \textit{internal $\infty$-categories}, namely, these are the ones that are internal 
to the (presheaf) $\infty$-topos of $G$-spaces. Other applications of 
internal category theory can be found in motivic homotopy theory \cite{BachmannElmantoHeller} 
and condensed mathematics \cite{BarwickGlasmanHaine} \cite{Wolf}. 
A different variant are the \textit{enriched  $\infty$-categories}, 
and important examples of these include $(\infty,n)$-categories, Gray $\infty$-categories \cite{HeineCatHom}, and spectral and dg-categories \cite[Example 5.11]{HaugsengRect}.
\textit{Relative $\infty$-categories} were studied by Lurie in \cite[Section 4.3.1]{HTT}.

All the aforementioned flavors of $\infty$-categories admit a tailored \textit{category theory}. For instance, these variants support suitable notions of:
\begin{enumerate}
	\item fully faithful functors,
	\item adjunctions,
	\item colimits and limits,
	\item point-wise Kan extensions,
	\item colimit and limit completions,
\end{enumerate}
that are generally different from the corresponding notions of their underlying $\infty$-cate\-gories. The theory of $\infty$-categories internal to presheaf $\infty$-toposes 
was first studied by Barwick--Dotto--Glasman--Nardin--Shah \cite{ExposeI}, Shah \cite{Shah}, and Nardin \cite{Nardin}, under the name of \textit{parametrized 
$\infty$-category theory}.
More generally, category theory internal to arbitrary $\infty$-toposes was studied  by Martini \cite{Martini}, and Martini--Wolf \cite{MartiniWolf}. 
Enriched $\infty$-category theory has been studied by Hinich \cite{Hinich}, and Heine \cite{Heine}. A theory of relative colimits and Kan extensions has been developed by Lurie \cite{HTT}.

This paper provides an approach to systematically studying all these variants of $\infty$-category at once. 
We will build on the work of Wood \cite{Wood}, Verity \cite{Verity}, and Shulman \cite{ShulmanFramedBicats} on strict proarrow equipments, 
to introduce an $\infty$-categorical adaptation of their work: the formalism of \textit{$\infty$-equipments}, certain two-dimensional $\infty$-categorical structures
that support an internal category theory. In this article,  will develop the notions (1) to (4)  internal 
to these structures in \cref{section.formal-category-theory}. There are suitable ambient $\infty$-equipments for enriched $\infty$-category theory \cite{HaugsengEnriched}, internal $\infty$-category theory (\cref{section.equipment-internal-cats}), 
and relative $\infty$-category theory. 

We will now give a more detailed overview of the material covered in this paper.

\subsection{Suitable ambient structures} As $\infty$-categories naturally assemble into an $(\infty,2)$-category, 
$(\infty,2$)-categories form a natural candidate in which to develop formal category theory. For instance, 
one may speak about adjunctions in an $(\infty,2)$-category \cite{RiehlVerityAdj}. 
However, an important insight in the development of formal category theory is that plain $(\infty,2)$-categories are often insufficient to define the correct categorical notions internally \cite[Section 0]{Wood}.
We will illustrate this with some examples.

One could attempt to define a notion of fully faithful arrows in an $(\infty,2)$-category $\C$ as follows. Namely, we could say that an arrow $f : x \rightarrow y$ in $\C$  
is fully faithful if and only if for every object $z \in \C$, the induced post-composition functor 
$
f_!: \C(z,x) \rightarrow \C(z,y)
$
is fully faithful on mapping $\infty$-categories. This recovers the notion of fully faithful functors between $\infty$-categories if $\C$ is the ambient $(\infty,2)$-category of $\infty$-categories, 
but one can check that this definition does not recover the correct notion of 
fully faithful functors when $\C$ is an arbitrary ambient $(\infty,2)$-category of enriched $\infty$-categories in general.

One could also attempt to define a notion of Kan extensions internal to $\C$. Namely, one could say that the datum of a 2-cell 
\[
    \begin{tikzcd}
        i \arrow[dr,"f"name=f]\arrow[d,"w"'] \\
        |[alias=t]|j\arrow[r,"g"'] \arrow[r] & x
        \arrow[from=f,to=t,Rightarrow, shorten <= 6pt]
    \end{tikzcd}
\]
in $\C$ exhibits $g$ as the \textit{left extension} of $f$ if and only if the 2-cell witnesses the copresheaf 
$$
\map_{\C(i,x)}(f, w^*(-)) : \C(j,x) \rightarrow \S
$$
to be corepresentable by $g$.\footnote{Here, $\S$ denotes the $\infty$-category of \textit{spaces} or \textit{$\infty$-groupoids}.} But this definition fails to capture the notion of \textit{point-wise} left Kan extensions already if 
$\C$ is the ambient 2-category of categories; see \cite{CampionCounterEx} for some counter-examples. The bare notion of these 
extensions do not come with a method to compute these as certain colimits that would be desirable in practice.
In this article, we have 
chosen to reserve the name \textit{Kan extension} for point-wise extensions; cf.\ \cref{rem.wke}.

To remedy the insufficiency of $(\infty,2)$-categories to formalize the adequate categorical notions, one would like to endow the ambient $(\infty,2)$-categories with more structure. 
In the strict context, a first successful approach to formal category theory is due to the Australian school of category theorists. In \cite{StreetWalters}, Street and Walters axiomatize 
so-called \textit{Yoneda structures} on 2-categories which allow for good notions of internal category theory.
The name of these structures already reflects what obstructs the development of formal category theory in a plain ambient 2-category: a 2-category 
does not always come with apparent notions of presheaves and representable presheaves, which are crucial to developing an internal category theory.
The work of Street and Walters inspired further approaches to formal category theory. In particular, the approach of Wood via \textit{proarrow equipments}
is of importance to us \cite{Wood}. We have chosen to generalize this strategy to the $\infty$-categorical setting.

\subsection{$\infty$-Equipments} 
To make this generalization, we will take on Verity's \textit{double categorical} perspective 
on proarrow equipments \cite{Verity}.\footnote{The relation to Wood's original formulation can be found in \cite[Definition 1.2.4]{Verity} and \cite[Appendix C]{ShulmanFramedBicats}.}
In the higher context, the formalism of \textit{double $\infty$-categories} was introduced by Haugseng \cite{HaugsengPhD}.
A double $\infty$-category is a 2-dimensional $\infty$-categorical object containing objects, 
 two types of arrows (1-cells) between objects: \textit{vertical arrows} and \textit{horizontal arrows}, 
and a notion of 2-cells. A 2-cell may be pictured as a square
\[
	\begin{tikzcd}
		a \arrow[r, ""'name=f]\arrow[d] & b \arrow[d] \\
		c \arrow[r, ""name=t] & d
		\arrow[from=f, to=t, Rightarrow]
	\end{tikzcd}
\]
where the direction of the arrows match their vertical/horizontal type. Informally, a double $\infty$-category 
may be viewed as the data of these spaces of objects, 1- and 2-cells, together with compatible and coherent composition laws for these cells. Double $\infty$-categories were also studied by Moser \cite{Moser}. 
We will recall the basics on double $\infty$-categories and $(\infty,2)$-categories in \cref{section.prelims}. In particular, we will show that one can extract 
two $(\infty,2)$-categories from a double $\infty$-category: its \textit{vertical fragment}, whose arrows are given by the vertical arrows, and its \textit{horizontal fragment}, with arrows 
given by the horizontal arrows.\footnote{These are really $(\infty,2)$-categories up to \textit{univalence}; see \cref{ssection.dbl-cats}.}

In \cref{section.equipment-infty-cats}, we will recall the construction by Ayala--Francis \cite{AyalaFrancis} of an archetypical example of a double $\infty$-category: the double $\infty$-category 
$$
\infty\CCAT 
$$
of $\infty$-categories. It is described as follows: 
\begin{itemize}
	\item its objects are $\infty$-categories,
	\item its vertical arrows are functors,
	\item its horizontal arrows are  \textit{profunctors}, or families of presheaves: a horizontal arrow $\C \rightarrow \D$ between $\infty$-categories corresponds precisely to a functor
	$\D^\op \times \C \rightarrow \S$,
	\item its 2-cells are given as follows: a 2-cell on the left below \[
\begin{tikzcd}
	\mathscr{A} \arrow[r ,"F"name=f] \arrow[d,"f"'] & \mathscr{B} \arrow[d, "g"] \\
	\C \arrow[r,"G"to=T] & \D
	\arrow[from=f,to=t,Rightarrow, shorten <= 6pt, shorten >= 6pt]
\end{tikzcd}\quad
\Leftrightarrow\quad
\begin{tikzcd}[column sep = large]
	\mathscr{B}^\op \times \mathscr{A} \arrow[dr,"F"name=f]\arrow[d, "g^\op \times f"']  \\
	|[alias=t]|\D^\op \times \C \arrow[r, "G"'] & \S
	\arrow[from=f,to=t,Rightarrow, shorten <= 6pt]
\end{tikzcd}
\]
corresponds precisely to the datum of a natural transformation as displayed on the right.
\end{itemize}
The vertical fragment of this double $\infty$-category is equivalent to Lurie's construction of $(\infty,2)$-category of $\infty$-categories \cite{LurieInfty2}, as shown in \cref{prop.comparison-models-infty2-infty-cats}.
An important feature that $\infty\CCAT$ enjoys is that every vertical arrow $f : \C \to \D$ both corepresents and represents a horizontal arrow. Namely, we may associate the profunctor 
$f_\circledast : \D^\op \times \C \to \S : (d,c) \mapsto \map_\D(d,f(c))$ and $f^\circledast : \C^\op \times \D \to \S : (c,d) \mapsto \map_\D(f(c),d)$.  This (co)representation of horizontal arrows 
can be phrased using the purely double categorical language of \textit{companions} and \textit{conjoints} \cite{GrandisPare} \cite{Comp}. We will recall this notion in \cref{ssection.compconj}.
In his language, $f_\circledast$ is called the companion of $f$, and $f^\circledast$ is called the conjoint of $f$. 

In \cref{section.equipments}, we will define $\infty$-equipments to be the double $\infty$-categories that behave like $\infty\CCAT$ in this regard. We will ask for an $\infty$-equipment to have companions and conjoints for all vertical arrows. 
Informally, we may view an $\infty$-equipment as extra structure on its vertical $(\infty,2)$-category that endows this fragment with a notion of \textit{proarrows}, abstract profunctors, that can be (co)represented 
by the arrows in its vertical fragment via the notion of companions and conjoints. We will show that any $\infty$-equipment admits 
two types of abstract Yoneda embeddings: the companion and conjoint embedding. These are locally fully faithful functors between the vertical and horizontal fragments 
that carry vertical arrows to their companion and conjoint respectively.

The above example of $\infty\CCAT$ may be generalized. In \cref{section.equipment-internal-cats}, we provide a construction of an ambient double $\infty$-category $\CCAT(\E)$ of $\infty$-categories internal to a suitable $\infty$-category $\E$. This 
is also an $\infty$-equipment. For example, when taking $\E$ to be the $\infty$-category of $G$-spaces, we will obtain an ambient equipment for the aforementioned $G$-$\infty$-categories.
The construction of $\CCAT(\E)$ is quite technical, and readers may safely skip the details in this section. We will use $\infty\CCAT$ as the main example to illustrate concepts throughout the paper.

\subsection{Formal category theory} We will demonstrate in \cref{section.formal-category-theory} 
that one can develop a satisfying category theory internal to an $\infty$-equipment.
The theory that we will cover here is inspired and often analogous to that of strict proarrow equipments. Albeit in a different set-up, 
some of the theory dates back to Street and Walters \cite{StreetWalters}, Guitart \cite{Guitart}, and Wood \cite{Wood}. In the context 
of strict double categorical equipments, parts of the formal category theory introduced here can be found in the work of Shulman \cite{ShulmanEnInCats}. 
The formal category theory inside strict equipments plays an important role in the synthetic approach to $\infty$-category theory by Riehl--Verity \cite{RiehlVerityFormal} \cite[Chapter 13]{RiehlVerity} as well,
which 
has also been a source of inspiration.

We will cover the following aspects of formal category theory inside a general $\infty$-equipment $\P$ in this paper:
\begin{itemize}
	\item \cref{ssection.injections}: \textit{injections}, i.e.\ abstract fully faithful arrows in $\P$,
	\item \cref{ssection.fct-adjunctions}: adjunctions in the vertical fragment of $\P$ and additional properties,
	\item \cref{ssection.weighted-colims}: \textit{weighted} colimits and limits,
	\item \cref{ssection.pke}: point-wise Kan extensions,
	\item \cref{ssection.exact-squares}: Beck--Chevalley conditions for point-wise Kan extensions,
	\item \cref{ssection.initial-final}: abstract initial and final vertical arrows.
\end{itemize}
When defining and developing each of these aspects, we will provide additional examples in the archetypical case that $\P = \infty\CCAT$, and show that it recovers the usual notions of $\infty$-category theory.
In the sequel to this paper \cite{EquipII} we will demonstrate that for $\P = \CCAT(\E)$, where $\E$ is an $\infty$-topos, the formal category theory in $\P$ recovers the theory 
that was developed by Shah \cite{Shah}, Martini \cite{Martini}, and Martini--Wolf \cite{MartiniWolf}.

\subsection*{Conventions} We will use the language of $\infty$-categories as developed by Joyal and Lurie. Moreover, we  use the following notational conventions:
\begin{itemize}
	\item The  $\infty$-categories of spaces (i.e.\ $\infty$-groupoids) and $\infty$-categories are denoted by $\S$ and $\infty\Cat$ respectively. 
	\item We will usually be imprecise about sizes, apart from the cases where it is important to keep track. On these occasions, we will write 
	$\widehat{\S}$ and $\widehat{\infty\Cat}$ for the (very large) $\infty$-categories of large spaces and large $\infty$-categories.
	\item If $\C$ is an $\infty$-category, we will write $\PSh(\C) := \fun(\C^\op, \S)$ for the $\infty$-cate\-go\-ry of presheaves on $\C$.
	\item We will use uppercase notation for $(\infty,2)$-categorical upgrades of particular $\infty$-categories. Double $\infty$-categorical variants will be decorated with blackboard bolds.
	For instance, in this article, we will encounter both the $(\infty,2)$-category $\infty\CAT$ and the double $\infty$-category $\infty\CCAT$ of $\infty$-categories.
\end{itemize}

\subsection*{Acknowledgements}

The theory in this article formed a part of my Ph.D. thesis, and  I would like to express my gratitude towards my advisor, Lennart Meier, for the helpful comments and discussions on the contents of this paper.
During a visit to Johns Hopkins University in the spring of 2023, I had many fruitful conversations with Emily Riehl 
regarding the research presented in this article. I would like to thank her once again for these discussions and the new insights they provided.
The author was funded by the Dutch Research Council (NWO) through the grant ``The interplay of orientations and symmetry'', grant no. OCENW.KLEIN.364.

\section{Two-dimensional \texorpdfstring{$\infty$}{∞}-categories}\label{section.prelims}

We will start by covering the preliminaries on two-dimensional $\infty$-categories. 
\subsection{Segal objects}

We recall the notion of \textit{Segal objects} \cite[Definition 1.1.1]{LurieInfty2} in general $\infty$-categories. 

\begin{definition}
	Let $\E$ be an $\infty$-category with pullbacks. 
	A \textit{Segal object $X$ in $\E$} is a simplicial object $X : \Delta^\op \rightarrow \E$
	such that the map
	$$ 
	X([n]) \rightarrow X(\{0 \leq 1\}) \times_{X(\{1\})} \dotsb \times_{X(\{n-1\})} X(\{n-1\leq n\})
	$$ 
	is an equivalence for all $n \geq 0$. We will write $$\Seg(\E) \subset \fun(\Delta^\op, \E)$$ for the full subcategory spanned by the Segal objects. For brevity, we will write 
	$\Seg := \Seg(\S)$.
\end{definition}

\begin{definition}\label{def.comp-cat-obj}
	Let $X$ be a Segal object in an $\infty$-category $\E$ with pullbacks. Then we define $X_\mathrm{eq} \in \E$ by the pullback square
	\[
	\begin{tikzcd}
		X_\mathrm{eq} \arrow[r]\arrow[d] & X([3]) \arrow[d,] \\
		X([0])\times X([0]) \arrow[r] & X(\{0\leq 2\}) \times X(\{1 \leq 3\}).
	\end{tikzcd}
	\]
	We say that $X$ is \textit{complete} if the canonical map $X([0]) \rightarrow X_\mathrm{eq}$ is an equivalence.  
	In this case, $X$ is also called a \textit{category internal to $\E$}. We will write $$\Cat(\E) \subset \Seg(\E)$$ for the full subcategory spanned by the complete Segal objects. 
\end{definition}

\begin{example}\label{ex.segal-spaces}
	Rezk's Segal spaces \cite{RezkSeg} are precisely the Segal objects in the $\infty$-category $\S$ of spaces. 
	Let $J$ be the simplicial set defined by the pushout square 
\[
	\begin{tikzcd}[column sep = large]
	{[1]^{\sqcup 2}} \arrow[r, "\{0\leq 2\} \sqcup \{1 \leq 3\}"]\arrow[d] & {[3]} \arrow[d] \\
	{[0]^{\sqcup 2}} \arrow[r] & J
	\end{tikzcd}
\] 
of simplicial spaces. Here, we implicitly view every $[n] \in \Delta$ as an object of $\PSh(\Delta)$ via the Yoneda embedding. A Segal space $X$ is complete in the sense of \cref{def.comp-cat-obj} if and only if the map 
$$
\map_{\PSh(\Delta)}([0], X) \rightarrow \map_{\PSh(\Delta)}(J,X)
$$
induced by $J \rightarrow [0]$ is an equivalence. This coincides with Rezk's notion of completeness for Segal spaces on account of \cite[Proposition 11.1]{RezkSeg}.

It was shown by Joyal and Tierney \cite{JoyalTierney} that the inclusion $\Delta \to \infty\Cat$ is dense. 
In other words, the restricted Yoneda embedding
$$
	\infty\Cat \rightarrow \PSh(\Delta) : \C \mapsto ([n] \mapsto \map_{\infty\Cat}([n], \C))
	$$
	is fully faithful. The essential image of this functor is precisely given by $\Cat(\S)$.
\end{example}

There are $n$-dimensional variants on Segal spaces \cite{BarwickPhD} \cite{RezkTheta}. We will discuss 
Rezk's model for $n=2$. To this end, we make use of the following shapes:

\begin{definition}
	Let $n$ be a non-negative integer, and let $\overline{m} = (m_1, \dotsc, m_{n})$ be a tuple of non-negative integers. 
	Then we write $[n;\overline{m}]$ for the 2-category with:
	\begin{itemize}
		\item objects given by the integers $0, \dotsc, n$,
		\item hom-categories given by  \[
		[n;\overline{m}](i,j) := \begin{cases}
			[m_{i+1}] \times \dotsb \times [m_{j}] & \text{if $i \leq j$}, \\
			\emptyset & \text{otherwise},
		\end{cases}
		\]
		for $0\leq i,j \leq n$,
	\end{itemize}
	with the obvious composition.
	The 2-category $[n;\overline{m}]$ is called a \textit{globular sum}, and can be pictured as 
	\[
		{[n;\overline{m}]} := \begin{tikzcd}[column sep = large]
			0 \arrow[r, "0"name=a0, looseness = 2, bend left = 70pt ]\arrow[r, "1"'name=a1, bend left = 45pt]\arrow[r, "m_1-1"name=am1, bend right = 45pt]\arrow[r, "m_1"'name=am, looseness = 2, bend right = 70pt] & 1	
			\arrow[from=a0,to=a1, Rightarrow, shorten <= 3pt, shorten >= 3pt]
			\arrow[from=a1,to=am1, phantom, "\vdots", yshift = 3pt]
			\arrow[from=am1,to=am, Rightarrow, shorten <= 3pt, shorten >= 3pt] \arrow[r, phantom, "\cdots"] &  
			n-1\arrow[r, "0"name=b0, looseness = 2, bend left = 70pt ]\arrow[r, "1"'name=b1, bend left = 45pt]\arrow[r, "m_n-1"name=bm1, bend right = 45pt]\arrow[r, "m_n"'name=bm, looseness = 2, bend right = 70pt] & n.	
			\arrow[from=b0,to=b1, Rightarrow, shorten <= 3pt, shorten >= 3pt]
			\arrow[from=b1,to=bm1, phantom, "\vdots", yshift = 3pt]
			\arrow[from=bm1,to=bm, Rightarrow, shorten <= 3pt, shorten >= 3pt]
		\end{tikzcd}
	\] We write $\Theta_2$ for the full subcategory of the category 
	of strict 2-categories spanned by these globular sums. 
	Note that there is a canonical inclusion $\Delta \to \Theta_2$.
\end{definition}

\begin{definition}
	A presheaf $X \in \PSh(\Theta_2)$ is called a \textit{$\Theta_2$-Segal space} if the following conditions are met: 
	\begin{enumerate}
		\item for any tuple $\overline{m} = (m_1, \dotsc, m_{n})$, the canonical map 
		$$
		X([n;\overline{m}]) \to X([1;m_1]) \times_{X([0])} \dotsb \times_{X([0])} X([1;m_n]),
		$$
		\item for any $n \geq 0$, the canonical map 
		$$
		X([1;n]) \to X([1;1]) \times_{X([1])} \dotsb \times_{X([1])} X([1;1])
		$$
		is an equivalence.
	\end{enumerate}
	We will write 
	$$
	\Theta_2\Seg \subset \PSh(\Theta_2)
	$$
	for the full subcategory spanned by the $\Theta_2$-Segal spaces.
\end{definition}

\begin{construction}
	The inclusion $\Delta \subset \Theta_2$ admits a left adjoint that carries a globular sum $[n;m]$ to $[n]$. 
	This induces an adjunction $\PSh(\Delta) \rightleftarrows \PSh(\Theta_2)$. In turn, one may readily verify that this restricts to an adjunction
	$$
		i : \Seg \rightleftarrows \Theta_2\Seg : \tau_1.
	$$ 
	The right adjoint is called the \textit{1-core} functor, and it forgets the 2-cells.
\end{construction}

\begin{definition}
	As customary, if $x$ and $y$ are objects of a $\Theta_2$-Segal space $X$, then we define the Segal space $X(x,y)$ 
	of maps from $x$ to $y$ by the pullback square
	\[
		\begin{tikzcd}
			X(x,y)\arrow[d] \arrow[r] & X([1; -]) \arrow[d] \\
			\{(x,y)\} \arrow[r] & X([0]) \times X([0]).
		\end{tikzcd}
	\]
\end{definition}

\begin{remark}
	Suppose that $f : x \to y$ is an arrow of a $\Theta_2$-Segal space $X$. Then $f$ induces a post-composition map 
	$f \circ (-): X(z,x) \to X(z,y)$ for every $z \in X$ that is induced by the following commutative square 
	of Segal spaces
	\[
		\begin{tikzcd}
			X([2;-]) \arrow[r, "{[\{0\leq 2\}; -]^*}"]\arrow[d, "{\{0\}^* \times [\{1\leq 2\}; -]^*}"'] & X([1;-]) \arrow[d] \\
			X([0]) \times X([1;-])\arrow[r] & X([0]) \times X([0])
		\end{tikzcd}
	\]
	after taking fibers above $(z,f)$.
	Similarly, we obtain a pre-composition map $(-) \circ f : X(y,z) \to X(x,z)$,
\end{remark}

\begin{definition}
	Let $X$ be a $\Theta_2$-Segal space. Then $X$ is called \textit{locally complete} if $X(x,y)$ is an $\infty$-category for 
	every $x,y \in X$. If, additionally, $\tau_1 X$ is an $\infty$-category, then we say that $X$ is an \textit{$(\infty,2)$-category}.
	We write
	$$(\infty,2)\Cat \subset \Theta_2\Seg$$ 
	for the full subcategory spanned by the $(\infty,2)$-categories.
\end{definition}

\begin{remark}
	The adjunction $\Seg \rightleftarrows \Theta_2\Seg : \tau_1$ restricts to an adjunction 
	$\infty\Cat \rightleftarrows (\infty,2)\Cat: \tau_1$.
\end{remark}

\begin{construction}
	There are involutions $\opo : \Theta_2 \to \Theta_2$ and $\opt : \Theta_2 \to \Theta_2$ that reverse 
	the directions of 1-cells and 2-cells, respectively. These induce involutions 
	$$
	(-)^{\opo} := \opo^* : \PSh(\Delta^{\times 2}) \rightarrow \PSh(\Delta^{\times 2}), \quad (-)^\opt := \opt^* : \PSh(\Delta^{\times 2}) \rightarrow \PSh(\Delta^{\times 2}),
	$$
	called the \textit{1-opposite} and \textit{2-opposite} respectively, and both restrict to involutions on 
	$\Theta_2\Seg$ and $(\infty,2)\Cat$.
\end{construction}

\subsection{Double $\infty$-categories}\label{ssection.dbl-cats} We are now ready to discuss the categorical structures that 
play a key role in this article.

\begin{definition}	
	A  \textit{double $\infty$-category} is a Segal object in $\infty\Cat$. We write $$\DbliCat := \Seg(\infty\Cat)$$ for the 
	ambient $\infty$-category of double $\infty$-categories.
\end{definition} 

It is often convenient and illuminating to view double $\infty$-categories as particular \textit{double Segal spaces}:
\begin{definition}
	A \textit{double Segal space} is a  Segal object in $\Seg$.
		We will write
		$$\DblSeg := \Seg(\Seg)$$ for the $\infty$-category of double Segal spaces.
\end{definition}

\begin{remark}\label{con.dbl-infty-vs-dbl-segal}
	We will view $\DblSeg \subset \fun(\Delta^\op, \Seg) \subset \fun(\Delta^\op, \PSh(\Delta))$ as a full subcategory of  $\PSh(\Delta^{\times 2}).$ 
	The Yoneda embedding $\Delta^{\times 2} \to \PSh(\Delta^{\times 2}) : ([n], [m]) \mapsto [n,m]$ factors through $\DblSeg$.
	As explained in \cite[Remark 3.35]{Comp}, we may view double $\infty$-categories as double Segal spaces. There is an inclusion
	$$
	i :\DbliCat \rightarrow \DblSeg
	$$
	so that a double $\infty$-category $\P$ to the double Segal space $i(\P)$ described by
	$$
	i(\P)_{n,m} = \map_{\DblSeg}([n,m], i(\P)) \simeq \map_{\infty\Cat}([m], \P_n').
	$$
	The essential image of the functor $i$ consists of those 
	double Segal spaces $\P$ such that $\P_{n,-}$ is a complete Segal space for all $n$. We will leave the inclusion 
	$i$ implicit in the rest of this article. We note that $[n,m]$ is always a double $\infty$-category.
\end{remark}

	The double $\infty$-categories $[0,0]$, $[1,0]$, $[0,1]$ and $[1,1]$ are called the \textit{free-living object/0-cell}, the \textit{free-living horizontal} and \textit{vertical arrow/1-cell}, and the \textit{free-living 2-cell}, 
	respectively.

\begin{remark}
Unwinding the definitions, a double Segal space $\P$ contains the following data: 
\begin{itemize} 
\item a space $\P_{0,0}$ of objects,
\item a space $\P_{0,1}$ of {vertical arrows},
\item a space $\P_{1,0}$ of {horizontal arrows},
\item a space $\P_{1,1}$ of 
2-cells,
\item a space $\P_{n,m}$ of $(n \times m)$-grids of 2-cells for $n,m \geq 1$.
\end{itemize}
As customary, we will picture 2-cells as squares
\[
	\begin{tikzcd}
		a \arrow[r, ""'name=f]\arrow[d] & b \arrow[d] \\
		c \arrow[r, ""name=t] & d
		\arrow[from=f, to=t, Rightarrow]
	\end{tikzcd}
\]
where the horizontal/vertical directed arrows are horizontal/vertical arrows of $\P$. 
\end{remark}

\begin{example}\label{example.dbl-cat-morita}
	Suppose that $\C$ is a monoidal $\infty$-category so that $\C$ admits all geometric realizations, 
	and the tensor product $\otimes$ of $\C$ preserves geometric realizations in each variable.
 	Lurie constructs in \cite[Subsection 4.4.3]{HA} an associated double $\infty$-category
	$$
	\mathbb{M}\mathrm{or}(\C),
	$$
	called the \textit{Morita double $\infty$-category}; see \cite[Remark 4.4.3.11]{HA}. Its objects are given by algebras in $\C$, the vertical arrows 
	are algebra morphisms, and a horizontal arrow $M:A \rightarrow B$ between algebras is given by an $(A,B)$-bimodule $M$. 
	A 2-cell 
	\[
	\begin{tikzcd}
		A \arrow[r ,"M"name=f] \arrow[d,"f"'] & B \arrow[d, "g"] \\
		C \arrow[r,"N"name=t] & D
		\arrow[from=f, to=t, Rightarrow, shorten <= 6pt]
	\end{tikzcd}
	\]
	in $\mathbb{M}\mathrm{or}(\C)$  
	corresponds to a map of $(C,D)$-bimodules $C \otimes_A M \otimes_B D \rightarrow N$.
\end{example}

\begin{example}\label{example.dbl-cat-spans}
	Let $\C$ be an $\infty$-category with finite limits. In \cite{HaugsengSpans}, Haug\-seng constructs the double $\infty$-category 
	$$
	\SSPAN(\C)
	$$
	of spans in $\C$. The objects of this double $\infty$-category are given by the objects of $\C$, and the vertical arrows are the arrows of $\C$. 
	A horizontal arrow $a \rightarrow b$ in this double $\infty$-category corresponds to a span
	$$
	a \xleftarrow{p} x \xrightarrow{q} b
	$$
	in $\C$, i.e.\ a map $(p,q) : x \rightarrow a\times b$.
	The 2-cells in $\SSPAN(\C)$ pictured on the left below
	\[
	\begin{tikzcd}
		a \arrow[r, "{(p,q)}"name=f] \arrow[d,"f"'] & b \arrow[d,"g"] \\
		c \arrow[r,"{(p',q')}"name=t] & d
		\arrow[from=f, to=t, Rightarrow, shorten <= 6pt]
	\end{tikzcd}
	\quad 
	\Leftrightarrow
	\quad
	\begin{tikzcd}
		a\arrow[d,"f"'] & \arrow[l,"p"']  x\arrow[d] \arrow[r,"q"] & b\arrow[d,"g"] \\
		c & \arrow[l,"p'"']  y \arrow[r,"q'"] & d
	\end{tikzcd}
	\]
	correspond precisely to commutative diagrams in $\C$ as displayed on the right.
\end{example}

\begin{example}\label{ex.squares}
	Let $\C$ be an $(\infty,2)$-category. Then there is a double $\infty$-category
	$$
	\Sq(\C)
	$$
	of squares in $\C$. The objects of this double $\infty$-category are given by the objects of $\C$, 
	and the vertical and horizontal arrows are given by the arrows of $\C$.
	The 2-cells in $\Sq(\C)$ pictured on the left below
	\[
	\begin{tikzcd}
		a \arrow[r, "h"name=f] \arrow[d,"f"'] & b \arrow[d,"g"] \\
		c \arrow[r,"k"name=t] & d
		\arrow[from=f, to=t, Rightarrow, shorten <= 6pt]
	\end{tikzcd}
	\quad 
	\Leftrightarrow
	\quad
	\begin{tikzcd}
		a \arrow[r, "h"] \arrow[d,"f"'] & |[alias=f]| b \arrow[d,"g"] \\
		|[alias=t]| c \arrow[r,"k"] & d
		\arrow[from=f, to=t, Rightarrow]
	\end{tikzcd}
	\]
	correspond precisely to the lax commutative diagrams in $\C$ as displayed on the right.
	A construction of $\Sq(\C)$ can be found in \cite[Section 3.6]{Comp}. Precisely, $\Sq(\C)$ is defined by
	$$
		\map_{\DbliCat}([n,m], \Sq(\C)) \simeq \map_{(\infty,2)\Cat}([n] \otimes [m], \C),
	$$
	where $[n] \otimes [m]$ is the (nerve) of the 2-category given by the oplax Gray tensor product
	of $[n]$ and $[m]$, as defined in \cite{Gray}.
\end{example}

\begin{construction}\label{cons.dualities} 
	We recall the following duality operations for double Segal spaces.
	The involution $\op : \Delta \rightarrow \Delta$ gives rise to two functors 
	$$
	\hop : \Delta \times \Delta \xrightarrow{\op \times \id} \Delta \times \Delta, \quad \vop : \Delta \times \Delta \xrightarrow{\id \times \op} \Delta \times \Delta,
	$$
	so that we obtain involutions
	$$
	(-)^\hop := \hop^* : \PSh(\Delta^{\times 2}) \rightarrow \PSh(\Delta^{\times 2}), \quad (-)^\vop := \vop^* : \PSh(\Delta^{\times 2}) \rightarrow \PSh(\Delta^{\times 2}),
	$$
	called the \textit{horizontal opposite} and \textit{vertical opposite} functors respectively.
	One can readily verify that these functors restrict to endofunctors on the full subcategories $\DblSeg$ and $\DbliCat$.
	There is yet one other useful involution. We will write 
	$$
	(-)^\tp = \tp^* : \PSh(\Delta^{\times 2}) \rightarrow \PSh(\Delta^{\times 2})
	$$ 
	for the functor given by restriction along $\tp : \Delta^{\times 2} \rightarrow \Delta^{\times 2} : ([n],[m]) \mapsto ([m],[n])$. 
	This is called the \textit{transpose} functor. One readily verifies that $(-)^\tp$ restricts to an endofunctor on $\DblSeg$. However, the transpose 
	of a double $\infty$-category is not necessarily a double $\infty$-category again.
\end{construction}

\begin{notation} 
	We recall the following notation from \cite[Section 3.3]{Comp}. We will write 
$$(-)_h  : \Theta_2\Seg \to \DblSeg$$
	for the \textit{horizontal inclusion} functor that is induced by the  bicosimplicial object $\Delta^{\times 2} \to  \Theta_2 : ([n], [m]) \mapsto [n;m,\dotsc, m]$.
	There is also a \textit{vertical inclusion} functor 
defined by the composite 
$$
	(-)_v : \Theta_2\Seg  \xrightarrow{(-)_h} \DblSeg \xrightarrow{(-)^\tp} \DblSeg \xrightarrow{(-)^\hop} \DblSeg.
	$$
	The functors $(-)_h$ and $(-)_v$ admit right adjoints that are denoted by
	$$
	\Hor(-), \Vert(-) : \DblSeg \rightarrow \Theta_2\Seg,
	$$
	respectively. These are called the \textit{horizontal fragment} and \textit{vertical fragment} functors.
\end{notation}

\begin{remark}\label{rem.h-v-inclusions}
	As explained in \cite[Section 3.2]{Comp}, there are pushout squares 
	\[
	\begin{tikzcd}
		\{0,\dotsc, n\}_h \times [m]_v \arrow[r]\arrow[d] & {[n,m]} \arrow[d] \\
		\{0, \dotsc, n\}_h  \arrow[r] & {[n;m,\dotsc,m]_h,}
	\end{tikzcd}
	\quad
	\begin{tikzcd}
		{ [m]_h^\hop} \times \{0,\dotsc, n\}_v \arrow[r]\arrow[d] & {[m,n]^\hop} \arrow[d] \\
		 \{0,\dotsc, n\}_v \arrow[r] & {[n;m,\dotsc,m]_v}
	\end{tikzcd}
	\]
	in $\DblSeg$ that are natural in $[n], [m] \in \Delta$. In particular, the horizontal and vertical inclusions of the free 2-cell 
	may be pictured as 
	\[
		[1;1]_h = \begin{tikzcd}
			0 \arrow[d,equal]\arrow[r, "0"name=f] & 1 \arrow[d,equal] \\
			0 \arrow[r, "1"name=t] & 1
			\arrow[from=f,to=t, Rightarrow, shorten <= 6pt]
		\end{tikzcd}
		\quad \text{and}\quad
		[1;1]_v = \begin{tikzcd}
			0 \arrow[r,equal, ""name=f]\arrow[d, "1"'] & 0 \arrow[d,"0"] \\
			0 \arrow[r, equal, ""name=t] & 1.
			\arrow[from=f,to=t,Rightarrow, shorten <= 6pt]
		\end{tikzcd}
	\]
\end{remark}

\begin{example}\label{ex.hor-vert-sq}
	Let $\C$ be an $(\infty,2)$-category. Then the vertical and horizontal fragments of the double $\infty$-category $\Sq(\C)$ 
	of squares (see \cref{ex.squares})
	are both given by $\C$. Precisely, in \cite[Construction 3.57]{Comp} we constructed a natural cospan 
	of functors $\C_h \to \Sq(\C) \leftarrow \C_v$ that is adjunct to a span of equivalences 
	$$\Hor(\Sq(\C)) \xleftarrow{\simeq} \C \xrightarrow{\simeq} \Vert(\Sq(\C)).$$
\end{example}

\begin{remark}\label{rem.vert-hor-dbl-infty-cat}
	Let $\P$ be a double $\infty$-category.
	One readily verifies that there is a canonical equivalence $\tau_1\Vert(\P) \simeq \P_0$. For every two objects $x,y \in \P$, $\Hor(\P)(x,y)$ is always an $\infty$-category 
	and it fits in a pullback square 
	\[
		\begin{tikzcd}
			\Hor(\P)(x,y) \arrow[r]\arrow[d] & \P_1 \arrow[d] \\
			\{(x,y)\} \arrow[r] & \P_0 \times \P_0.
		\end{tikzcd}
	\]
\end{remark}

It will be important later in the paper to consider double $\infty$-categories with an extra completeness condition.
We recall the following:

\begin{proposition}[{\cite[Proposition 3.39]{Comp}}]
	Let $\P$ be  a double Segal space. Then the following assertions are equivalent:
	\begin{enumerate}
		\item $\Vert(\P)$ is an $(\infty,2)$-category and $\Hor(\P)$ is locally complete,
		\item $\P$ is a double $\infty$-category and the canonical functor $\P_0 \to \P_\mathrm{eq}$ of \cref{def.comp-cat-obj} is fully faithful.
	\end{enumerate}
\end{proposition}

\begin{corollary}\label{cor.loc-comp}
	Let $\P$ be a double $\infty$-category. Then the following assertions are equivalent:
	\begin{enumerate}
		\item $\Vert(\P)$ is an $(\infty,2)$-category,
		\item $\Vert(\P)$ is locally complete,
		\item the canonical functor $\P_0 \to \P_\mathrm{eq}$ is fully faithful.
	\end{enumerate}
\end{corollary}
\begin{proof}
This follows from the above and \cref{rem.vert-hor-dbl-infty-cat}.
\end{proof}

\begin{definition}
	A double $\infty$-category is called \textit{locally complete}
	if it meets the equivalent conditions of \cref{cor.loc-comp}.
\end{definition}

\subsection{Adjunctions} We briefly discuss adjunctions in locally complete $\Theta_2$-Segal spaces, and collect some basic facts about these. 

\begin{definition}\label{def.adjunction}
	Suppose that $u : x \rightarrow y$ and $v: y \rightarrow x$ are arrows in a locally complete $\Theta_2$-Segal space $X$. Then $(u,v)$ is called an \textit{adjoint pair} if there 
	exist 2-cells 
	$$
	\eta : \id_x \rightarrow vu, \quad \epsilon : uv \rightarrow \id_y
	$$
	in $X(x,x)$ and $X(y,y)$ respectively, so that the two triangle identities hold:
	\begin{enumerate}
		\item the composite of the whiskered 2-cells
		$$
		u \xrightarrow{u\eta} uvu \xrightarrow{\epsilon u} u
		$$
		is equivalent to the identity 2-cell $\id_u$ in $X(x,y)$,
		\item the composite of the whiskered 2-cells
		$$
		v \xrightarrow{\eta v} vuv \xrightarrow{v\epsilon} v
		$$
		is equivalent to the identity 2-cell $\id_v$ in $X(y,x)$.
	\end{enumerate}
	In this case, $\eta$ and $\epsilon$ are called the \textit{adjunction unit} and \textit{counit} respectively.
\end{definition}

\begin{proposition}\label{prop.uniqueness-adjoints}
	Suppose that we have an adjunction
	\[
		u : x \rightleftarrows y : v
	\]
	in a locally complete $\Theta_2$-Segal space $X$ with unit $\eta$ and counit $\epsilon$.  For an arrow $v' \in X(y,x)$, the following assertions 
	are equivalent:
	\begin{enumerate}
		\item the pair $(u,v')$ is an adjoint pair,
		\item there exists an equivalence  $\alpha : v \rightarrow v'$ in $X(y,x)$.
	\end{enumerate} 
	Suppose that (1) holds, and write $\eta'$ and counit $\epsilon'$ for the adjunction $(u,v')$. Then $\alpha$ may be 
	constructed in such a way that
	$\eta'$ can be recovered as the composite 
	$$
	\id_x \rightarrow vu \xrightarrow{\alpha u} v'u,
	$$
	and $\epsilon$ can be recovered as the composite 
	$$
	uv \xrightarrow{u\alpha} uv' \xrightarrow{\epsilon'} \id_y.
	$$
	Conversely, if (2) holds, then $\eta'$ and $\epsilon'$ can be constructed so that this is true.
\end{proposition}
\begin{proof}
	This proof is completely similar as the proof for 2-categories; see \cite[Proposition 2.1.10]{RiehlVerity} for instance.  
\end{proof}

We will show that adjunctions can be lifted along \textit{locally fully faithful functors}:

\begin{definition}
	A map $f : X \rightarrow Y$ between locally complete $\Theta_2$-Segal spaces is called \textit{locally fully faithful} when the
	induced functor $f_{x,y} : X(x,y) \rightarrow Y(f(x), f(y))$ between $\infty$-categories is fully faithful for all $x,y \in X$.
\end{definition}

\begin{proposition}\label{prop.lff-adjoints}
	Suppose that $f : X \rightarrow Y$ is a locally fully faithful functor between locally complete $\Theta_2$-Segal spaces. 
	Let $u : x\rightarrow y$ and $v : y\rightarrow x$ be arrows in $X$. Then $(u,v)$ is an adjoint pair in $X$ 
	if and only if $(f(u), f(v))$ is an adjoint pair in $Y$.
\end{proposition}
\begin{proof}
	It is clear that functors between locally complete $\Theta_2$-Segal spaces preserve adjoint pairs. Hence, we must show that 
	$(u,v)$ is an adjoint pair if $(f(u), f(v))$ is an adjoint pair. Suppose that $\eta'$ and $\epsilon'$ are the unit and counit 
	for the adjunction $(f(u),f(v))$. Since the functor $f_{x,x} : X(x,x) \rightarrow Y(f(x), f(x))$ is fully faithful, we find 
	$\eta : \id_x \rightarrow vu$ in $X(x,x)$ such that $f(\eta) \simeq \eta'$. Similarly, we find such a lift $\epsilon \in X(y,y)$ of the counit $\eta'$. 
	Now, we have to verify that the triangle identities hold. We will just show (1) of \cref{def.adjunction}; identity (2) is handled similarly.
	Note that identity (1) for the adjunction $(f(u), f(v))$ witnesses the commutativity of the square 
	\[
		\begin{tikzcd}
			{[1]} \arrow[r]\arrow[d] & X(x,y) \arrow[d, "{f_{x,y}}"] \\
			{[0]} \arrow[r] & Y(f(x),f(y)),
		\end{tikzcd}
	\]
	where the top arrow selects the whiskered composite $u \xrightarrow{u\eta} uvu \xrightarrow{\epsilon u} u$. Since the right arrow in this square is fully faithful, it is left orthogonal 
	to the arrow on the left. Thus a (unique) diagonal lift in this square exists, which exhibits identity (1). 
\end{proof}

\begin{proposition}\label{prop.adjoints-mapping-cats}
Suppose that 
$u : x \rightleftarrows y : v$
is an adjunction in a locally complete $\Theta_2$-Segal space $X$. Then for every $z \in X$, 
we obtain induced adjunctions
\[
		u \circ (-) :X(z,x) \rightleftarrows X(z,y) : v \circ (-)
\]
\[
		(-) \circ v : X(y,z) \rightleftarrows X(x,z): (-) \circ u
\]
of $\infty$-categories.
\end{proposition}
\begin{proof}
We will just show the first adjunction; the other case is dual.
Let $\eta$ and $\epsilon$ be the unit and counit for the adjunction $(u,v)$ respectively. Then we obtain induced natural transformations
$$
\eta \circ (-) : [1] \times X(z,x) \xrightarrow{\eta \times \id} X(x,x) \times X(z,x) \xrightarrow{(-)\circ(-)} X(z,x)
$$
and 
$$
\epsilon \circ (-) : [1] \times X(z,y) \xrightarrow{\epsilon \times \id} X(y,y) \times X(z,y)  \xrightarrow{(-)\circ(-)} X(z,y).
$$
These are the candidate unit and counit of the desired adjunction. The triangle identities 
for $\eta$ and $\epsilon$ readily witness that the triangle identities hold for $\eta \circ (-)$ and $\epsilon \circ (-)$ as well.
\end{proof}

The following notion in the context of $\Theta_2$-Segal spaces (cf.\ \cite[Definition 5.8]{ShulmanFramedBicats}) will be used:

\begin{definition}\label{def.hor-closed}
	A locally complete $\Theta_2$-Segal space $X$ is said to be \textit{closed} if for every arrow $f : x\rightarrow y$ in $X$ and $z \in X$,
	the induced functors $$f \circ (-) : X(z,x) \rightarrow X(z,y) \quad \text { and } \quad (-) \circ f : X(y,z) \rightarrow X(x,z)$$ admit right adjoints.
	A double $\infty$-category $\P$ is said to be \textit{horizontally closed}
	if the horizontal fragment $\Hor(\P)$ is closed.
\end{definition}

We refer to \cite[Remark 5.10]{ShulmanFramedBicats} for a motivation of this terminology.

\section{Example: the double \texorpdfstring{$\infty$}{∞}-category of \texorpdfstring{$\infty$}{∞}-categories}\label{section.equipment-infty-cats}

We will now discuss the prototypical example of a double $\infty$-category. We recall the construction of 
 the double $\infty$-category 
$
\infty\CCAT  
$
with:
\begin{itemize} 
	\item the objects given by (small) $\infty$-categories,
	\item the vertical arrows given by functors,
	\item the horizontal arrows given by \textit{profunctors}, or equivalently by \textit{two-sided discrete fibrations},\footnote{These types of fibrations were named \textit{bifibrations} in \cite{HTT}. We will instead 
	use the terminology that is used in classical category theory for the same kind of fibrations.} or by \textit{correspondences},
\end{itemize}
that is due to Ayala and Francis \cite{AyalaFrancis}.

\subsection{The different perspectives on profunctors}\label{ssection.profunctors}

As a preparation, we will give a brief recollection of the equivalent notions that appear in the above 
description of the horizontal arrows of the (to be constructed) double $\infty$-category $\infty\CCAT$.

\begin{definition}
	Let $\C$ and $\D$ be two $\infty$-categories. Then a \textit{correspondence} 
	from $\C$ to $\D$ is a functor 
	$
	E \rightarrow [1]
	$
	together with equivalences $E_1 \simeq \C$ and $E_0 \simeq \D$. 
\end{definition}

\begin{definition}
	Let $\C$ and $\D$ be two $\infty$-categories.  Then a \textit{two-sided discrete fibration} from $\C$ to $\D$
	is a span 
	$
		(p,q) : E \rightarrow \C \times \D
		$
	with the following properties:
	\begin{itemize}
		\item $p$ is a cocartesian fibration so that its cocartesian arrows lie 
		over equivalences in $\D$,
		\item $q$ is a cartesian fibration so that its cartesian arrows lie over equivalences in $\C$,
		\item for every $c\in \C$ and $d \in \D$, the fiber $E_{c,d}$ is a space.
	\end{itemize}
	We will write
	$$
\DFib \subset \fun([1] \cup_{\{0\}} [1], \infty\Cat)
$$
	for the full subcategory of spans in $\infty\Cat$ given by the two-sided discrete fibrations. It comes 
	with a canonical projection to $\infty\Cat^{\times 2}$.
\end{definition}

\begin{remark}
	The above definition is equivalent to \cite[Definition 2.4.7.2]{HTT} and the description in \cite[Subsection 4.1]{AyalaFrancis} on account of \cite[Proposition 2.3.13]{HHLN}.
\end{remark}

\begin{definition}
	Let $\C$ and $\D$ be two (small) $\infty$-categories. Then a \textit{profunctor} from $\C$ to $\D$ is a functor 
	$
   F : \D^\op\times \C \rightarrow \S.
   $
   The $\infty$-category of profunctors is given by unstraightening the functor
$
\infty\Cat^{\op, \times 2} \rightarrow \widehat{\infty\Cat} : (\C, \D) \mapsto \fun(\D^\op \times \C, \S)
$
to a (large) cartesian fibration 
$$
\prof \rightarrow \infty\Cat^{\times 2}.
$$
\end{definition}

Profunctors, two-sided discrete fibrations, and correspondences encode the same data, as shown by Ayala and Francis:

\begin{proposition}[{\cite[Subsection 4.1]{AyalaFrancis}}]\label{prop.equiv-cor-bifib-prof}
	There are equivalences that fit in a commutative diagram
	\[ 
	\begin{tikzcd}[column sep = large]
		\DFib \arrow[r, "\simeq"]\arrow[dr]& \infty\Cat_{/[1]} \arrow[r,"\simeq"]\arrow[d, "{(\{1\}^*, \{0\}^*)}"'] & \prof \arrow[dl] \\
		& \infty\Cat^{\times 2}
	\end{tikzcd}
	\]
	of $\infty$-categories. 
	The top left functor carries a two-sided discrete fibration $E \rightarrow \C\times \D$ 
	to the cocomma correspondence $$\C  \cup_{E \times \{1\}} E \times [1] \cup_{E \times \{0\}} \D \rightarrow [1].$$
	In turn, the top right functor carries a correspondence $E \rightarrow [1]$ to the profunctor 
	$$
	E_0^\op \times E_1 \rightarrow E^\op \times E \xrightarrow{\map_E} \S. 
	$$
\end{proposition}

\begin{remark}
The equivalence between $\DFib$ and $\infty\Cat_{/[1]}$
was first exhibited by Stevenson \cite{Stevenson} using 
the language of model categories.
\end{remark}

\subsection{The construction via Conduch\'e fibrations} 

The notion 
of \textit{Conduch\'e fibrations} is central to the construction of the double $\infty$-category $\infty\CCAT$ of $\infty$-categories by Ayala and Francis.

\begin{definition}
	A functor $p:E\rightarrow \C$ between $\infty$-categories is called a \textit{Conduch\'e fibration} if the pullback functor 
	$
	p^* : \infty\Cat_{/\C} \rightarrow \infty\Cat_{/E}
	$
	admits a right adjoint. We will write $$\Con(\C) \subset \infty\Cat_{/\C}$$ for the
	 full subcategory of $\infty\Cat_{/\C}$ spanned by 
	the Conduch\'e fibrations.
\end{definition}

\begin{remark}
Conduch\'e fibrations are precisely the \textit{exponentiable} morphisms in $\infty\Cat$, see e.g.\ \cite[Section 2]{Comp} for a review of this general notion.
Conduch\'e fibrations are also called \textit{exponentiable fibrations} \cite{AyalaFrancisRozenblyum} \cite{AyalaFrancis}, or \textit{flat inner fibrations} \cite[Appendix B.3]{HA}.
\end{remark}

In the double $\infty$-category $\infty\CCAT$ that we will construct, there is a natural candidate for the horizontal composition 
of two correspondences $E \rightarrow [1]$ and $E' \rightarrow [1]$ that 
are equipped with an equivalence $E_0 \simeq E_1'$. Namely,
we could form the pushout $E \cup_{E_0} E'$, which is canonically fibered over $[2]$. Pulling back this pushout  
along the inner face map $d_1 : [1] \rightarrow [2]$, we obtain a candidate $E' \circ E\rightarrow [1]$ for the composite of $E$ followed by $E'$. The following 
characterization asserts that the Conduch\'e fibrations 
over $[2]$ are precisely those functors that are obtained 
by combining two correspondences via a pushout in the manner above:

\begin{proposition}[{\cite[Lemma 5.16]{AyalaFrancisRozenblyum}}]\label{prop.conduche-fib-characterization}
	A functor $p : E \rightarrow \C$ is a Conduch\'e fibration if and only if for every map $[2] \rightarrow \C$, 
	the induced commutative square 
	\[
		\begin{tikzcd}
			E \times_{\C} \{1\} \arrow[r]\arrow[d] & E \times_{\C} \{0\leq 1\} \arrow[d] \\ 
			E \times_{\C} \{1\leq 2\} \arrow[r] & E \times_{\C} [2]
		\end{tikzcd}
	\]
	of $\infty$-categories is a pushout square.
\end{proposition}

\begin{corollary}
	Every correspondence is a Conduch\'e fibration.
\end{corollary}
\begin{proof}
	This appears as \cite[Corollary 1.13]{AyalaFrancis}, cf.\ \cite[Example B.3.10]{HA}, and follows directly from the above characterization. 
\end{proof}

\begin{construction}\label{constr.ccat}
Conduch\'e fibrations are closed under base change. This can be deduced from the characterization in \cref{prop.conduche-fib-characterization}, 
but also holds more generally for exponentiable morphisms in an $\infty$-category (see e.g.\ \cite[Corollary 2.7]{Comp}). 
It follows that the functor
$
\infty\Cat^\op \rightarrow \widehat{\infty\Cat} : \C \mapsto \infty\Cat_{/\C}
$
obtained by straightening the (large) cartesian fibration $\ev_1 : \fun([1], \infty\Cat) \rightarrow \infty\Cat$, 
gives rise to a subfunctor 
$$
\infty\Cat^\op \rightarrow \widehat{\infty\Cat} : \C \mapsto \Con(\C).
$$
Restricting this functor along the inclusion $\Delta \xrightarrow{\op} \Delta \rightarrow \infty\Cat$,
we obtain a simplicial $\infty$-category
$$
\infty\CCAT : \Delta^\op \rightarrow \widehat{\infty\Cat} : [n] \mapsto \Con([n]).
$$
\end{construction}

\begin{proposition}
	The simplicial $\infty$-category $\infty\CCAT$ is a double $\infty$-category.
\end{proposition}
\begin{proof}
	This is demonstrated  in the proof of \cite[Corollary 1.25]{AyalaFrancis}.
\end{proof}

\begin{example}\label{ex.ccat-hor-identities}
	We can describe the horizontal identity arrow $\C \rightarrow \C$ in $\infty\CCAT$ associated with an $\infty$-category $\C$ as follows. As a correspondence, 
	it is given by the projection $\C \times [1] \rightarrow [1]$. Using the identifications of \cref{prop.equiv-cor-bifib-prof}, one verifies that it is given as a profunctor 
	and two-sided discrete fibration by 
	the mapping space functor $\map_\C(-,-) : \C^\op \times \C \rightarrow \S$ and the projection $(\ev_1, \ev_0) : \fun([1], \C) \rightarrow \C^{\times 2}$ respectively (cf.\ \cite[Example 4.3]{AyalaFrancis}).
\end{example}

From the perspective of two-sided discrete fibrations and profunctors, the horizontal composition in $\infty\CCAT$ can be described as follows:

\begin{proposition}[{\cite[Lemma 4.5]{AyalaFrancis}}]\label{prop.ccat-hor-composition}
	Let $F : \C \rightarrow \D$ and $G : \D \rightarrow \E$ be proarrows of $\infty\CCAT$. 
	Then their composite $G\circ F$ can be described as follows.
	\begin{enumerate}
		\item If $F$ and $G$ correspond to two-sided discrete fibrations $E \rightarrow \C \times \D$ and $ E' \rightarrow \D \times \E$ respectively, 
		then we can consider the induced span  
		$
		E \times_D E' \rightarrow \C \times \E.
		$
		If $W$ denotes the subcategory of $E \times_\D E'$ spanned by the morphisms that are carried to equivalences in $\C \times \E$ by this span, 
		we obtain an induced span 
	$$( E \times_\D E')[W^{-1}] \rightarrow  \C \times \E.$$ 
	This is precisely the two-sided discrete fibration that classifies $G\circ F$.
		\item If $F$ and $G$ correspond to profunctors $F' : \D^\op \times \C \rightarrow \S$ and $G' : \E^\op \times \D \rightarrow \S$ respectively, then 
		the composite $G\circ F$ corresponds to the profunctor
		$$
		\E^\op \times \C \rightarrow \S : (e,c) \mapsto \int^{d \in \D} G'(e,d) \times F'(d,c).
		$$
	\end{enumerate}
\end{proposition}

Using (co)end calculus and the above coend formula for the composite of profunctors, we may obtain the following result:

\begin{corollary}\label{cor.ccat-hor-closed}
	The double $\infty$-category $\infty\CCAT$ is horizontally closed. Concretely, let $W : \C \rightarrow \D$ be a profunctor 
	between $\infty$-categories, viewed as a horizontal arrow of $\infty\CCAT$. Then for every $X \in \infty\CCAT$, we obtain an adjunction
	\[
		W \circ (-) :  \Hor(\infty\CCAT)(X,\C) \rightleftarrows \Hor(\infty\CCAT)(X,\D)  : {}^W{(-)},
	\]
	where the right adjoint carries a profunctor $F :  X \rightarrow \D$ to the profunctor defined by the end
	$$
	{}^W{F}{}(c,x) = \int_{d \in \D} \map_\S(W(d,c), F(d,x)).
	$$
	Similarly, we have an adjunction 
	\[
		(-) \circ W : \Hor(\infty\CCAT)(\D,X) \rightleftarrows \Hor(\infty\CCAT)(\C,X) : (-)^W,
	\]
	where the right adjoint is given by 
	$$
	F^W(x,d) = \int_{c \in \C} \map_\S(W(d,c),F(x,c))
	$$
	for profunctors $F : \C \rightarrow X$.
\end{corollary}
\begin{proof}
	Throughout this proof, we will view the proarrows of $\infty\CCAT$ as profunctors under the equivalences 
	$\Hor(\infty\CCAT)(X,Y) \simeq \fun(Y^\op \times X, \S)
	$
	for $\infty$-categories $X$ and $Y$
	that are supplied by (see \cref{prop.equiv-cor-bifib-prof}).
	Let  $F : X \rightarrow \D$ be a profunctor. We have to show that the presheaf 
	$$
	\map_{\Hor(\infty\CCAT)(X,\D)}(W \circ (-), F) : \Hor(X,\C)^\op \rightarrow \S
	$$ 
	is representable. Let $G : X \rightarrow \C$ be a profunctor. Then we have the following (natural) end formula 
	\begin{align*}
	\map_{\Hor(\infty\CCAT)(X,\D)}(W \circ G, F) & \simeq \int_{(d,x) \in \D^\op \times X} \map_\S((W\circ G)(d,x), F(d,x))
	\end{align*}
	on account of \cite[Proposition 5.1]{GepnerHaugsengNikolaus}. By the Fubini rule for ends \cite[Theorem 2.2]{Loregian} (or see \cite{HaugsengEnds} for an alternative proof),
	this can be computed as the iterated end
	$$
	\map_{\Hor(\infty\CCAT)(X,\D)}(W \circ G, F) \simeq \int_{x \in X}\int_{d \in \D} \map_\S((W\circ G)(d,x), F(d,x)).
	$$
	On account of \cref{prop.ccat-hor-composition}, the right-hand side is computed by 
	\begin{gather*}
	 \int_{x \in X}\int_{d \in \D}\int_{c \in \C}\map_\S(W(d,c) \times G(c,x), F(d,x)) \\\mathrel{\rotatebox[origin=c]{-90}{$\simeq$}}  \\
	 \int_{x \in X}\int_{d \in \D}\int_{c \in \C}\map_\S(G(c,x), \map_\S(W(d,c), F(d,x))) \\\mathrel{\rotatebox[origin=c]{-90}{$\simeq$}} \\
	 \int_{x \in X}\int_{c \in \C}\map_\S(G(c,x), \int_{d \in \D} \map_\S(W(d,c), F(d,x))).
	\end{gather*}
	Using the end formula for mapping spaces of functor categories once again, we deduce that 
	\begin{align*}
		\map_{\Hor(\infty\CCAT)(X,\D)}(W \circ G, F) &\simeq \map_{\Hor(\infty\CCAT)(X,\C)}(G, {}^{W}F),
	\end{align*}
	naturally in $G$.
\end{proof}

\begin{remark}\label{rem.ccat-hor-closed-bifib}
	Riehl and Verity have exhibited the horizontal closure of $\infty\CCAT$ (albeit using a different language) from the perspective of 
	two-sided discrete fibrations in \cite[Theorem 12.3.6]{RiehlVerity}. 

	In the context of \cref{cor.ccat-hor-closed},  we may explicitly compute the two-sided discrete fibration 
	associated to 
	${}^WF$ 
	in the case that $\C = [0]$. In this case, $W$ corresponds to a right fibration $r : A \rightarrow \D$. If $F$ corresponds to a two-sided discrete 
	fibration $(p,q) : E \rightarrow X \times \D$, then ${}^WF$ corresponds to the left fibration $B \rightarrow X$ that sits in the pullback square 
	\[
		\begin{tikzcd}
			B \arrow[r]\arrow[d] & \fun(A, E) \arrow[d, "{\fun(A, (p,q))}"] \\
			X \arrow[r] & \fun(A, X) \times \fun(A, \D),
		\end{tikzcd}
	\]
	of $\infty$-categories. Here the bottom arrow is the product of the diagonal functor and the constant functor that selects $r$. This follows directly 
	from \cite[Proposition 4.24]{Stevenson}, or may be deduced from \cite[Lemma 12.3.7]{RiehlVerity}. An alternative computation is given in \cite{EquipII} using 
	double $\infty$-categorical methods.
\end{remark}

\begin{remark}\label{rem.ccat-size}
	Passing to larger universes, one can repeat the same constructions in this section to obtain double $\infty$-categories for larger $\infty$-cate\-gories.
	For instance, we have a (very large) double $\infty$-category 
	$$
	\widehat{\infty\CCAT} 
	$$ 
	of large $\infty$-categories. The proarrows of this double $\infty$-category 
	are again given by correspondences, two-sided discrete fibrations, and profunctors (but the latter are now valued in $\hat{\S}$, the $\infty$-category of large spaces).
	Note that there is a canonical functor $\infty\CCAT \rightarrow \widehat{\infty\CCAT}$ that is level-wise fully faithful. 
	Consequently, throughout this article, we will mostly be implicit about enlarging universes, 
	and hence leave out the decorations with hat symbols.
\end{remark}

\subsection{The vertical fragment}\label{ssection.comparison-models-infty2-infty-cats}
The vertical fragment $$\infty\CAT := \Vert(\infty\CCAT)$$ will be our construction of the $(\infty,2)$-cate\-go\-ry $\infty\CAT$ of 
$\infty$-categories. We will justify this by comparing this vertical fragment 
to another model for the $(\infty,2)$-category of $\infty$-categories due to Lurie \cite{LurieInfty2}.

There is a model structure on the category of \textit{marked simplicial sets} 
$
\sSet^+,
$
which is equivalent to the Joyal model structure on simplicial sets \cite[Chapter 3]{HTT}. We may then consider the category 
$
\Cat^{\sSet^+}
$
of \textit{marked simplicial categories}: categories enriched over $\sSet^+$ with respect to the cartesian monoidal structure.
This carries an induced model structure \cite[Section A.3]{HTT}. In \cite{LurieInfty2}, Lurie shows 
that there exists a \textit{nerve functor} 
$$
N_\mathrm{Lurie} : \Cat^{\sSet^+} \rightarrow (\infty,2)\Cat,
$$
that carries the weak equivalences $W$ between fibrant objects
to equivalences (by first passing through the category of \textit{scaled simplicial sets}). 
Let us write $\sSet^{+,\circ} \subset \sSet^+$ for the full marked simplicial subcategory of $\sSet^+$ spanned by the fibrant objects: the quasi-categories with their equivalences marked. 
Note that this is a fibrant marked simplicial category, and we will write 
$$
\infty\CAT_\mathrm{Lurie} := N_\mathrm{Lurie}(\sSet^{+,\circ}).
$$
The goal is now to show the following:

\begin{proposition}\label{prop.comparison-models-infty2-infty-cats}
	There exists an equivalence of $(\infty,2)$-categories
	$$
	\infty\CAT \simeq \infty\CAT_\mathrm{Lurie}.
	$$
\end{proposition}

To prove this, we will make use of a lax version of the straightening theorem due to Haugseng--Hebestreit--Linskens--Nuiten \cite{HHLN}. Recall that the ordinary straightening theorem sets up an equivalence 
$
 \fun(\C, \infty\Cat) \simeq \Cocart(\C),
$
 for every $\infty$-category
$\C$. Here 
$
\Cocart(\C) \subset \infty\Cat_{/\C} 
$
is the subcategory spanned by the cocartesian fibrations and morphisms that preserve cocartesian morphisms. 
Following \cite{HHLN}, let  us write 
$$
\Cocart^\lax(\C) \subset \infty\Cat_{/\C} 
$$
for the \textit{full} subcategory spanned by the cocartesian fibrations.
\begin{theorem}[{\cite[Theorem 5.3.1]{HHLN}}]\label{thm.lax-straightening}
	There exists an equivalence
	$$
	\Sq(\infty\CAT_\mathrm{Lurie})_{n,m} = \map_{(\infty,2)\Cat}([n] \otimes [m], \infty\CAT_\mathrm{Lurie}) \simeq \map_{\infty\Cat}([n], \Cocart^\lax([m]))
	$$
	that is natural in $[n],[m] \in \Delta$.
\end{theorem}

\begin{proof}[Proof of \cref{prop.comparison-models-infty2-infty-cats}]
	On account of \cref{ex.hor-vert-sq}, and fully faithfulness of $(-)_v$,
	it is necessary and sufficient to show that there is an equivalence $$\map_{\DbliCat}([n;m, \dotsc, m]_v, \CCAT_\infty) \xleftarrow{\phi_{n,m}} \map_{\DbliCat}([n;m, \dotsc, m]_v, \Sq(\infty\CAT_\mathrm{Lurie}))$$
	that is natural in $[n], [m] \in \Delta^{\times 2}$.
	We have a natural commutative diagram 
	\[ 
		\begin{tikzcd}
			\map(\{0,\dotsc,n\}, \infty\Cat) \arrow[r]\arrow[d,equal] & \map(\{0,\dotsc,n\}, \Cocart^\lax([m]))\arrow[d] & \arrow[l] \map([n], \Cocart^\lax([m]))\arrow[d]\\
			\map(\{0,\dotsc,n\}, \infty\Cat) \arrow[r] & \map(\{0,\dotsc,n\}, \Con([m])) & \arrow[l] \map([n], \Con([m]))
		\end{tikzcd}
	\]
	where the left horizontal arrows are induced by the functor $[m] \times - : \infty\Cat \to \infty\Cat_{/[m]}$. The middle and right vertical arrows 
	are induced by the full inclusion $\Cocart^\lax([m]) \subset \Con([m])$, see \cite[Lemma 2.15]{AyalaFrancis}. The pullbacks of the first and second row are naturally computed 
	by the source and target of the desired map $\phi_{n,m}$ on account of \cref{rem.h-v-inclusions} and \cref{thm.lax-straightening}, so that the diagram gives rise to a candidate map $\phi_{n,m}$.
	We now note that the middle and right vertical arrows in the above diagram are monomorphisms, 
	and that the square on the right is a pullback square. This allows us to conclude that $\phi_{n,m}$ is an equivalence.
\end{proof}

\section{Example: the double \texorpdfstring{$\infty$}{∞}-category of internal categories}\label{section.equipment-internal-cats}

This section is dedicated to the construction of the ambient double $\infty$-category
$$
\CCAT(\E)
$$ 
of categories internal to suitable $\infty$-category $\E$. It generalizes the construction of \cref{section.equipment-infty-cats}
so that we recover $\infty\CCAT$ when $\E=\S$. We will continue the study 
of this double $\infty$-category in the sequel article \cite{EquipII} to recover the foundations of internal category theory using the material of \cref{section.formal-category-theory}.

Throughout this section, we will fix an $\infty$-category $\E$ 
so that: 
\begin{enumerate}[noitemsep]
	\item $\E$ is limit complete,
	\item $\E$ has \textit{universal geometric realizations}, i.e.\ $\E$ admits all colimits 
	of $\Delta^\op$-shaped diagrams, and
	 each base change functor 
	$$
	f^* : \E_{/y} \rightarrow \E_{/x}
	$$
	preserves $\Delta^\op$-shaped colimits for every arrow $f : x\rightarrow y$ in $\E$.
\end{enumerate}
\noindent 
For example, complete and cocomplete locally cartesian closed $\infty$-categories meet these conditions. Assumption (2) implies the following:

\begin{lemma}\label{lem.geom-fib-prod-compatible}
	For any pullback square
	\[
		\begin{tikzcd}
			A \arrow[r]\arrow[d] & B\arrow[d] \\
			C \arrow[r] & D
		\end{tikzcd}
	\]
	in $\fun(\Delta^\op, \E)$ so that  $D$ is essentially constant, the induced square 
	\[
		\begin{tikzcd}
			\colim_{\Delta^\op} A \arrow[r]\arrow[d] & \colim_{\Delta^\op} B\arrow[d] \\
			\colim_{\Delta^\op} C \arrow[r] & \colim_{\Delta^\op} D \simeq D_0
		\end{tikzcd}
	\]
	on geometric realizations is a pullback square as well.
\end{lemma}
\begin{proof}
One may construct a canonical commutative triangle
\[
	\begin{tikzcd}
		\colim_{\Delta^\op} A \arrow[r]\arrow[dr] & \colim_{([n],[m])\in \Delta^\op \times \Delta^\op}B_n \times_{D_0} C_m \arrow[d] \\
		 & \colim_{\Delta^\op} B \times_{D_0} \colim_{\Delta^\op} C
	\end{tikzcd}
\]
in $\E$. The horizontal arrow is an equivalence since the diagonal $\Delta^\op \rightarrow \Delta^\op \times \Delta^\op$ is final. Hence, 
it remains to verify that the comparison map
$$
\colim_{([n],[m])\in \Delta^\op \times \Delta^\op}B_n \times_{D_0} C_m \simeq \colim_{[n] \in \Delta^\op} \colim_{[m]\in \Delta^\op} B_n \times_{D_0} C_m \rightarrow \colim_{\Delta^\op} B \times_{D_0} \colim_{\Delta^\op} C
$$
is an equivalence. But this now readily follows from the universality of geometric realizations in $\E$.
\end{proof}

\subsection{The double $\infty$-category of Segal objects}\label{ssection.conduche-modules}

We commence by introducing an internal variant of Conduch\'e fibrations. This is essentially an adaptation of the material in \cite{HaugsengMorita} 
applied to the cartesian case. 

\begin{definition}\label{def.cat-modules}
	Let $n \geq 0$ be an integer.
	 A \textit{Segal $[n]$-module in $\E$} 
	is a functor 
	$
	M : (\Delta_{/[n]})^\op \rightarrow \E 
	$
	so that for any map $f : [k] \rightarrow [n]$ in $\Delta$, the comparison map
	$$
	M(f) \rightarrow M(f|\{0,1\}) \times_{M(f|\{1\})} \dotsb \times_{M(f|\{k-1\})} M(f|\{k-1,k\})
	$$
	is an equivalence. We will write 
	$$
	{\Delta_{/[n]}}\Seg(\E) \subset \fun((\Delta_{/[n]})^\op, \E)
	$$
	for the full subcategory spanned by the Segal $[n]$-modules.
\end{definition}

Let us consider the cosimplicial object
$$
C : \Delta \xrightarrow{\op} \Delta \xrightarrow{\Delta_{/[-]}} \infty\Cat \xrightarrow{(-)^\op} \infty\Cat
$$
where the middle functor acts by pushing forward.
Then we obtain a simplicial $\infty$-category $$\fun(C_-, \E) : \Delta^\op \rightarrow \widehat{\infty\Cat} : [n] \mapsto \fun((\Delta_{/[n]})^\op, \E).$$
One may readily verify that 
 if $\alpha : [m] \rightarrow [n]$ is a structure map and $M$ is a Segal $[n]$-module in $\E$,
then the \textit{restriction of $M$ along $\alpha$} 
$$
\alpha^*M : (\Delta_{/[m]})^\op \xrightarrow{(\alpha^\op_!)^\op} (\Delta_{/[n]})^\op \rightarrow \E
$$
is a Segal $[m]$-module. Hence, we get a subobject $\Delta_{/[-]}\Seg(\E)$ of $\fun(C_-, \E)$ 
that is spanned by the Segal modules. To obtain a double $\infty$-category, we will need to 
restrict to the 
$[2]$-modules that are obtained by composing two $[1]$-modules. We will use the following auxiliary shape category that was defined by Haugseng:

\begin{definition}[{\cite[Definition 4.13]{HaugsengMorita}}]
	A map $f : [m] \rightarrow [n]$ is called \textit{cellular} if $f(i+1) \leq f(i) +1$ for all $i \in [m]$, i.e.\ if $f$ has convex image. The 
	full subcategory of $\Delta_{/[n]}$ spanned by these cellular maps is denoted by $\Lambda_{/[n]}$.
\end{definition}

We then have the following analogue of \cref{prop.conduche-fib-characterization}:

\begin{proposition}\label{prop.conduche-modules}
	Let $M$ be a Segal $[n]$-module in $\E$. Then, the following statements are equivalent: 
	\begin{enumerate}
		\item $M$ is left Kan extended along the inclusion $$(\Lambda_{/[n]})^\op \rightarrow (\Delta_{/[n]})^\op,$$
		\item for any map $\alpha : [2] \rightarrow [n]$, the restriction $\alpha^*M$ is left Kan extended along the inclusion 
		$$(\Lambda_{/[2]})^\op \rightarrow (\Delta_{/[2]})^\op,$$
		\item for any map $\alpha : [2] \rightarrow [n]$, the canonical map 
		$$
		\colim_{[k] \in \Delta^\op} M([k+2] \xrightarrow{0 \leq 1 \dotsb \leq 1\leq 2} [2] \xrightarrow{\alpha}
		 [n]) \rightarrow M([1] \xrightarrow{0\leq 2} [2] \xrightarrow{\alpha} [n])
		$$
		is an equivalence.
	\end{enumerate}
\end{proposition}

\begin{remark}
Note that condition (3) of \cref{prop.conduche-modules} essentially asserts that if $M$ is a Segal $[2]$-module in $\E$, then $M$ is 
left Kan extended along $(\Lambda_{/[2]})^\op \rightarrow (\Delta_{/[2]})^\op$ if and only if $M(\{0\leq 2\})$ is the two-sided bar construction 
of $M(\{0 \leq 1\})$ and $M(\{1 \leq 2\})$ over $M(\{1\})$. For $\E = \S$, the same bar construction formula makes an appearance in \cite[Lemma 5.16(5)]{AyalaFrancisRozenblyum}.
\end{remark}

In order to prove the above proposition, we will make use of the following construction and observation (cf.\ Proposition 4.16 and Lemma 4.17 of \cite{HaugsengMorita}):

\begin{construction}\label{con.conduche-modules-finality}
	Let $f : [m] \rightarrow [n]$ be a map of posets. Then we will write $D(f)$ for the set that measures the non-cellularity of $f$, 
	i.e.\, it is given by the difference between the convex hull of $\im(f)$ and $\im(f)$:
	$$
		D(f) := \{f(0), f(0)+1,\dotsc, f(m)-1, f(m)\} \setminus \im(f).
	$$
	We define a functor
	$$
	\textstyle T_f : \prod_{D(f)}\Delta \rightarrow (\Lambda_{/[n]})_{f/} := (\Delta_{/[n]})_{f/} \times_{\Delta_{/[n]}} \Lambda_{/[n]}
	$$
	that carries an indexed poset $[k_\bullet]$ to the triangle
	\[
	\begin{tikzcd}
		{[m]} \arrow[rr,"{\iota_{[k_\bullet]}}"]\arrow[dr, "f"'] && {[m + \sum_{i \in D(f)} (k_i + 1)]} \arrow[dl, "{f_{[k_\bullet]}}"] \\ 
		& {[n]}
	\end{tikzcd}
	\]
	where $f_{[k_\bullet]}$ is the map
	$$
	\textstyle f_{[k_\bullet]} : [1+\sum_{i \in D(f|\{0,1\})} (k_i + 1)] \cup_{[0]} \dotsb \cup_{[0]} [1 + \sum_{i \in D(f|\{m-1,m\})} (k_i + 1)] \rightarrow [n]
	$$
	so that $f^{j,j+1}_{[k_\bullet]} := f_{[k_\bullet]}|[1 + \sum_{i \in D(f|\{j,j+1\})}( k_i+1)]$ is given by 
	$$
	f^{j,j+1}_{[k_\bullet]}(x) = \begin{cases}
		f(j), & x = 0 \\
		t, & \sum_{i=f(j)+1}^{t-1} (k_i + 1) < x \leq \sum_{i=f(j)+1}^t(k_i+1) \\
		f(j+1), & \text{if } x= 1 + \sum_{i \in D(f|\{j,j+1\})} (k_i + 1),
	\end{cases}
	$$
	and the top map $\iota_{[k_\bullet]}$ is the unique injective map that fits in the triangle. It is readily verified 
	that this construction is functorial, with the functoriality 
	of $T_f$ being uniquely determined by pushing forward structure maps along the
	 canonical inclusions $$\textstyle \rho^i_{[k_\bullet]} : [k_i] \rightarrow [1 + \sum_{i \in D(f|\{j,j+1\})}( k_i+1)] \rightarrow [m + \sum_{i \in D(f)} (k_i + 1)].$$
\end{construction}

\begin{lemma}\label{lemma.conduche-modules-finality}
	Let $f : [m] \rightarrow [n]$ be a map of posets. Then the functor $T_f$ of \cref{con.conduche-modules-finality} is final.
\end{lemma}
\begin{proof}
	In light of Quillen's theorem A \cite[Theorem 4.1.3.1]{HTT}, it suffices to show that the category 
	$$
	\textstyle T_{f/X} := (\Lambda_{/[n]})_{f//X} \times_{(\Lambda_{/[n]})_{f/}}\prod_{D(f)}\Delta
	$$
	is weakly contractible for every $X \in (\Lambda_{/[n]})_{f/}$. Note that such $X$ is given by a triangle 
	\[
		\begin{tikzcd}
			{[m]} \arrow[rr, "h"]\arrow[dr, "f"'] && {[l]} \arrow[dl, "{g}"] \\ 
			& {[n]},
		\end{tikzcd}
	\]
	so that $g$ is cellular. It is readily verified that maps $\phi : T_f([k_\bullet]) \rightarrow X$ in $(\Lambda_{/[n]})_{f/}$
	are uniquely determined by the restrictions $\phi\rho^i_{[k_\bullet]} : [k_i] \rightarrow [l]$. 
	Since each restriction should have image in the (non-empty) fiber $g^{-1}(i) \subset [l]$, 
	we see that the choice of $\phi$ so that 
	$
	\phi\rho^i_{[k_\bullet]} : [k_i] \rightarrow [l]
	$
	is the unique injective map with image $g^{-1}(i)$, constitutes a final object of the lax slice $T_{f/X}$.  
\end{proof}

Additionally, we will need the following two results:

\begin{lemma}\label{lemma.conduche-modules-kan-extension}
	Suppose that $N : (\Lambda_{/[n]})^\op \rightarrow \E$ satisfies the Segal condition as in \cref{def.cat-modules}. Then 
	the left Kan extension $M : (\Delta_{/[n]})^\op \rightarrow \E$ of $N$ along $(\Lambda_{/[n]})^\op \rightarrow (\Delta_{/[n]})^\op$ satisfies the Segal condition as well.
\end{lemma}
\begin{proof} 
	Let $f : [m] \rightarrow [n]$ be a map. Then, in view of \cref{lemma.conduche-modules-finality} and the Segal condition 
	for $N$, we compute that
	\begin{align*}
		M(f) & = \colim_{[k_\bullet] \in \prod_{D(f)}\Delta^{\op}} N(f_{[k_\bullet]}) \\
		&\simeq \colim_{[k_\bullet] \in \prod_{D(f)}\Delta^{\op}} N(f^{0,1}_{[k_\bullet]}) \times_{N(f|\{1\})} \dotsb \times_{N(f|\{m-1\})} N(f^{m-1,m}_{[k_\bullet]}).
	\end{align*}
	Using the fact that we have a decomposition $D(f) = D(f|\{0,1\}) \coprod \dotsb \coprod D(f|\{m-1,m\})$ 
	and \cref{lem.geom-fib-prod-compatible}, we deduce that 
	$$
	M(f) \simeq M(f|\{0,1\}) \times_{M(f|\{1\})} \dotsb \times_{M(f|\{m-1\})} M(f|\{m-1,m\}),
	$$
	as desired.
\end{proof}

\begin{lemma}\label{lemma.conduche-modules-induction}
	Suppose that $M$ is a Segal $[n]$-module in $\E$. 
	Let $f : [1] \rightarrow [n]$ be a map of posets that  
	 factors as $$[1] \xrightarrow{g} [m] \xrightarrow{\alpha} [n]$$ so that $\alpha$ is injective.
	Then there is a canonical commutative triangle 
	\[
		\begin{tikzcd}[column sep = tiny]
			\colim_{[k_\bullet] \in \prod_{D(g)}\Delta^\op} M(\alpha g_{[k_\bullet]})\arrow[dr] \arrow[rr] && \colim_{[k_\bullet] \in \prod_{D(f)}\Delta^\op} M(f_{[k_\bullet]}),\arrow[dl] \\
			& M(f)
		\end{tikzcd}
	\]
	and the top map is an equivalence as soon as 
	 for all $0 \leq j < m$, the comparison map 
	$$
	\colim_{[k_\bullet] \in \prod_{D(\alpha|\{j,j+1\})}\Delta^\op} M(\alpha^{j,j+1}_{[k_\bullet]}) \rightarrow M(\alpha|\{j,j+1\}) 
	$$
	is an equivalence.
\end{lemma}
\begin{proof}
	Note that $\alpha$ restricts to an inclusion $D(g) \rightarrow D(f)$ and gives rise to a canonical 
	natural transformation 
	\[
		\begin{tikzcd}[row sep = tiny]
			&  \prod_{D(g)} \Delta^\op\arrow[r, "T_g"name=f]  & (\Delta_{/[m]})^\op \arrow[dr, "(\alpha_!)^\op", bend left = 10pt] \\
			\prod_{D(f)} \Delta^\op\arrow[ur, "\alpha^*", bend left = 10pt] \arrow[rrr, "T_f"'name=t, bend right = 10pt] &&& (\Delta_{/[n]})^\op.
			\arrow[from=f,to=t, Rightarrow, shorten >= 6pt, shorten <= 6pt]
		\end{tikzcd}
	\]
	Hence, we obtain maps
	$$
	\colim_{[k_\bullet] \in \prod_{D(g)}\Delta^\op} M(\alpha g_{[k_{\alpha(\bullet)}]}) \xleftarrow{\simeq} \colim_{[k_\bullet] \in \prod_{D(f)}\Delta^\op} M(\alpha g_{[k_{\alpha(\bullet)}]}) \rightarrow \colim_{[k_\bullet] \in \prod_{D(f)}\Delta^\op} M(f_{[k_\bullet]}) ,
	$$
	which are all compatible with the canonical projections to $M(f)$. 
	The first map is an equivalence as $\alpha^*$ is final, 
	so that we may obtain the desired commutative triangle by picking an inverse. 
	To show the latter statement, we 
	suppose that the comparison maps for $\alpha$ as stated in the lemma, 
	are equivalences. By 2-out-of-3, it suffices to show that the right arrow in the span displayed above is an equivalence, 
	which readily follows from a similar computation as in \cref{lemma.conduche-modules-kan-extension}.
\end{proof}

\begin{proof}[Proof of \cref{prop.conduche-modules}]
	Let us first show that (2) and (3) are equivalent. A Segal module 
	$M : (\Delta_{/[2]})^\op \rightarrow \E$ is left Kan extended along the inclusion $(\Lambda_{/[2]})^\op \rightarrow (\Delta_{/[2]})^\op$ if and only if 
	for any map $f : [k] \rightarrow [2]$, the canonical unit map 
	$$
	\colim_{g \rightarrow f \in (\Lambda_{/[2]})^\op/f} M(g) \rightarrow M(f)
	$$
	is an equivalence. In light of the assumption on $M$ and 
	\cref{lemma.conduche-modules-kan-extension}, both sides satisfy the Segal condition when varying $f$.
	Hence, one may reduce to checking that the above map is an equivalence for $k = 1$. 
	The only non-cellular map is then given by $f = d_1$, so that after
	rewriting the above colimit using \cref{lemma.conduche-modules-finality}, one directly deduces that (2) and (3) are equivalent.
	
	Next, we show that (1) implies (3). Suppose that $M$ 
	is a Segal $[n]$-module so that (1) holds. If $\alpha : [2] \rightarrow [n]$ 
	is an injective map, then the desired result 
	follows from applying \cref{lemma.conduche-modules-induction} with $g = d_1$. In the case 
	that $\alpha$ is not injective, the result follows from diagrammatic reasons. 
	To wit, we note that the simplicial object 
	$$
	S : \Delta^\op \rightarrow (\Delta_{/[n]})^\op : [k] \mapsto ([k+2] \xrightarrow{0 \leq 1 \dotsb \leq 1 \leq 2} [2] \xrightarrow{\alpha} [n]) 
	$$
	is (always) augmented by $\alpha d_1$, and this augmentation induces the comparison map
	$$
	\colim_{\Delta^\op} MS \rightarrow M(\alpha d_1)
	$$
	in question.
	In the case that
	$\alpha$ is not injective, one can verify that $S$ admits a \textit{splitting} (see \cite[Definition 2.3.13]{RiehlVerity}). 
	Hence, the composite $MS$ admits a splitting and the comparison map will be an equivalence on account of \cite[Proposition 2.3.15]{RiehlVerity}.

	Finally, suppose that $M$ is a Segal $[n]$-module so that (3) holds. Then we will show that $M$ satisfies (1). 
	To this end, we proceed by induction as follows. 
	For $f : [m] \rightarrow [n]$, let use define its \textit{non-cellularity index} $I(f)$ by 
	$$
	I(f) := \max_{0\leq j < m} \abs{D(f|\{j,j+1\})}.
	$$
	We will show that for every map $f$, the comparison map 
	$$
	c(f) : \colim_{[k_\bullet] \in \prod_{D(f)}\Delta^\op} M(f_{[k_\bullet]}) \rightarrow M(f)
	$$
	is an equivalence by induction on $I(f)$. If $I(f) = 0$, the statement is trivial. 
	Suppose now that the statement holds for all maps $f$ so that $I(f) \leq i$, with $i \geq 0$. In order to show that the statement holds for $i+1$, 
	we may use \cref{lemma.conduche-modules-kan-extension} to reduce to showing the statement 
	for any map $f : [1] \rightarrow [n]$ with $I(f) = i+1$. We may factor such a map $f$ as 
	$$
	[1] \xrightarrow{d_1} [2] \xrightarrow{\alpha} [n]
	$$
	so that $\alpha$ is injective with $I(\alpha) = i$. 
	It follows now directly from the assumptions and \cref{lemma.conduche-modules-induction} that $c(f)$ is an equivalence as well. 
\end{proof}

\begin{definition}
	If $M$ is a Segal $[n]$-module in $\E$ that satisfies the equivalent conditions of \cref{prop.conduche-modules}, 
	then we say that $M$ is a \textit{Conduch\'e module}. We write $$\Delta_{/[n]}\Con(\E) \subset {\Delta_{/[n]}}\Seg(\E)$$ for the full subcategory spanned by 
	the Conduch\'e modules.
\end{definition}

\begin{corollary}
	All Segal $[0]$-modules and Segal $[1]$-modules in $\E$ satisfy the Conduch\'e condition.
\end{corollary}

\begin{construction}
	In light of \cref{prop.conduche-modules}, the Conduch\'e property for Segal modules is preserved under 
	restriction along maps in $\Delta$. Consequently, we obtain a further subobject 
	$$
	\SSEG(\E) := \Delta_{/[-]}\Con(\E) \subset \Delta_{/[-]}\Seg(\E) \subset \fun(C_-, \E)
	$$
	that is spanned by the Conduch\'e modules in each level.
\end{construction}

Now, a standard argument shows the following:

\begin{proposition}\label{prop.double-infty-cat-cat-objs}
	The simplical $\infty$-category $\SSEG(\E)$ is a double $\infty$-category.
\end{proposition}
\begin{proof}
	Let $\Delta_{/[n]}\Con'(\E) \subset \fun((\Lambda_{/[n]})^\op, \E)$ be the full subcategory spanned 
	by those functors that satisfy the Segal condition as in \cref{def.cat-modules}.
	Then it readily 
	follows from characterization (1) of \cref{prop.conduche-modules} and \cref{lemma.conduche-modules-kan-extension} that the fully faithful left Kan extension functor 
	$
	i_!: \fun((\Lambda_{/[n]})^\op, \E) \rightarrow \fun((\Delta_{/[n]})^\op, \E)
	$
	along $i : (\Lambda_{/[n]})^\op \rightarrow (\Delta_{/[n]})^\op$
	restricts to an equivalence between $\Delta_{/[n]}\Con'(\E)$ and $\Delta_{/[n]}\Con(\E)$. Consequently, 
	we see that it suffices to check that the canonical functor 
	$$
	p : \Delta_{/[n]}\Con'(\E) \rightarrow {\Delta_{/[1]}}\Seg(\E) \times_{\Seg(\E)} \dotsb \times_{\Seg(\E)} {\Delta_{/[1]}}\Seg(\E)
	$$
	is an equivalence.

	To this end, we note that this functor may be identified with the restriction of the left adjoint 
	that sits in the adjunction
	$$
	j^* : \fun((\Lambda_{/[n]})^\op, \E) \rightleftarrows \fun((\Lambda_{/[n]})^{\mathrm{cell}, \op},\E)  : j_*
	$$
	associated to \textit{right} Kan extending along the inclusion
	$$j : (\Lambda_{/[n]})^{\mathrm{cell},\op} \simeq (\Delta_{/[1]})^\op \cup_{\Delta^\op} \dotsb \cup_{\Delta^\op} (\Delta_{/[1]})^\op \rightarrow (\Lambda_{/[n]})^\op$$
	of the full subcategory of those cellular maps $f : [m] \rightarrow [n]$ so that $\im(f) \subset \{i,i+1\}$ for some $i$. 
	Let us write $\Delta_{/[n]}\Con''(\E) \subset \fun((\Lambda_{/[n]})^{\mathrm{cell}, \op},\E)$ for the full subcategory of 
	those functors that satisfy the Segal condition. Then $p$ may be identified with the restriction of $j^*$ 
	to $\Delta_{/[n]}\Con'(\E)$ and $\Delta_{/[n]}\Con''(\E)$. 
	Since $j_*$ is fully faithful, we thus see that $p$ is an equivalence if the following holds:
	\begin{enumerate}
		\item the left adjoint $j_*$ carries $\Delta_{/[n]}\Con''(\E)$ into $\Delta_{/[n]}\Con'(\E)$,
		\item if $M \in \Delta_{/[n]}\Con'(\E)$, then $M$ is right Kan extended along $j$, i.e.\, the unit
		$M \rightarrow j_*j^*M$ is an equivalence,
	\end{enumerate}
	in which case $j_*$ restricts to an inverse of $p$. 
	
	Both (1) and (2) follow from the following computation. Suppose that 
	 $M : (\Lambda_{/[n]})^{\mathrm{cell}, \op} \rightarrow \E$ is a functor, and let $f : [m] \rightarrow [n]$ be a cellular map, then 
	$$
	j_*M(f) \simeq M(f|f^{-1}(\{\alpha, \alpha+1\})) \times_{M(f|f^{-1}\{\alpha+1\})} \dotsb \times_{M(f|f^{-1}\{\omega-1\})} M(f|f^{-1}(\{\omega-1,\omega\})),
	$$
	where $\alpha := \min(f)$ and $\omega := \max(f)$. 
	Consequently, if $M$ satisfies the Segal condition then we get that
	\begin{displaymath}
	j_*M(f) \simeq M(f|\{0,1\}) \times_{M(f|\{1\})} \dotsb \times_{M(f|\{m-1\})} M(f|\{m-1,m\}). \qedhere
	\end{displaymath}
\end{proof}

Finally, we record the following functoriality of the construction above. It follows immediately from \cref{def.cat-modules} and \cref{prop.conduche-modules}.

\begin{proposition}\label{prop.int-ccat-functoriality}
	Suppose that $u : \E \rightarrow \F$ is a functor to an $\infty$-category $\F$ that also has properties (1) and (2) stated at the start of \cref{section.equipment-internal-cats}. 
	If $u$ preserves pullbacks and geometric realizations, then post-composition with $u$ gives 
	rise to a functor $$\SSEG(u) : \SSEG(\E) \rightarrow \SSEG(\F).$$
\end{proposition}	

\subsection{Completeness conditions} The constructed double $\infty$-category $\SSEG(\E)$ fails to be locally complete. To remedy this, we will need 
to pass to the subobject spanned by the complete Segal objects (see \cref{def.comp-cat-obj}).

\begin{construction}\label{con.equipment-int-ccat}
	We will consider the sub double $\infty$-category
	$$
	\CCAT(\E) \subset \SSEG(\E)
	$$
	so that $\CCAT(\E)_n \subset \SSEG(\E)_n = \Delta_{/[n]}\Con(\E)$ is the full subcategory spanned by those Conduch\'e $[n]$-modules 
	$M$ so that the restriction $(\{i\} \to [n])^*M$ is a complete Segal object in $\E$ for each $i \in [n]$. 
\end{construction}

The next goal is now to show that $\CCAT(\E)$ is a locally complete double $\infty$-category. We will use the following auxiliary construction.

\begin{construction} 
	The evaluation functor
	$
	\Delta^\op \times \fun(\Delta^\op, \E) \rightarrow \E 
	$
	may be right Kan extended to a functor
	$$
	\PSh(\Delta)^\op \times \fun(\Delta^\op, \E) \rightarrow \E : (A, X) \mapsto X(A)
	$$
	along the fully faithful inclusion $y^\op \times \id_{\fun(\Delta^\op, \E)}$, where 
	$y : \Delta \rightarrow \PSh(\Delta)$ denotes the Yoneda embedding. We will make use of this extension throughout this section. Its values are computed by the formula
	$$
	X(A) = \lim_{[n] \rightarrow A \in (\Delta_{/A})^\op} X([n])
	$$
	where we wrote $\Delta_{/A} := \PSh(\Delta)_{/A} \times_{\PSh(\Delta)} \Delta.$
\end{construction}

\begin{remark}
	Note that $X \in \Seg(\E)$ is complete if and only if the map 
	$
	X([0]) \rightarrow X(J)
	$
	induced by $J\rightarrow [0]$ is an equivalence. Here $J$ is the simplicial set that was defined in \cref{ex.segal-spaces}.
\end{remark}

\begin{lemma}\label{lem.int-ccat-slices}
	Suppose that 
	\[
	\begin{tikzcd}
		A \arrow[r,""]\arrow[d,""'] & B \arrow[d,""] \\
		C \arrow[r,""] & D
	\end{tikzcd}
	\]
	is a pushout square of simplicial spaces. Then 
	the induced commutative square
	\[
	\begin{tikzcd}
		\Delta_{/A} \arrow[r]\arrow[d] & \Delta_{/B} \arrow[d] \\
		\Delta_{/C} \arrow[r] & \Delta_{/D}
	\end{tikzcd}
	\]
	is a pushout square of $\infty$-categories. 
\end{lemma}
\begin{proof}
	The pushout square of simplicial spaces corresponds to the pushout square 
	\[
	\begin{tikzcd}
		\Delta_{/A} \arrow[r]\arrow[d] & \Delta_{/B} \arrow[d] \\
		\Delta_{/C} \arrow[r] & \Delta_{/D}
	\end{tikzcd}
	\]
	in the $\infty$-category $\RFib(\Delta)$ of right fibrations over $\Delta$ under the straightening equivalence 
	$\RFib(\Delta) \simeq \PSh(\Delta)$. The desired result now follows from the fact that the projection 
	$
	\RFib(\Delta) \subset \infty\Cat_{/\Delta} \rightarrow \infty\Cat 
	$
	preserves colimits on account of \cite[Corollary 7.6]{GepnerHaugsengNikolaus}.
\end{proof}

\begin{proposition}\label{prop.int-ccat-loc-comp}
	The double $\infty$-category $\CCAT(\E)$ is locally complete.
\end{proposition}
\begin{proof}
	To show that $\CCAT(\E)$ is locally complete, we have to show that the functor 
	$$
	\CCAT(\E)_0 \rightarrow \CCAT(\E)_\mathrm{eq}
	$$
	is fully faithful. In light of \cref{lem.int-ccat-slices}, this can be identified with the restriction of the functor 
	$
	p^* : \fun(\Delta^\op,\E) \rightarrow \fun((\Delta_{/J})^\op, \E)
	$
	induced by the projection $p : (\Delta_{/J})^\op \rightarrow \Delta^\op$. This functor admits a right adjoint $p_*$ given by right Kan extensions. It thus suffices to show that the unit 
	$$
	\C \rightarrow p_*p^*\C 
	$$
	is an equivalence for every $\C \in \Cat(\E)$. This may be checked level-wise. At level $[m] \in \Delta$, one may compute using the usual 
	formula for Kan extensions that it corresponds to the map
	$$
	\C([m]) \rightarrow \C([m] \times J)
	$$
	that is induced by the projection $[m] \times J \rightarrow [m]$. Let $S$ be the class of maps $f : A \rightarrow B$ in $\PSh(\Delta)$ so that $\C(f) : \C(B) \rightarrow \C(A)$ is an equivalence. Then one 
	may readily verify that $S$ is strongly saturated in the sense of \cite[Definition 5.5.4.5]{HTT}. Now, by assumption, $S$ contains the spine inclusions 
	$[1] \cup_{[0]} \dotsb \cup_{[0]} [1] \rightarrow [n]$ ($\C$ is a Segal object) and the map $J \rightarrow [0]$ ($\C$ is complete). To show that $[m] \times J \rightarrow [m]$ is contained in $S$, one may use 
	the spine inclusions to reduce to the case that $m =1$. This can now be deduced from Rezk's work \cite[Proposition 12.1]{RezkSeg}.
\end{proof}

The double $\infty$-categories of internal categories are functorial with respect to the following class of maps:

\begin{proposition}\label{prop.int-ccat-topos-functoriality}
	Suppose that $u : \E \rightarrow \F$ is a functor to an $\infty$-category $\F$ that also meets conditions (1) and (2) stated at the start of \cref{section.equipment-internal-cats}. If 
	$u$ preserves pullbacks and geometric realizations, 
	then post-composition with $u$ 
	gives rise to a functor $$\CCAT(u) : \CCAT(\E) \rightarrow \CCAT(\F).$$
\end{proposition}
\begin{proof}
	We have seen in \cref{prop.int-ccat-functoriality} that $u$ induces a functor 
	$ \SSEG(\E) \rightarrow \SSEG(\F)$. Hence, it remains to check that post-composition with $u$ preserves the completeness condition for Segal objects and this follows directly.
\end{proof}

\subsection{The case of spaces}
We conclude by showing that the above construction recovers the double $\infty$-category $\infty\CCAT$  when 
 $\E$ is given by the $\infty$-topos of spaces.

\begin{proposition}\label{prop.equipment-comparison-spaces}
	There exists an equivalence 
	$$
	\infty\CCAT \xrightarrow{\simeq} \CCAT(\S)
	$$
	of double $\infty$-categories.
\end{proposition}

To this end, we will make use of the following fact:

\begin{lemma}\label{lemma.presheaves-slice}
Let $\C$ be an $\infty$-category, and let us write $y : \C \to \PSh(\C)$ for the Yoneda embedding.
For every $c \in \C$, left Kan extension along the projection $\C_{/c}\to \C$ gives rise 
to a natural equivalence
$$
\PSh(\C_{/c}) \to \PSh(\C)_{/y(c)}.
$$
\end{lemma}
\begin{proof}
	This is a direct consequence of \cite[Corollary 9.8]{GepnerHaugsengNikolaus}.
\end{proof}

\begin{proof}[Proof of \cref{prop.equipment-comparison-spaces}]
	We recall from \cref{ex.segal-spaces}, that there is a fully faithful functor 
	$\infty\Cat \to \PSh(\Delta)$. This functor induces a natural inclusion
	$$
	y_n : \infty\Cat_{/[n]} \to \PSh(\Delta)_{/[n]} \simeq \PSh(\Delta_{/[n]})
	$$
	for $[n] \in \Delta^\op$, where the second natural equivalence comes from \cref{lemma.presheaves-slice}.
	The essential image of the first functor in the composite is spanned by the complete Segal spaces over $[n]$. 
	One readily verifies that, under the second functor, these correspond precisely to the Segal $[n]$-modules $M$ for which the restriction 
	$(\{i\} \to [n])^*M$ is complete for each vertex $i \in [n]$.
	We now claim that $y_n$ carries the Conduch\'e fibrations on $[n]$ to Conduch\'e $[n]$-modules.  Using the naturality of 
	$y_n$, this may be reduced to checking for level $n=2$, where it follows directly from characterization (5) of \cite[Lemma 5.16]{AyalaFrancisRozenblyum}. 
	All in all, we conclude that $y$ restricts to a comparison functor
	$\infty\CCAT \to \CCAT(\S)$ between double $\infty$-categories. This must be an equivalence since it induces equivalences of $\infty$-categories at levels $n = 0, 1$.
\end{proof}

\section{\texorpdfstring{$\infty$}{∞}-Equipments}\label{section.equipments}

The next section is dedicated to defining a special class of double $\infty$-categories:
the
 $\infty$-equipments. We will show that the definition and equivalent characterizations of proarrow equipments 
in the strict context \cite[Definition 4.8]{ShulmanFramedBicats} 
may be adapted to the $\infty$-categorical context.
Subsequently, we will explain
that the double $\infty$-categories that were constructed in \cref{section.equipment-infty-cats} and \cref{section.equipment-internal-cats} 
are $\infty$-equip\-ments. 
We will conclude the section by showing that every $\infty$-equipment admits a co- and contravariant Yoneda embeddings.

\subsection{Companions and conjoints}\label{ssection.compconj} We recall the theory of companions and conjoints in double $\infty$-categories, as also defined in \cite{Comp}, 
and originally for double categories introduced by Grandis and Par\'e \cite{GrandisPare}.
These are the double categorical analogues of adjunctions in 2-category theory. Using this language, we will be able to speak about (co)representability of horizontal arrows, and this will be an important 
concept in the theory of $\infty$-equipments. 

\begin{definition}\label{def.compconj}
	Let $f : x \rightarrow y$ be a vertical arrow in a double $\infty$-category $\P$. A horizontal arrow 
	$F : x \rightarrow y$ in $\P$ is called the \textit{companion} of $f$  
	if there exist two 2-cells
	\[
		\eta = \begin{tikzcd}
			x\arrow[d, equal] \arrow[r, equal, ""name=f] & x\arrow[d, "f"] \\
			x \arrow[r, "F"'name=t] & y
			\arrow[from=f,to=t,Rightarrow, shorten <= 6pt, shorten >= 6pt]
		\end{tikzcd} 
		\quad 
		\text{and}
		\quad
		\epsilon = \begin{tikzcd}
			x \arrow[r, "F"name=f]\arrow[d, "f"' ] & y \arrow[d, equal] \\
			y \arrow[r, equal, ""'name=t] & y
			\arrow[from=f,to=t,Rightarrow, shorten <= 6pt, shorten >= 6pt]
		\end{tikzcd}
	\] 
	that satisfy the following two \textit{triangle identities}:
	\[
		\begin{tikzcd}
			x\arrow[d, equal] \arrow[r, equal, ""'name=h1] & x\arrow[d, "f"] \\
			x \arrow[r, ""name=h2]\arrow[d, "f"'] & y \arrow[d, equal] \\
			y \arrow[r, equal, ""'name=h3] & y
			\arrow[from=h1, to=h2, phantom, "\scriptstyle\eta"]
			\arrow[from=h2, to=h3, phantom, "\scriptstyle\epsilon"]
		\end{tikzcd}
		\simeq
		\begin{tikzcd}
			x\arrow[r,equal] \arrow[d, "f"'name=t1] & x \arrow[d, "f"name=f1] \\
			y \arrow[r, equal]  & y,
			\arrow[from=f1, to=t1, equal, shorten <= 14pt, shorten >= 14pt]
		\end{tikzcd}
		\quad 
		\begin{tikzcd}
			x \arrow[r, equal, ""name=h1]\arrow[d, equal] & x \arrow[d] \arrow[r, "F"name=h3] & y\arrow[d, equal] \\
			x \arrow[r, "F"'name=h2] & y \arrow[r, equal, ""'name=h4] & y
			\arrow[from=h1, to=h2, phantom, "\scriptstyle\eta"]
			\arrow[from=h3, to=h4, phantom, "\scriptstyle\epsilon"]
		\end{tikzcd}
		\simeq
		\begin{tikzcd}
			x\arrow[d,equal] \arrow[r, "F"name=f1] & y \arrow[d,equal] \\
			x \arrow[r, "F"'name=t1] & y.
			\arrow[from=f1, to=t1, equal, shorten <= 12pt, shorten >= 10pt]
		\end{tikzcd}
	\]
	In this case, $(f,F)$ is called a \textit{companionship}, $\eta$ is called the \textit{companionship unit}, and $\epsilon$ is called the 
	\textit{companionship counit}.

	Dually, a horizontal arrow $F':y\rightarrow x$ is called the \textit{conjoint} of $f$ 
	when there exist two 2-cells in $\P$
	\[
		\eta' = \begin{tikzcd}
			x\arrow[d, "f"' ] \arrow[r, equal, ""name=f] & x\arrow[d, equal] \\
			y \arrow[r, "F'"'name=t] & x
			\arrow[from=f,to=t,Rightarrow, shorten <= 6pt, shorten >= 6pt]
		\end{tikzcd} 
		\quad 
		\text{and}
		\quad
		\epsilon' = \begin{tikzcd}
			y \arrow[r, "F'"name=f]\arrow[d, equal] & x \arrow[d, "f"] \\
			y \arrow[r, equal, ""'name=t] & y
			\arrow[from=f,to=t,Rightarrow, shorten <= 6pt, shorten >= 6pt]
		\end{tikzcd}
	\] 
	that compose as follows:
	\[
		\begin{tikzcd}
			x\arrow[d, "f"'] \arrow[r, equal, ""'name=h1] & x\arrow[d, equal] \\
			y \arrow[d,equal ]\arrow[r, ""name=h2] & x \arrow[d, "f"] \\
			y \arrow[r, equal, ""'name=h3] & y
			\arrow[from=h1, to=h2, phantom, "\scriptstyle\eta'"]
			\arrow[from=h2, to=h3, phantom, "\scriptstyle\epsilon'"]
		\end{tikzcd}
		\simeq
		\begin{tikzcd}
			x\arrow[r,equal] \arrow[d, "f"'name=t1] & x \arrow[d, "f"name=f1] \\
			y \arrow[r, equal]  & y,
			\arrow[from=f1, to=t1, equal, shorten <= 14pt, shorten >= 14pt]
		\end{tikzcd}
		\quad 
		\begin{tikzcd}
			y \arrow[r, "F'"name=h1]\arrow[d, equal] & x \arrow[d] \arrow[r, equal, ""name=h3] & x\arrow[d, equal] \\
			y \arrow[r, equal, ""'name=h2] & y \arrow[r,"F'"'name=h4] & x
			\arrow[from=h1, to=h2, phantom, "\scriptstyle\epsilon'"]
			\arrow[from=h3, to=h4, phantom, "\scriptstyle\eta'"]
		\end{tikzcd}
		\simeq
		\begin{tikzcd}
			y\arrow[d,equal] \arrow[r, "F'"name=f1] & x \arrow[d,equal] \\
			y \arrow[r, "F'"'name=t1] & x.
			\arrow[from=f1, to=t1, equal, shorten <= 12pt, shorten >= 10pt]
		\end{tikzcd}
	\]
	In this case, $(f,F')$ is called a \textit{conjunction}, $\eta'$ is called the \textit{conjunction unit}, and $\epsilon'$ is called the 
	\textit{conjunction counit}.
\end{definition}

\begin{remark}
	It was shown in \cite{Comp} that the spaces of companionship units and conjunction units for a fixed vertical arrow are either contractible or empty. Therefore, it is justified 
	to speak of \textit{the} companion and \textit{the} conjoint of a vertical arrow.
\end{remark}

\begin{notation}\label{not.comp-conj}
	Suppose that $f:x\rightarrow y$ is a vertical arrow 
	of a double $\infty$-category $\P$. Then the companion and conjoint of  $f$ 
	are denoted by 
	$$
	f_\circledast : x \rightarrow y \quad \text{and} \quad f^\circledast : y \rightarrow x
	$$
	respectively (if they exist).
\end{notation}

\begin{proposition}\label{prop.conj-comp-adjunction}
	Suppose that $f : x \to y$ is a vertical arrow of a double $\infty$-category $\P$ that admits both a companion and a conjoint. Then the pair $(f_\circledast, f^\circledast)$ forms an adjunction 
	in $\Hor(\P)$ so that the associated unit and counit are given by the pasting
	\[
		\eta = \begin{tikzcd}
			x \arrow[r, equal, ""name=h1]\arrow[d, equal] & x \arrow[r, equal, ""name=h3] \arrow[d, "f"] & x \arrow[d,equal] \\
			x \arrow[r, "f_\circledast"name=h2] \arrow[r] & y \arrow[r, "f^\circledast"name=h4] & x 
			\arrow[from=h1,to=h2, Rightarrow, shorten <= 6pt]
			\arrow[from=h3,to=h4, Rightarrow, shorten <= 6pt]
		\end{tikzcd}
		\quad 
		\epsilon = \begin{tikzcd}
			y \arrow[r, "f^\circledast"name=f]\arrow[d, equal] & x \arrow[r, "f_\circledast"name=f2]\arrow[d, "f" ] & y \arrow[d, equal] \\
			y \arrow[r, equal,""'name=t] & y \arrow[r, equal,""'name=t2] & y.
			\arrow[from=f,to=t,Rightarrow, shorten <= 6pt, shorten >= 6pt]
			\arrow[from=f2,to=t2,Rightarrow, shorten <= 6pt, shorten >= 6pt]
		\end{tikzcd} 
	\]
	of the companionship and conjunction units and counits, respectively.
\end{proposition}
\begin{proof}
	This also appeared as \cite[Proposition 4.8]{Comp}, and it is an easy verification using the companionship and conjunction triangle identities.
\end{proof}

\begin{example}\label{ex.compconj-in-dbl-cat-spans}
	Let $\C$ be an $\infty$-category with finite limits. Then every vertical arrow 
	in $\SSPAN(\C)$ admits a companion and conjoint; see \cite[Example 4.6]{Comp}.
\end{example}

\begin{example}\label{ex.compconj-ccat}
	Let $f : \C \rightarrow \D$ be a functor between $\infty$-categories. Then we claim that it has a 
	conjoint in $\infty\CCAT$ that is given by the Grothendieck construction of $f : [1] \rightarrow \infty\Cat$. 
	To show this, one can consider the commutative diagram 
		\[
			\begin{tikzcd}[row sep = 1.8pt]
				\C \arrow[r,equal]\arrow[dd,equal] & \C \arrow[dd, "f"] \arrow[r, "f"]& \D \arrow[dd,equal] && E \arrow[dd, "g"'] \arrow[dr] \\
				&&& \Leftrightarrow & & {[2]} \\
				\C \arrow[r, "f"] & \D \arrow[r,equal] & \D & & E' \arrow[ur]
			\end{tikzcd}
		\]
		of $\infty$-categories on the left. We may view 
		this diagram as an arrow of $\fun([2], \infty\Cat)$, so that 
		this diagram corresponds to a map $g$ of cocartesian fibrations over $[2]$ displayed above on the right. 
	 	We claim that the restriction $F := E \times_{[2]} \{1 \leq 2\} \simeq E' \times_{[2]}  \{0 \leq 1\}$ is the conjoint of $f$ in $\infty\CCAT$ witnessed 
		by the 2-cells 
		\[
			\begin{tikzcd}
			\C \arrow[r,equal,""name=f] \arrow[d, "f"'] & \C  \arrow[d,equal] \\
			\D \arrow[r, "F"'name=t] & \C,
			\arrow[from = f, to = t, phantom, "\scriptstyle{g \times_{[2]}\{0 \leq 1\}}"]
			\end{tikzcd}
			\quad \quad
			\begin{tikzcd}
				\D \arrow[r,"F"name=f] \arrow[d, equal] & \C  \arrow[d,"f"] \\
				\D \arrow[r, equal, ""'name=t] & \D,
				\arrow[from = f, to = t, phantom, "\scriptstyle{g \times_{[2]}\{1 \leq 2\}}"]
			\end{tikzcd}
		\]
		of $\infty\CCAT$.
		The projections $E, E'\to [2]$ are in particular Conduch\'e fibrations (see \cite[Lemma 2.15]{AyalaFrancis} for a proof) so that we have the following computation 
		of the horizontal composition in $\infty\CCAT$ of the candidate counit and unit:
		\[
			\begin{tikzcd}
				\D \arrow[r,"F"name=f1]\arrow[d,equal] & \arrow[r,equal,""name=f2] \C\arrow[d] & \C \arrow[d,equal] \\ 
				\D \arrow[r,equal,""'name=t1] & \D \arrow[r, "F"'name=t2] & \C
				\arrow[from = f1, to = t1, phantom, "\scriptstyle{g \times_{[2]}\{1 \leq 2\}}"]
				\arrow[from = f2, to = t2, phantom, "\scriptstyle{g \times_{[2]}\{0 \leq 1\}}"]
			\end{tikzcd}
			\simeq 
			\begin{tikzcd}
				\D \arrow[r,"F"name=f] \arrow[d, equal] & \C  \arrow[d,equal] \\
				\D \arrow[r, "F"'name=t] & \C.
				\arrow[from = f, to = t, phantom, "\scriptstyle{g \times_{[2]}\{0 \leq 2\}}"]
			\end{tikzcd}
		\]
		By naturality of the Grothendieck construction, $g \times_{[2]} \{0\leq 2\}$ corresponds to the identity morphism $\id_f : f \to f$ of $\fun([1], \infty\Cat)$. 
		Hence, we deduce that $g \times_{[2]} \{0\leq 2\} \simeq \id_F$, as desired.
		The vertical pasting in $\infty\CCAT$ displayed below on the left
		\[
			\begin{tikzcd}
				\C \arrow[r,equal,""name=h1]\arrow[d,"f"'] & \C\arrow[d,equal] && E \times_{[2]} \{0 \leq 1\} \arrow[d] \arrow[dr] &\\
				\D \arrow[r,""'name=h2]\arrow[r,phantom, ""name=h3]\arrow[d,equal] & \C\arrow[d,"f"] & \Leftrightarrow & E' \times_{[2]} \{0 \leq 1\}  \simeq  E \times_{[2]} \{1 \leq 2\} \arrow[d] \arrow[r] & {[1]} \\ 
				\D \arrow[r,equal,""'name=h4] & \D && E \times_{[2]} \{1 \leq 2\} \arrow[ur] &   
				\arrow[from = h3, to = h4, phantom, "\scriptstyle{g \times_{[2]}\{1 \leq 2\}}"]
				\arrow[from = h1, to = h2, phantom, "\scriptstyle{g \times_{[2]}\{0 \leq 1\}}"]
			\end{tikzcd}
		\]
		is classified by the composite in $\infty\Cat_{/[1]}$ on the right. By construction, this must recover the Grothendieck construction 
		of the composite $\id_c \to f \to \id_d$ in $\fun([1], \infty\Cat)$. Thus the left vertical composite is equivalent to 
		the identity 2-cell $\id_f$ on $f$.
		
		A similar argument shows that $f$ admits a companion in $\infty\CCAT$ that classifies the functor $[1]^\op \to \infty\Cat$ corresponding to $f$.
\end{example}

We recall the following from \cite{Comp}:

\begin{definition}
	The \textit{free-living companionship} and \textit{free-living conjunction} double $\infty$-categories are defined by $$\comp := \Sq([1]), \quad \conj := \comp^\hop,$$
	respectively (see \cref{ex.squares}).
\end{definition}

A quick inspection shows that $\comp$ contains a single non-trivial companionship, and $\conj$ contains a single non-trivial conjunction.
It was demonstrated in \cite{Comp} that these are universal examples:

\begin{theorem}\label{thm.free-living-comp}
	Let $\P$ be a double $\infty$-category.
	Then the canonical maps $[1]_v \to \comp$ and $[1]_v \to \conj$ induce monomorphisms
	\begin{gather*}
	\map_{\DbliCat}(\comp, \P) \to \map_{\DbliCat}([1]_v, \P), \\ \map_{\DbliCat}(\conj, \P) \to \map_{\DbliCat}([1]_v, \P),
	\end{gather*}
	whose images are given by the vertical arrows that admit a companion, and the vertical arrows that admit a conjoint, respectively.
	
	Moreover, if $\P$ is locally complete, then the canonical maps $[1]_h \to \comp$ and $[1]_h \to \conj$ induce monomorphisms
	\begin{gather*}
	\map_{\DbliCat}(\comp, \P) \to \map_{\DbliCat}([1]_h, \P), \\ \map_{\DbliCat}(\conj, \P) \to \map_{\DbliCat}([1]_h, \P),
	\end{gather*}
	whose images are given by the horizontal arrows that are a companion, and the horizontal arrows that are a conjoint, respectively.
\end{theorem}
\begin{proof}
	This is precisely the combined content of Theorem 4.13, Corollary 4.15, and Corollary 4.16 of \cite{Comp}.
\end{proof}

\subsection{Cartesian and cocartesian 2-cells} The notion of (co)cartesian 2-cells in strict double categories \cite[Section 4]{ShulmanFramedBicats} readily extends to double $\infty$-categories as follows.

\begin{definition}\label{def.cocart-cells}
	Let $\P$ be a double $\infty$-category.
	A 2-cell $\alpha$ in $\P$ is called \textit{(co)cartesian} if it corresponds to a (co)cartesian arrow 
	of the source-target projection 
	$$
	(d_1^*,d_0^*) : \P_1 \rightarrow \P_{0}^{\times 2}.
	$$
\end{definition}

\begin{remark}\label{rem.cart-cocart-duality}
	Note that a 2-cell $\alpha$ in $\P$ is cocartesian if and only if it the associated 2-cell $\alpha^\vop$ is cartesian in $\P^\vop$.
\end{remark}

\begin{example}\label{ex.ccat-cart-lifts}
	Suppose that $\P = \infty\CCAT$. Then the source-target projection of $\P$ is a cartesian fibration on account of \cref{prop.equiv-cor-bifib-prof}. 
	Let $F : \C \rightarrow \D$ be a horizontal arrow of $\P$. Suppose that $f : X \rightarrow \C$ and $g : Y \rightarrow \D$ are functors. 
	Then the cartesian restriction $(f,g)^*F$ of $F$ along $(f,g)$ with respect to $\P_1 \rightarrow \P_0^{\times 2}$ is described as follows:
	\begin{enumerate}
		\item If $F$ corresponds to a two-sided discrete fibration $E \rightarrow \C \times \D$, then $(f,g)^*F$ is classified by the base change  
		$$
		E \times_{(\C \times \D)} (X\times Y) \rightarrow X\times Y.
		$$
		\item If $F$ corresponds to a profunctor $P : \D^\op \times \C \rightarrow \S$, then $(f,g)^*F$ is classified by the restricted profunctor 
		$$
		Y^\op \times X \xrightarrow{g^\op \times f} \D^\op \times \C \xrightarrow{P} \S.
		$$
	\end{enumerate}
	To verify (1), one uses that the inclusion $\DFib \subset \fun([1] \cup_{\{0\}} [1], \infty\Cat)$ of two-sided discrete fibrations 
	into the $\infty$-category of spans of $\infty$-categories is fully faithful, and two-sided discrete fibrations are closed under base change. 
		Observation (2) follows directly from the fact that $\prof \rightarrow \infty\Cat^{\times 2}$ classifies the functor $(\C, \D) \mapsto \fun(\D^\op \times \C, \S)$.
\end{example}

There is also another perspective on these (co)cartesian 2-cells that is often useful in practice. To this end, we make some preparatory definitions.

\begin{notation}
	We will make use of the sub double category $N\subset [1,1]$, called the \textit{free-living niche}, that may be pictured as 
	\[ 
		 N= \begin{tikzcd}[column sep = small, row sep =small]
			(0,0) \arrow[d] & (1,0)\arrow[d] \\
			(0,1) \arrow[r] & (1,1),
		\end{tikzcd}
	\] and fits in the pushout square
	\[ 
	\begin{tikzcd}[column sep = large]
		\{0,1\}_h \times [0]_v \arrow[d, "{\{0,1\}_h \times d_0}"']\arrow[r]& {[1,0]}\arrow[d] \\
		\{0,1\}_h \times [1]_v \arrow[r] & N
	\end{tikzcd}
	\]
	of double $\infty$-categories. 
	We define also make use of the double category $\vrectangle$ that is defined by the pushout square 
	\[ 
	\begin{tikzcd}[column sep = large]
		\{0,1\}_h \times [1]_v \arrow[d, "{\{0,1\}_h \times d_1}"']\arrow[r] & {[1,1]}\arrow[d] \\
		\{0,1\}_h \times [2]_v \arrow[r] & \vrectangle
	\end{tikzcd}
	\]
	of double $\infty$-categories.  In what follows, the commutative diagram
	\[
		\begin{tikzcd}
			N \arrow[d] \arrow[r] & {[1,1]} \arrow[d, "{[1,d_0]}"] \\
			\vrectangle \arrow[r] & {[1,2]}
		\end{tikzcd}
	\]
	of inclusions will play a role.
\end{notation}

\begin{lemma}\label{lem.cart-niche-fillers}
	Let \[
	\alpha  = \begin{tikzcd} 
		a \arrow[r,"F"name=f]\arrow[d,"f"'] & b \arrow[d, "g"] \\ 
		c \arrow[r, "G"name=t] & d
		\arrow[from=f,to=t, Rightarrow, shorten <= 6pt]
	\end{tikzcd}
	\] be a 2-cell of a double $\infty$-category $\P$. Then the following are equivalent: 
	\begin{enumerate}[noitemsep]
		\item the top map in the commutative square
		\[
		\begin{tikzcd}
			\map_{\DbliCat}([1,2], \P) \arrow[r]\arrow[d] & \map_{\DbliCat}(\vrectangle, \P) \arrow[d] \\
			\map_{\DbliCat}([1,1], \P) \arrow[r] & \map_{\DbliCat}(N, \P)
		\end{tikzcd}
		\]
		of spaces
		induces an equivalence when restricting to the fibers above $\alpha$ and $\alpha|N$,
		\item $\alpha$ is cartesian.
	\end{enumerate}
\end{lemma}
\begin{proof}
	By definition, $\alpha$ classifies a cartesian edge of the source-target projection if and only if the induced square 
	\[
		\begin{tikzcd}
			\P_{1/\alpha} \arrow[r]\arrow[d] & \P_{1/F} \arrow[d] \\
			\P_{0/f} \times \P_{0/g} \arrow[r] & \P_{0/a} \times \P_{0/b} 
		\end{tikzcd}
	\]
	is a pullback square of $\infty$-categories. But note that the comparison map 
	$$
	\P_{1/\alpha}\rightarrow \P_{1/F} \times_{(\P_{0/a} \times \P_{0/b} )} (\P_{0/f} \times \P_{0/g} )
	$$
	is a map between right fibrations over $\P_1$. Hence, it suffices to check that the above square is a pullback square on 
	underlying spaces, i.e.\ after applying $\map_{\infty\Cat}([0], -)$. 
	By definition of slice $\infty$-categories, this square on underlying spaces may be obtained from the commutative diagram
	\[
		\begin{tikzcd}
			\map_{\DbliCat}([1]_h \times [2]_v, \P) \arrow[r]\arrow[d] &\map_{\DbliCat}([1]_h \times [1]_v, \P)  \arrow[d] \\
			\map_{\DbliCat}(\{0,1\}_h, \times [2]_v, \P) \arrow[r] & \map_{\DbliCat}(\{0,1\}_h, \times [1]_v, \P)
		\end{tikzcd}
	\]
	by taking appropriate fibers in each corner above the restrictions of $\alpha$. This now readily translates to condition (1).
\end{proof}

\begin{remark}\label{rem.cart-fact}
	Let $\P$ be a double $\infty$-category.
	Suppose that $\alpha$ is a 2-cell as in \cref{lem.cart-niche-fillers}, and suppose that $\alpha$ is cartesian.
	In the picture below, if we start with a 2-cell as on the left, then there exists a unique factorization through $\alpha$ as on the right: 
	\[
	\begin{tikzcd} 
		x \arrow[d,"h"'] \arrow[r,"H"name=h1] & y \arrow[d,"k"] \\
		a \arrow[d,"f"'] & b \arrow[d, "g"] \\ 
		c \arrow[r, "G"name=h2] & d
		\arrow[from=h1,to=h2,Rightarrow, shorten <= 6pt]
	\end{tikzcd}
	\simeq 
	\begin{tikzcd} 
		x \arrow[d,"h"'] \arrow[r,"H"name=h1] & y \arrow[d,"k"] \\
		a \arrow[r,"F"name=h2]\arrow[d,"f"'] & b \arrow[d, "g"] \\ 
		c \arrow[r, "G"'name=h3] & d.
		\arrow[from=h1,to=h2, Rightarrow, dashed, shorten <= 6pt]
		\arrow[from=h2,to=h3,phantom,"\scriptstyle \alpha"]
	\end{tikzcd}
	\]
	Dually, if $\alpha$ is \textit{cocartesian}, then we have unique factorizations as pictured below: 
	\[
	\begin{tikzcd} 
		a \arrow[d,"f"'] \arrow[r,"H"name=h1] & b \arrow[d,"g"] \\
		c \arrow[d,"f"'] & d \arrow[d, "g"] \\ 
		x \arrow[r, "G"name=h2] & y
		\arrow[from=h1,to=h2,Rightarrow, shorten <= 6pt]
	\end{tikzcd}
	\simeq 
	\begin{tikzcd} 
		a \arrow[d,"f"'] \arrow[r,"F"name=h1] & b \arrow[d,"g"] \\
		c \arrow[r,"G"'name=h2]\arrow[d,"h"'] & d \arrow[d, "k"] \\ 
		x \arrow[r, "H"'name=h3] & y.
		\arrow[from=h2,to=h3, Rightarrow, dashed, shorten >= 5pt]
		\arrow[from=h1,to=h2,phantom,"\scriptstyle \alpha"]
	\end{tikzcd}
	\]
\end{remark}

The notions of (co)cartesian 2-cells, companions, and conjoints are closely related:

\begin{theorem}\label{th.comp-conj-cart}
	Let $f : a \rightarrow x$ and $g : b \rightarrow y$ be arrows of a double $\infty$-category $\P$. Suppose that 
	$f$ admits a companion and $g$ admits a conjoint. If $F : x\rightarrow y$ is a horizontal arrow 
	of $\P$, then the pasted 2-cell
	\[
		\begin{tikzcd}
			a \arrow[r, "{f_\circledast}"name=f1]\arrow[d,"f"'] & x \arrow[r,	"F"name=f2]\arrow[d,equal] & y \arrow[r, "{g^\circledast}"name=f3] \arrow[d, equal] & b\arrow[d, "g"] \\
			x \arrow[r, equal, ""name=t1] & x \arrow[r, "F"name=t2] & y \arrow[r,equal, ""name=t3] & y,
			\arrow[from=f1, to=t1, Rightarrow, shorten <= 6pt]
			\arrow[from=f2, to=t2, equal, shorten <= 6pt]
			\arrow[from=f3, to=t3, Rightarrow, shorten <= 6pt]
		\end{tikzcd}
	\]
	is a cartesian 2-cell. Here, the left and right 2-cells are the companionship and conjunction counits respectively.
	Dually, if $G : b\rightarrow a$ is a horizontal arrow of $\P$, then the pasting
	\[
		\begin{tikzcd}
			b \arrow[r, equal, ""name=f1]\arrow[d,"g"'] & b\arrow[r,	"G"name=f2]\arrow[d,equal] & a \arrow[r, equal, ""name=f3] \arrow[d, equal] & a\arrow[d, "f"] \\
			y \arrow[r, "g^\circledast"name=t1] & b \arrow[r, "G"name=t2] & a \arrow[r,"f_\circledast"name=t3] & x,
			\arrow[from=f1, to=t1, Rightarrow, shorten <= 6pt]
			\arrow[from=f2, to=t2, equal, shorten <= 6pt]
			\arrow[from=f3, to=t3, Rightarrow, shorten <= 6pt]
		\end{tikzcd}
	\]
	is a cocartesian 2-cell. Here, the left and right 2-cells are now the conjunction and companionship units.
\end{theorem}

We will 
postpone the proof for now. Let us first discuss the following implication:

\begin{corollary}\label{cor.compconj-roof-niches}
	Suppose that $\P$ is a double $\infty$-category. 
	Let $f : x \rightarrow y$ be a vertical arrow in $\P$. Then a 2-cell of the form
	\[
	\alpha = \begin{tikzcd}
		x\arrow[d,"f"']\arrow[r,"F"name=f] & y \arrow[d,equal] \\ 
		y \arrow[r,equal,""name=t] & y,
		\arrow[from=f,to=t,Rightarrow, shorten <= 6pt]
	\end{tikzcd} 
	\]
	is cartesian if and only if it is a companionship counit. Dually, a 2-cell of the form 
	\[
		\beta = \begin{tikzcd}
			y\arrow[d,equal]\arrow[r,"F'"name=f] & x \arrow[d, "f"] \\ 
			y \arrow[r,equal,""name=t] & y,
			\arrow[from=f,to=t,Rightarrow, shorten <= 6pt]
		\end{tikzcd} 
	\]
	is cartesian if and only if it is a conjunction counit. 
\end{corollary}
\begin{proof}
	We will just handle the assertion for $\alpha$; the second assertion is shown similarly. If $\alpha$ is a companionship counit, then \cref{th.comp-conj-cart} implies that $\alpha$ is cartesian. 
	Conversely, suppose that $\alpha$ is cartesian. Then one can obtain the candidate companionship unit $\eta$ by factoring the vertical identity 2-cell of $f$ through $\alpha$ (see \cref{rem.cart-fact}).
	It remains to verify that the horizontal composite $\gamma$ of the 2-cells $\eta$ and $\alpha$ is the horizontal identity on $F$. But, by the universal property of cartesian 2-cells, 
	it suffices to check that the vertical composite of $\gamma$ and $\alpha$ recovers $\alpha$, which is true by construction.
\end{proof}

\begin{example}\label{ex.compconj-ccat-prof-tdfib}
	We may compute companions in $\infty\CCAT$ from the perspective 
	of two-sided discrete fibrations and profunctors using \cref{cor.compconj-roof-niches} as follows. Let $f : \C \rightarrow \D$ be a functor. 
	In light of \cref{ex.ccat-hor-identities} and  \cref{ex.ccat-cart-lifts}, we conclude that the companion of $f$ is classified by the profunctor 
	$$
	\map_\D(-, f(-)) : \D^\op \times \C \xrightarrow{\id \times f} \D^\op \times \D \xrightarrow{\map_\D(-,-)}\S,
	$$
	and classified by the two-sided discrete fibration 
	$$
		\C \times_{\fun(\{1\}, \D)} \fun([1], \D)  \rightarrow \C \times \D.
	$$
	There are similar formulas for conjoints.
\end{example}

\begin{proof}[Proof of \cref{th.comp-conj-cart}]
	The demonstration is very similar to the proof of \cite[Theorem 4.22]{Comp}. 
	It suffices to show the first assertion since the second assertion can be obtained by applying 
	$(-)^{\mathrm{vop}}$; see also \cref{rem.cart-cocart-duality}. We will write $\alpha$ for the pasted
	2-cell that appears in the first assertion throughout the proof. 
	
	We will make use of the \textit{thickened niche} $N'$ that is defined by the pushout square
	\[
		\begin{tikzcd}
			{[1]_v} \sqcup {[1]_v} \arrow[d]\arrow[r] & \comp \sqcup \conj \arrow[d] \\
			N \arrow[r] & N'
		\end{tikzcd}
	\]
	of double $\infty$-categories. The double $\infty$-category $\P$ receives maps from $\comp$ and $\conj$ 
	that extend the maps that select the vertical arrows $f$ and $g$ respectively; this follows from \cref{thm.free-living-comp}. Together with the restriction $\alpha|N$, 
	this provides a canonical map 
	$
	\alpha' : N' \rightarrow \P.
	$
	By construction, the 2-cell $\alpha$ can be written as a composite 
	\[
		[1,1] \xrightarrow{[\{0\leq 3\},1]} [3,1] \xrightarrow{\beta} N' \rightarrow \P,
	\]
	where the map $\beta$ selects the following compatible 2-cells in $N'$: the image of the free companionship counit 
	under $\comp \rightarrow N'$, the identity 2-cell on the non-trivial horizontal arrow in the image of $N \rightarrow N'$, and 
	the image of the free conjunction counit under $\conj \rightarrow N'$.  
	Let us consider the map 
	$$
	i : \vrectangle' := N' \cup_{N} \vrectangle \rightarrow N' \cup_{[1,1]} [1,2] =: [1,2]'
	$$
	of double $\infty$-categories; the pushouts are computed in $\DbliCat$. 
	Then we obtain a commutative square 
	\[
	\begin{tikzcd}
		\map([1,2]', \P) \times_{\map(N',\P)} \{\alpha'\} \arrow[d]\arrow[r] & \map([1,2], \P) \times_{\map([1,1], \P)} \{\alpha\} \arrow[d] \\ 
		\map(\vrectangle, \P) \times_{\map(N',\P)} \{\alpha'\} \arrow[r] & \map(\vrectangle,\P)\times_{\map(N,\P)} \{\alpha|N\},
	\end{tikzcd}
	\]
	and the horizontal maps are equivalences on account of the pasting lemma for pullback squares.
	It thus suffices to show that $i$
	is an equivalence.

	We will exhibit an explicit inverse to $i$, analogous as in the demonstration of \cite[Theorem 4.22]{Comp}.
	 We will make use of the double $\infty$-category  $P$ that may be pictured as
	$$
	P := \begin{tikzcd}[column sep = tiny, row sep = tiny]
		(0,0) \arrow[r,""name=a1]\arrow[d] & (1,0) \arrow[r,""name=b1]\arrow[d] & (2,0) \arrow[r,""name=c1] \arrow[d] & (3,0) \arrow[d] \\
		(0,1)\arrow[d]  \arrow[r,""name=a2] & (1,1)\arrow[d]  & (2,1)\arrow[d]  \arrow[r, ""name=c2] & (3,1)\arrow[d]  \\
		(0,2) \arrow[r,""name=a3]\arrow[d] & (1,2) \arrow[r,""name=b2]\arrow[d] & (2,2) \arrow[r,""name=c3] \arrow[d] & (3,2) \arrow[d] \\
		(0,3) \arrow[r,""name=a4] & (1,3) \arrow[r, ""name=b3] & (2,3) \arrow[r,""name=c4]  & (3,3)
		\arrow[from=a1,to=a2, Rightarrow, shorten <= 5pt]
		\arrow[from=a2,to=a3, Rightarrow, shorten <= 5pt]
		\arrow[from=a3,to=a4, Rightarrow, shorten <= 5pt]
		\arrow[from=b1,to=b2, Rightarrow, shorten <= 5pt]
		\arrow[from=b2,to=b3, Rightarrow, shorten <= 5pt]
		\arrow[from=c1,to=c2, Rightarrow, shorten <= 5pt]
		\arrow[from=c2,to=c3, Rightarrow, shorten <= 5pt]
		\arrow[from=c3,to=c4, Rightarrow, shorten <= 5pt]
	\end{tikzcd}
	\subset [3,3].
	$$ 
	Here, all squares are filled with 2-cells. In the terminology of \cite{Pasting}, note that $P$ can be obtained as the nerve of the similarly described \textit{composable 2-dimensional pasting shape}.
	As a double $\infty$-category, $P$ can be written as the colimit of all the 2-cells, with possibly subdivided boundaries, that appear in $P$ and that cannot be further subdivided into other non-trivial 2-cells. This follows from 
	the pasting theorem \cite[Theorem A]{Pasting}. One may use this colimit description to construct a map $p : P \rightarrow \vrectangle'$ that corresponds to the diagram
	\[
		\begin{tikzcd}[column sep = tiny, row sep = tiny]
			(0,0) \arrow[r,equal]\arrow[d,""name=f1] & (0,0) \arrow[r,""name=b1]\arrow[d, ""'name=t1] & (1,0) \arrow[r, equal] \arrow[d, ""name=f2] & (1,0) \arrow[d, ""'name=t2] \\
			(0,1)\arrow[d,equal]  \arrow[r,equal,""name=a2] & (0,1)\arrow[d]  & (1,1)\arrow[d]  \arrow[r, equal,""name=c2] & (1,1)\arrow[d,equal]  \\
			(0,1) \arrow[r, ""name=a3]\arrow[d] & (0,2) \arrow[r, ""name=f3]\arrow[d,equal] & (1,2) \arrow[r,""name=c3] \arrow[d, equal] & (1,1) \arrow[d] \\
			(0,2) \arrow[r,equal, ""name=a4] & (0,2) \arrow[r, ""'name=t3] & (1,2) \arrow[r,equal,""name=c4]  & (1,2),
			\arrow[from=f1,to=t1,equal, shorten <=12pt, shorten >= 12pt]
			\arrow[from=f2,to=t2,equal, shorten <=12pt, shorten >= 12pt]
			\arrow[from=f3,to=t3,equal, shorten <=9pt, shorten >= 9pt]
			\arrow[from=a2,to=a3, Rightarrow, shorten <= 5pt]
			\arrow[from=a3,to=a4, Rightarrow, shorten <= 5pt]
			\arrow[from=c2,to=c3, Rightarrow, shorten <= 5pt]
			\arrow[from=c3,to=c4, Rightarrow, shorten <= 5pt]
			\arrow[from=b1,to=f3, Rightarrow, shorten <= 5pt]
		\end{tikzcd}
	\]
	so that the non-identity 2-cells that appear on the bottom left (resp.\ right) are the companionship (resp.\ conjunction) unit and counit. The non-identity 2-cell that appears in the middle is given by (the image of) $\vrectangle$.  
	Note that $P$ admits an inclusion 
	$[1,2] \rightarrow P$ that factors the inclusion $t : [1,1] \rightarrow P$ of the outer 2-cell. 
	Restricting $p$ along this inclusion, we obtain a functor $$q  : [1,2] \rightarrow \vrectangle'$$
	with the following properties:
	\begin{itemize}
		\item the restriction $q|[1,1]$ is given by the composite $[1,1] \xrightarrow{[\{0\leq 3\},1]} [3,1] \xrightarrow{\beta} N'$,
		\item the restriction $q|\vrectangle$ is equivalent to the canonical inclusion $\vrectangle \rightarrow \vrectangle'$,
		\item the composite $iq$ is equivalent to the canonical inclusion $[1,2] \rightarrow [1,2]'$.
	\end{itemize}
	These properties may be verified using the companionship and conjunction identities (cf.\ 
	the proof of \cite[Theorem 4.1]{ShulmanFramedBicats}) similarly to the verification in \cite[Theorem 4.22]{Comp}.
	Consequently, the map $q$ gives rise to a functor
	$$
	r : [1,2]' \rightarrow \vrectangle'
	$$
	via the universal property of the pushout which is an inverse to $i$. 
\end{proof}

\subsection{Characterizations of \texorpdfstring{$\infty$}{∞}-equipments} We are now ready to give the promised definition of our $\infty$-cate\-go\-rical proarrow equipments. To this end, 
we highlight the following consequence (cf.\ \cite[Theorem 4.1]{ShulmanFramedBicats}):

\begin{corollary}\label{cor.chars-equipments}
	Let $\P$ be a double $\infty$-category. Then the following assertions are equivalent: 
	\begin{enumerate}[noitemsep]
		\item every vertical arrow in $\P$ admits both a conjoint and a companion,
		\item the source-target projection $\P_1 \rightarrow \P_0^{\times 2}$ is a cartesian fibration,
		\item the source-target projection $\P_1 \rightarrow \P_0^{\times 2}$ is a cocartesian fibration.
	\end{enumerate}
	\end{corollary}
\begin{proof}
	One may use \cref{th.comp-conj-cart} to go from (1) to (2) and (3). On account of \cref{cor.compconj-roof-niches}, (2) in turn implies (1). Note that there is a dual form of \cref{cor.compconj-roof-niches} involving cocartesian 2-cells, and companionship and conjunction units, 
	which shows that (3) implies (1).
\end{proof}
	
	\begin{definition}\label{def.equipment}
		An \textit{$\infty$-equipment} is a locally complete double $\infty$-category that meets 
		the equivalent conditions of \cref{cor.chars-equipments}.
		We will write $\infty\Equip \subset \DbliCat$ for the full subcategory spanned by the $\infty$-equipments.
	\end{definition}
	
	\begin{proposition}\label{prop.dbl-fun-pres-cart}
		A functor between $\infty$-equipments preserves (co)car\-tesian cells.
	\end{proposition}
	\begin{proof}
		This follows directly from the description of (co)cartesian cells of \cref{th.comp-conj-cart}, and the observation 
		that functors between double $\infty$-catego\-ries preserve companionships and conjunctions.
	\end{proof}
	
	\begin{proposition}\label{prop.equips-reflective}
		The subcategory $\infty\Equip \subset \DbliCat$ is a reflective subcategory. In particular, 
		$\infty$-equipments are closed under limits.
	\end{proposition}
	\begin{proof}
		The full subcategory $\DbliCat^{lc}$ of $\DbliCat$ spanned by the locally complete ones is reflective; see \cite[Section 3.4]{Comp}. 
		Hence, it suffices to show that 
		$\infty\Equip$ is a reflective subcategory of $\DbliCat^{lc}$. In turn, this follows from \cref{thm.free-living-comp}.
		Namely, the
		$\infty$-equipments are precisely the locally complete double $\infty$-categories that are local with respect to the inclusions 
		$[1]_v \rightarrow \comp$ and $[1]_v \rightarrow \conj$.
	\end{proof}

	It follows from \cref{prop.int-ccat-loc-comp} and \cref{ex.compconj-ccat} that $\infty\CCAT$ is an $\infty$-equipment. More generally, the following is true:
	
	\begin{proposition}\label{prop.int-ccat-equip}
		Let $\E$ be a complete $\infty$-category with universal geometric realizations. Then the  double $\infty$-category $\CCAT(\E)$ 
		is an $\infty$-equipment.
	\end{proposition}
	\begin{proof}
		It follows from \cref{prop.int-ccat-loc-comp} that $\CCAT(\E)$ is locally complete. 
		We will show that its source-target projection of $\CCAT(\E)$ is a cartesian fibration. Note that we have a pullback square 
		\[
			\begin{tikzcd}
				\CCAT(\E)_1 \arrow[r]\arrow[d] & {\Delta_{/[1]}}\Seg(\E) \arrow[d] \\
				\Cat(\E)^{\times 2} \arrow[r] & \Seg(\E)^{\times 2},
			\end{tikzcd}
		\]
		where the vertical arrows are the source-target projections of the double $\infty$-categories $\CCAT(\E)$ and $\SSEG(\E)$.  We will show that the right arrow is a cartesian fibration. 
		To this end, we note that the full subcategory $\Seg(\E) \subset \fun(\Delta^\op, \E)$ is closed under all limits.
		 Hence, it suffices to show that the functor
		$\Delta_{/[1]}\Seg(\E) \rightarrow \Seg(\E)^{\times 2}$ admits a fully faithful right adjoint.\footnote{It is a general fact that any functor $p : E \rightarrow B$ with a fully faithful right adjoint and a target $B$ with all pullbacks is automatically a cartesian fibration, see \cite[Lemma 4.4.6]{GepnerHaugseng}.}
		By construction, this functor arises as the restriction of the functor 
		$$
		(d_0^*, d_1^*) : \fun((\Delta_{/[1]})^\op, \E) \rightarrow \fun(\Delta^\op, \E)^{\times 2},
		$$
	which admits a fully faithful right adjoint $(-) \star (-)$ given by right Kan extension along $ \Delta^\op \sqcup \Delta^\op \rightarrow (\Delta_{/[1]})^\op$, so that 
	$$
	(X \star Y)(f) = X(f^{-1}(1)) \times Y(f^{-1}(0))
	$$
	for $f : [n] \rightarrow [1]$ and $X, Y \in \fun(\Delta^\op, \E)$.
	One may readily verify that $X \star Y$ is a Segal $[1]$-module if $X$ and $Y$ are Segal objects in $\E$. Hence, $(-) \star (-)$ restricts 
	to the desired fully faithful right adjoint.
	\end{proof}
	
	\begin{example}
		The double $\infty$-categories of spans (see \cref{example.dbl-cat-spans}) are examples of $\infty$-equipments.
		This can be deduced from \cref{ex.compconj-in-dbl-cat-spans}. 
		Although we do not demonstrate this here,
		one may also show that the Morita double $\infty$-categories of \cref{example.dbl-cat-morita} are $\infty$-equipments.
	\end{example}

	\begin{convention}\label{conv.terminology-arrows}
		Throughout this article, we will use the following terminology for $\infty$-equipments:
		\begin{itemize}
		\item the vertical arrows of an $\infty$-equipment will be called \textit{arrows},
		\item the horizontal arrows of an $\infty$-equipment will be called \textit{proarrows}.
		\end{itemize} 
	\end{convention}
	
	\begin{remark}\label{rem.proarrow-interchange}
		Let $\P$ be an $\infty$-equipment. Suppose that we have arrows $f : a \rightarrow c$, $g : b \rightarrow d$,
		and proarrows $F : a \rightarrow b$, $G : c \rightarrow d$ in $\P$. Then it follows from \cref{prop.adjoints-mapping-cats} and \cref{prop.conj-comp-adjunction} that we have an equivalence 
		$$
		\map_{\Hor(\P)(c,d)}(g_\circledast F f^\circledast, G) \simeq \map_{\Hor(\P)(a,b)}(F, g^\circledast Gf_\circledast)
		$$
		that is natural in $F$ and $G$. This interchanging property is very useful and prominent in the proarrow calculus that we will see in \cref{section.formal-category-theory}.
	\end{remark}

\subsection{Abstract Yoneda embeddings}

To develop a formal category theory inside an $\infty$-equipment in \cref{section.formal-category-theory}, we will need an abstract form of the Yoneda embedding internal to an $\infty$-equipment. The way 
we would like to proceed is by using \textit{the universal property of the squares construction}.

Let $\C$ be an $(\infty,2)$-category. One readily verifies that its squares construction $\Sq(\C)$ 
has all companions. In fact, every horizontal arrow of $\Sq(\C)$ is a companion of a vertical arrow. 
The universal property of squares asserts that $\Sq(\C)$ is the smallest double $\infty$-category that is obtained from $\C_v$ by adjoining 
companions to the vertical arrows of $\C_v$. There is also dual variant, where $\C_v$ is replaced by the horizontal inclusion $\C_h$ of $\C$. 
This was made precise in \cite[Theorem 3.14]{GRconj} and \cite[Corollary 3.11]{GRconj}, where \cref{thm.free-living-comp} was generalized as follows:

\begin{theorem}[Universal property of squares \cite{GRconj}]\label{thm.uni-prop-sq}
	Let $\P$ be a double $\infty$-category. Then the canonical map $\C_v \to \Sq(\C)$ induces  a monomorphism 
	$$
	\map_{\DbliCat}(\Sq(\C), \P) \to \map_{\DbliCat}(\C_v, \P) \simeq \map_{\Theta_2\Seg}(\C, \Vert(\P))
	$$
	whose image is given by the functors $\C \to \Vert(\P)$ that carry every morphism in $\C$ to a vertical arrow of $\P$ that admits a companion.

	Moreover, suppose that $\P$ is a locally complete double $\infty$-category. Then the canonical map $\C_h \to \Sq(\C)$ induces  a monomorphism 
	$$
	\map_{\DbliCat}(\Sq(\C), \P) \to \map_{\DbliCat}(\C_h, \P) \simeq \map_{(\infty,2)\Cat}(\C, \Hor(\P))
	$$
	whose image is given by the functors $\C \to \Hor(\P)$ that carry every morphism in $\C$ to a horizontal arrow of $\P$ that is a companion.
\end{theorem}

We can now show the following:

\begin{proposition}\label{prop.sq-ext-lff}
	Let $\C$ be an $(\infty,2)$-category, and $\P$ be a locally complete 
	double $\infty$-category. For a functor $f : \Sq(\C) \rightarrow \P$, the following two assertions are equivalent:
	\begin{enumerate}
		\item the functor $\Vert(f) : \C \to \Vert(\P)$ is locally fully faithful,
		\item the functor $\Hor(f) : \C \to \Hor(\P)$ is locally fully faithful.
	\end{enumerate}
\end{proposition}
\begin{proof}
	This entails showing that
	$f$ is left orthogonal to the morphism $[1]_h^{\sqcup 2} \rightarrow [1;1]_h$ if and only if 
	$f$ is left orthogonal to the morphism $[1]_v^{\sqcup 2} \rightarrow [1;1]_v$. For any of the two directions 
	$d \in \{h,v\}$, we obtain a commutative square
	\[
		\begin{tikzcd}[column sep = tiny]
			\map(\Sq([1;1]), \Sq(\C))\arrow[d] \arrow[r] &  \map(\Sq([1])^{\sqcup 2}, \Sq(\C)) \times_{\map(\Sq([1])^{\sqcup 2}, \P)} \map(\Sq([1;1]), \P)  \arrow[d]\\
			\map([1;1]_d, \Sq(\C)) \arrow[r] &  \map([1]_d^{\sqcup 2}, \Sq(\C)) \times_{\map([1]_d^{\sqcup 2}, \P)} \map([1;1]_d, \P).
		\end{tikzcd}
	\]
	The desired result follows if we show that the two vertical arrows are equivalences. For the left vertical arrow, this follows directly from 
	\cref{thm.uni-prop-sq}: every vertical arrow of $\Sq(\C)$ has a companion, and every horizontal arrow of $\Sq(\C)$ is a companion. The right vertical arrow factors as a composite:
	\[
		\begin{tikzcd}[row sep = small]
		\map(\Sq([1])^{\sqcup 2}, \Sq(\C)) \times_{\map(\Sq([1])^{\sqcup 2}, \P)} \map(\Sq([1;1]), \P) \arrow[d] \\ 
		\map(\Sq([1])^{\sqcup 2}, \Sq(\C)) \times_{\map([1]_d^{\sqcup 2}, \P)} \map(\Sq([1;1]), \P) \arrow[d] \\
		\map(\Sq([1])^{\sqcup 2}, \Sq(\C)) \times_{\map([1]_d^{\sqcup 2}, \P)} \map([1;1]_d, \P) \arrow[d] \\
		\map([1]_d^{\sqcup 2}, \Sq(\C)) \times_{\map([1]_d^{\sqcup 2}, \P)} \map([1;1]_d, \P).
		\end{tikzcd}
	\]
	In light of \cref{thm.uni-prop-sq}, each arrow in this factorization is a monomorphism. The bottom arrow is an equivalence for similar reasons as detailed above.
	We claim that the middle arrow is an equivalence as well. Indeed, it is an epimorphism (i.e.\ surjective on $\pi_0$) for the following two reasons: 
	\begin{itemize}
	\item the restriction map $\map(\Sq([1;1]), \P) \rightarrow \map([1;1]_d, \P)$ is a monomorphism 
	onto the subspace of maps $[1;1]_d \rightarrow \P$ so that both non-identity arrows of $[1;1]_d$ are carried to an arrow in $\P$ 
	that has/is a companion, 
	\item $f$ preserves arrows that are/have a companion.
	\end{itemize}
	The top arrow is an epimorphism for similar reasons as well.
\end{proof}

We are now ready to construct an abstract form of the Yoneda embedding for $\infty$-equipments using \cref{thm.uni-prop-sq}.

\begin{construction}[Companion embedding]\label{con.companion-embedding}
	Let $\P$ be a locally complete double $\infty$-category with all companions (e.g.\ an $\infty$-equipment). On account of 
	\cref{thm.uni-prop-sq}, there exists a unique dotted extension in the diagram below 
	\[
		\begin{tikzcd}
			\Vert(\P)_v \arrow[r] \arrow[d] & \P, \\
			\Sq(\Vert(\P)) \arrow[ur, dotted, "f"']
		\end{tikzcd}
	\] 
	where the top arrow is given by the canonical inclusion. 
	The functor on horizontal fragments
	$$
	\iota_\P : \Vert(\P) = \Hor(\Sq(\Vert(\P))) \xrightarrow{\Hor(f)} \Hor(\P) 
	$$
	is locally fully faithful on account of \cref{prop.sq-ext-lff}. It acts as the identity on objects, and carries arrows in $\Vert(\P)$ 
	to their associated companions in $\Hor(\P)$. This will be called the \textit{companion embedding for $\P$} throughout this article.
\end{construction}

\begin{example}
	The locally fully faithful companion embedding $\infty\CAT \rightarrow \Hor(\infty\CCAT)$ can  locally be described as follows. 
	Suppose that $\C$ and $\D$ are $\infty$-categories. From 
	the perspective of profunctors, the induced fully faithful functor between mapping $\infty$-categories
	$$
	\fun(\C,\D) \rightarrow \Hor(\infty\CCAT)(\C, \D) \simeq \fun(\D^\op \times \C, \S)
	$$
	carries a functor $f : \C \rightarrow \D$ to the profunctor $\map_\D(-, f(-))$, see \cref{ex.compconj-ccat-prof-tdfib}. In particular, we recover the Yoneda embedding if $\C = \ast$.
\end{example}

\begin{construction}[Conjoint embedding]\label{con.conjoint-embedding}
	There is also a dual construction.
	Let $\P$ be a locally complete double $\infty$-category that admits all conjoints (e.g.\ an $\infty$-equipment). Then $\P^\hop$ admits 
	all companions (see \cite[Remark 4.3]{Comp}), so that we obtain a companion embedding 
	$
	\iota_{\P^\hop} : \Vert(\P)^{\opt} \simeq \Vert(\P^\hop) \rightarrow \Hor(\P^\hop) \simeq \Hor(\P)^{\opo}.
	$
	Here we used the identifications of \cite[Remark 3.19]{Comp}. This functor is adjunct to a functor 
	$$
	\gamma_\P : \Vert(\P) \rightarrow \Hor(\P)^{1,\opt}.
	$$
	Note that it is locally fully faithful, and acts as the identity on objects and carries arrows to their conjoints.
	We call $\gamma_\P$ the \textit{conjoint embedding} for $\P$.
\end{construction}

\section{First notions of formal category theory}\label{section.formal-category-theory}

This last section is devoted to developing the basic notions of the abstract category theory internal to an $\infty$-equipment. The central insight is that one may encode and phrase
desirable properties in the internal category theory via proarrows, and pass from vertical to horizontal fragments via the companion and conjoint embeddings. For convenience, we will fix an $\infty$-equipment $\P$ throughout this section.

\subsection{Injections}  \label{ssection.injections}
We will start by defining an abstract notion of fully faithfulness.

\begin{definition}\label{def.ff-arrows}
	An arrow $f : x \rightarrow y$ in $\P$ is called an \textit{injection} if the vertical identity 2-cell
	\[ 
		\id_f = \begin{tikzcd}
			x\arrow[d,"f"'name=f1] \arrow[r, equal] & x \arrow[d,"f"name=t1] \\
			y \arrow[r,equal] & y
			\arrow[from=f1, to=t1, equal, shorten <= 12pt, shorten >= 12pt]
		\end{tikzcd}
	\]
	is cartesian.
\end{definition}

\begin{remark}\label{rem.injection}
	Let $f : x \to y$ be an arrow in $\P$. Then one readily verfies using the companionship and conjunction triangle identities 
	that we have the following decomposition of the identity 2-cell on $f$: 
	\[
		\id_f \simeq
		\begin{tikzcd}
			x \arrow[r, equal,""name=h1]\arrow[d, equal] & x \arrow[r, equal, ""name=h4] \arrow[d, "f"] & x \arrow[d,equal] \\
			x \arrow[r, "f_\circledast"name=h2] \arrow[r]\arrow[d,"f"'] & y\arrow[d,equal] \arrow[r, "f^\circledast"name=h5] & x\arrow[d,"f"] \\
			y \arrow[r, equal, ""name=h3] & y \arrow[r, equal, ""name=h6] & y,
			\arrow[from=h1,to=h2, Rightarrow, shorten <= 6pt]
			\arrow[from=h2,to=h3, Rightarrow, shorten <= 6pt]
			\arrow[from=h4,to=h5, Rightarrow, shorten <= 6pt]
			\arrow[from=h5,to=h6, Rightarrow, shorten <= 6pt]
		\end{tikzcd}
	\]
	where the 2-cells in the decomposition are companionship/conjunction units and counits.  
	The composite of the bottom row is a cartesian 2-cell by \cref{th.comp-conj-cart}. Hence by \cref{rem.cart-fact}, $f$ is an injection if and only if the composite of the top row defines an equivalence $\id_x \to f^\circledast f_\circledast$ in $\Hor(\P)(x,x)$.
\end{remark}

\begin{example}
	If $\P = \infty\CCAT$, then the injections are precisely the fully faithful functors. To wit, suppose that $f : \C \rightarrow \D$ is a functor between 
	$\infty$-categories. Then  the identity 2-cell $\id_f$ factors uniquely as the vertical composite
	\[
		\id_f \simeq \begin{tikzcd}
			\C\arrow[d,equal] \arrow[r, equal, ""name=h1] & \C \arrow[d,equal] \\
			\C\arrow[d,"f"'] \arrow[r, "{\cart{f}{}{f}}"name=h2] & \C \arrow[d,"f"] \\
			\D \arrow[r,equal,""name=h3] & \D
			\arrow[from=h1, to=h2, Rightarrow, shorten <= 6pt]
			\arrow[from=h2, to=h3, Rightarrow, shorten <= 6pt]
		\end{tikzcd}
	\]
	so that the bottom 2-cell is cartesian (and the top 2-cell coincides with the composite of the top row of \cref{rem.injection}). From the perspective of profunctors,
	it follows from \cref{ex.ccat-cart-lifts} that the top comparison 2-cell is given by the arrow 
	$\map_\C(-,-) \rightarrow \map_\D(f(-),f(-))$ in $\fun(\C^\op \times \C, \S)$ that is induced by $f$. Note that 
	$f$ is an injection if and only if this comparison 2-cell is an equivalence, and this is precisely the case when $f$ is a fully faithful functor.
\end{example}

\begin{example}
	 Let $f : \C \rightarrow \D$ be a functor of $\infty$-categories 
	internal to a complete $\infty$-category $\E$ with universal geometric realizations. Then one readily verifies using the description 
	of cartesian cells of \cref{prop.int-ccat-equip} that $f$ is injection in $\CCAT(\E)$ if and only if the square 
	\[
		\begin{tikzcd}
			(s_0^*\C)(g) \arrow[r, "(s_0^*f)(g)"] \arrow[d] & (s_0^*\D)(g) \arrow[d] \\
			(\C \star \C)(g) \arrow[r, "(f \star f)(g)"] & (\D \star \D)(g)
		\end{tikzcd}
	\]
	is a pullback square in $\E$, for all $g : [n] \rightarrow [1]$. Note that this square may be identified with the commutative square
	$$
	\begin{tikzcd}
		\C([n]) \arrow[r] \arrow[d] & \D([n]) \arrow[d] \\
		\C(g^{-1}(0)) \times \C(g^{-1}(1)) \arrow[r] & \D(g^{-1}(0)) \times \D(g^{-1}(1)). 
	\end{tikzcd}
	$$
	In light of the Segal condition, it is sufficient to check for $n=1$. If $g$ is constant, then the above square is trivially a pullback square. If $g = \id$, then it is precisely 
	the condition for the induced square
	\[
		\begin{tikzcd}
			\C([1]) \arrow[r ] \arrow[d] & \D([1]) \arrow[d] \\
			\C([0])^{\times 2}\arrow[r] & \D([0])^{\times 2}
		\end{tikzcd}
	\]
	to be a pullback square, where the vertical arrows are given by source-target maps. Hence $f$ is an injection if the above square is a pullback square.
\end{example}

Every injective arrow is \textit{corepresentably fully faithful} in the vertical fragment:

\begin{proposition}\label{prop.corep-ff}
	Let $f : x \rightarrow y$ be an injection in $\P$. Then the post-composition functor
	$$
	f \circ (-) : \Vert(\P)(z,x) \rightarrow \Vert(\P)(z,y)
	$$
	is fully faithful for all $z \in \P$.
\end{proposition}
\begin{proof}
The companion embedding for $\P$ (see \cref{con.companion-embedding}) gives rise to a commutative square
\[
	\begin{tikzcd}
		\Vert(\P)(z,x) \arrow[r, "f \circ (-)"]\arrow[d, "(\iota_\P)_{z, x}"'] & \Vert(\P)(z,y)\arrow[d, "(\iota_\P)_{z,y}"] \\
		\Hor(\P)(z,x) \arrow[r, "f_\circledast \circ (-)"] & \Hor(\P)(z,y),
	\end{tikzcd}
\]
so that the vertical arrows are fully faithful. Hence, it suffices to show that the bottom functor is fully faithful.
In view of \cref{prop.adjoints-mapping-cats} and \cref{prop.conj-comp-adjunction}, this bottom functor is left adjoint to 
$f^\circledast \circ (-)$. Consequently, it suffices to show that the unit for this adjunction is an equivalence. 
The associated unit is induced by the horizontal 2-cell $\eta$ that appears in the proof of \cref{prop.conj-comp-adjunction}, 
and $\eta$ is an equivalence in $\Hor(\P)(x,x)$ if and only if $f$ is injective by \cref{rem.injection}.
\end{proof}

\subsection{Adjunctions} \label{ssection.fct-adjunctions} We can give a different characterization of the adjunctions in $\Vert(\P)$ 
using its ambient $\infty$-equipment. 

\begin{proposition}\label{prop.comp-conj-char-adjunction}
	Suppose 
	that  $u : x \rightarrow y$ and $v : y \rightarrow x$ are arrows in $\P$. Then the following assertions are equivalent:
	\begin{enumerate}
		\item $(u,v)$ is an adjoint pair of arrows in $\Vert(\P)$,
		\item there exists an equivalence $u^\circledast \simeq v_\circledast$ in $\Hor(\P)(y,x)$.
	\end{enumerate}
\end{proposition}
\begin{proof}
	The companion embedding $\iota_\P : \Vert(\P) \rightarrow \Hor(\P)$ is a locally fully faithful functor. 
	Consequently, \cref{prop.lff-adjoints} asserts that  $(u,v)$ is an adjoint pair 
	in $\Vert(\P)$ if and only if $(u_\circledast, v_\circledast)$ is an adjoint pair in $\Hor(\P)$. We have shown in \cref{prop.conj-comp-adjunction} that
	the pair  $(u_\circledast, u^\circledast )$ is always 
	an adjoint pair in $\Hor(\P)$. Thus the desired result follows directly from \cref{prop.uniqueness-adjoints}.
\end{proof}

\begin{example}\label{ex.hom-char}
	Let us consider the case that $\P = \infty\CCAT$, and view its proarrows as profunctors between $\infty$-categories. Then \cref{prop.comp-conj-char-adjunction} recovers the fact that
	two functors $u : \C \rightarrow \D$ and $v: \D \rightarrow \C$ form an adjoint pair $(u,v)$ if and only if 
	there exists an equivalence 
	$
	\map_\D(u(-), -) \simeq \map_\C(-, v(-))
	$
	in $\fun(\C^\op \times \D , \S)$; see \cref{ex.compconj-ccat-prof-tdfib}. 
\end{example}

\begin{corollary}\label{cor.comp-conj-char-adjunction}
	An arrow $f : x\rightarrow y$ in $\P$ admits a right adjoint in $\Vert(\P)$ if and only if $f^\circledast$ is a companion. Dually, $f$ admits a left adjoint in $\Vert(\P)$ if and only if 
	$f_\circledast$ is a conjoint.
\end{corollary}

\begin{remark}
	If we consider $\P = \infty\CCAT$ and view proarrows as correspondences, this recovers the fact that 
	a functor $f : \C \rightarrow \D$  admits a right adjoint if and only if the cocartesian fibration that classifies 
	$f$ is also a cartesian fibration; see \cref{ex.compconj-ccat}. This is the definition of adjunctions between $\infty$-categories that is used by Lurie \cite[Definition 5.2.2.1]{HTT}.
\end{remark}

\begin{proposition}\label{prop.ff-adjoints}
	Suppose that 
	$
	u : x \rightleftarrows y : v
	$
	is an adjunction in $\Vert(\P)$. Then $u$ is an injection if and only if the unit  $\eta$ of this adjunction is an equivalence in $\Vert(\P)(x,x)$. Dually, 
	$v$ is an injection if and only if the counit $\epsilon$ is an equivalence in $\Vert(\P)(y,y)$.
\end{proposition}
\begin{proof}
	We only demonstrate the first assertion as the statement for counits is shown similarly.
	The companion embedding induces a fully faithful functor 
	$(\iota_\P)_{x,x}:\Vert(\P)(x,x) \rightarrow \Hor(\P)(x,x)$. Hence, $\eta$ is an equivalence if and only if its image $\eta'$ under the companion embedding is an equivalence.
	Note that $\eta'$ can be identified with the unit of the adjunction $(u_\circledast, u^\circledast)$ by \cref{prop.comp-conj-char-adjunction}.
	On account of \cref{prop.conj-comp-adjunction}, the image $\eta'$ is given by the pasting
	\[
		\begin{tikzcd}
			x \arrow[r, equal, ""name=h1]\arrow[d, equal] & x \arrow[r, equal, ""name=h3] \arrow[d, "u"] & x \arrow[d,equal] \\
			x \arrow[r, "u_\circledast"name=h2] \arrow[r] & y \arrow[r, "u^\circledast"name=h4] & x 
			\arrow[from=h1,to=h2, Rightarrow, shorten <= 6pt]
			\arrow[from=h3,to=h4, Rightarrow, shorten <= 6pt]
		\end{tikzcd}
	\]
	of the appropriate companionship and conjunction units. Thus the desired conclusion follows from \cref{rem.injection}.
\end{proof}

\subsection{Weighted colimits and limits}\label{ssection.weighted-colims} We will now demonstrate 
that $\infty$-equipments support internal notions of \textit{weighted} colimits and limits natively.
Weighted colimits and limits have roots in enriched category theory, where they are used as the suitable generalization of non-enriched colimits and limits \cite{BorceuxKelly}. 
In the context of $\infty$-categories, there are two 
different approaches to weighted colimits and limits by Gepner--Haugseng--Nikolaus \cite{GepnerHaugsengNikolaus}, and Rovelli \cite{Rovelli}. It has been 
shown that the two approaches agree by Haugseng  \cite[Corollary 4.5]{HaugsengEnds}. Hinich has recently introduced a notion of weighted colimits and limits for enriched $\infty$-categories as well \cite{Hinich}. 

\begin{definition}\label{def.proarrow-cones}
Suppose that $f : i \rightarrow x$ is an arrow in $\P$. Then we single out the following special proarrows:
\begin{itemize}[noitemsep]
	\item If $W : j \rightarrow i$ is a proarrow of $\P$, then a proarrow $C : x \rightarrow j$ is called a \textit{proarrow of $W$-weighted cones under $f$} 
	if it represents the presheaf 
$$
\map_{\Hor(\P)(x,i)}(W \circ (-), f^\circledast) : \Hor(\P)(x,j)^\op \rightarrow \S.
$$
	\item If $W : i \rightarrow j$ is a proarrow of $\P$, then a proarrow $C : j \rightarrow x$ is called a \textit{proarrow of $W$-weighted cones over $f$} if it represents the presheaf 
$$
\map_{\Hor(\P)(i,x)}((-) \circ W, f_\circledast) : \Hor(\P)(j, x)^\op \rightarrow \S.
$$
\end{itemize}
In this context, the proarrows $W$ that appear above are also called \textit{weights}.
\end{definition}

\begin{remark}\label{rem.proarrow-cones}
	Let $W : j \to i$ be a proarrow of $\P$. A proarrow $C$ of $W$-weighted cones under $f$ defines a universal horizontal 2-cell of shape
	\[ 
		\begin{tikzcd}
			x \arrow[d,equal] \arrow[r, dashed, "C"] & |[alias=f]| j \arrow[r, "W"] & i \arrow[d,equal] \\
			x \arrow[rr, "f^\circledast"name=t] && i
			\arrow[from=f,to=t,Rightarrow]
		\end{tikzcd}
	\]
	in $\P$. In 2-categorical language, the proarrow $C$ is an example of a \textit{lift} \cite[Section 1]{StreetWalters}. There is a dual picture for proarrows of weighted cones \textit{over} $f$.
\end{remark}

\begin{remark}
	If $\P$ is horizontally closed (see \cref{def.hor-closed}), then it admits all proarrows of cones. For instance, 
	in the context of \cref{rem.proarrow-cones}, the proarrow $C$ is obtained by applying the right adjoint of $W \circ (-) : \Hor(\P)(x,j) \to \Hor(\P)(x,i)$ 
	to the conjoint $f^\circledast$.
\end{remark}

\begin{example}\label{ex.ccat-cones}
Suppose that $\P = \infty\CCAT$ and let $f : I \rightarrow \C$ be a functor. We may consider the \textit{coconical weight} $W : [0] \rightarrow I$, which has the defining property 
that it is the conjoint of the functor $I \rightarrow [0]$. From the viewpoint of profunctors, 
$W$ is given by the constant presheaf $W : I^\op \rightarrow \S : i \mapsto \ast.$
In light of \cref{cor.ccat-hor-closed} and \cref{ex.ccat-cart-lifts}, the proarrow of $W$-weighted cones under $f$  is computed as
$$
\C \rightarrow \S : c \mapsto \int_{i \in I} \map_\C(f(i), c) \simeq \map_{\fun(I,\C)}(f, \Delta c),
$$
where  $\Delta c : I \rightarrow \C$ denotes the diagram that is constant at $c$. We used \cite[Proposition 5.1]{GepnerHaugsengNikolaus} to obtain the second equivalence.
\end{example}

\begin{definition}\label{def.weighted-colims}
	Let $f : i \rightarrow x$ be an arrow in $\P$. Suppose that $W : j \to i$ is a proarrow. The $g : j \rightarrow x$ 
	is called the \textit{$W$-weighted colimit of $f$} if
	$g^\circledast$ is the proarrow of $W$-weighted cones under $f$.

	Dually, consider a proarrow $W : i \to j$. Then an arrow $g : j\to x$ is called the \textit{$W$-weighted limit of $f$} if 
	the companion $g_\circledast$ is the proarrow of $W$-weighted cones over $f$.
\end{definition}

\begin{example}\label{ex.ccat-colims}
	Suppose that $\P = \infty\CCAT$. Let $f : I \rightarrow \C$ be a functor between $\infty$-categories. In light of \cref{ex.ccat-cones}, 
	we conclude that the colimit of $f$ weighted by the coconical weight $W : [0] \rightarrow I$, coincides with the usual colimit of $f$. 
\end{example}

\begin{example}\label{ex.ccat-weighted-colimits}
	Suppose that $W : J \rightarrow I$ is any weight. Viewed as a profunctor, 
	such a weight is given by a presheaf $W : I^\op \times J \rightarrow \S$. From this perspective, the proarrow $C$ of $W$-weighted cones under $f$ is computed by an end
	$$
	C : J^\op \times \C \rightarrow \S : c \mapsto \int_{i \in I} \map_\S(W(i,j), \map_\C(f(i), c))
	$$
	on account of \cref{cor.ccat-hor-closed}.
	A functor $g : J \to \C$ is the $W$-weighted colimit of $f$ if and only if $C$ is equivalent to $g^\circledast : J^\op \times \C \to \S : (j,c) \mapsto \map_\C(g(j), c)$. This coincides with the notion of weighted colimits of \cite{GepnerHaugsengNikolaus} (see also \cite[Section 4]{HaugsengEnds}) in the case that $J=[0]$.

	Suppose for now that $J = [0]$. From the perspective of two-sided discrete fibrations, $W$ is classified by a right fibration $p : I' \rightarrow I$. 
On account of \cref{rem.ccat-hor-closed-bifib}, the left fibration
of $W$-weighted cones under $f$ fits in the diagram 
\[
	\begin{tikzcd}
	E \arrow[r]\arrow[d]& \fun(I', f^\circledast) \arrow[r]\arrow[d] & \fun(I \times [1], \C) \arrow[d, "{(\ev_1, \ev_0)}"] \\ 
	\C \arrow[r] & \fun(I', \C) \times \fun(I', I) \arrow[r, "\id \times f_*"] & \fun(I,\C) \times \fun(I,\C),
	\end{tikzcd}
\]
where both squares are pullback squares. The left bottom functor is given by the product of the diagonal functor 
and  
the constant map that selects $p$. Hence, we conclude that $C$ is classified by the left fibration
	$$
	\C_{fp/} \rightarrow \C.
	$$
	Hence, an object $x \in \C$ is the $W$-weighted colimit of $f$ if and only if it is the colimit of the diagram $fp : I' \rightarrow \C$.
	In view of \cite[Theorem 2.37]{Rovelli}, this coincides with Rovelli's definition of weighted colimits. In particular, we recover \cite[Theorem 1.2]{HaugsengEnds} and the fact that the definitions of weighted colimits by Rovelli and Gepner--Haugseng--Nikolaus agree \cite[Corollary 4.5]{HaugsengEnds}.
\end{example}

Using some proarrow calculus, we may verify that adjunctions have the desired preservation properties of weighted colimits and limits. 

\begin{proposition}\label{prop.cones-adjunctions}
Suppose that we have an adjunction 
$u : x \rightleftarrows y : v$
in $\Vert(\P)$. Then the following is true:
\begin{enumerate}
	\item If $C : x\rightarrow j$ is a proarrow of cones under an arrow $f : i \rightarrow x$ weighted by a proarrow $W : j \rightarrow i$, then $Cu^\circledast$ is 
	a  proarrow of $W$-weighted cones under $uf$.
	\item If $C : j \rightarrow y$ is a proarrow of cones over an arrow $f : i \rightarrow y$ weighted by a proarrow $W : i \rightarrow j$, 
	then $v_\circledast C$ is a proarrow
	of $W$-weighted cones over $vf$.
\end{enumerate}
In particular, $u$ preserves weighted colimits, and $v$ preserves weighted limits.
\end{proposition}
\begin{proof}
We will only handle the case of colimits, the other case is shown similarly. Since $u$ is left adjoint to $v$, \cref{prop.comp-conj-char-adjunction}
implies that $u^\circledast \simeq v_\circledast$. Let $F : y \rightarrow j$ be a proarrow. Using the definition of proarrows of cones 
and  \cref{rem.proarrow-interchange}, we compute that
\begin{align*}
\map_{\Hor(\P)(y,j)}(F, Cu^\circledast) & \simeq \map_{\Hor(\P)(x,j)}(Fv^\circledast, C) \\
&\simeq \map_{\Hor(\P)(x,i)}(WFv^\circledast, f^\circledast) \\
&\simeq \map_{\Hor(\P)(y,i)}(WF, f^\circledast v_\circledast) \\
&\simeq \map_{\Hor(\P)(y,i)}(WF, f^\circledast u^\circledast) \\
&\simeq \map_{\Hor(\P)(y,i)}(WF, (uf)^\circledast),
\end{align*}
and these equivalences are natural. 
For the latter assertion, suppose that $g : j \to x$ is the $W$-weighted colimit of $f$, i.e.\ suppose that $g^\circledast = C$. Then $(ug)^\circledast \simeq Cu^\circledast$ so that $ug$ is the $W$-weighted colimit of $uf$ by the above.
\end{proof}
\begin{construction}\label{con.colim-functor}
	Suppose that $W : j \rightarrow i$ is a weight in $\P$. Let $x \in \P$. We will write $$A \subset \Vert(\P)(i,x)$$ for the full subcategory spanned by those arrows $f : i \rightarrow x$ 
	that admit a $W$-weighted colimit. Consider the composite functor
	\begin{align*}
	C : \Vert(\P)(i,x)^\op \xrightarrow{(\gamma_\P)_{i,x}} \Hor(\P)(x,i) &\xrightarrow{\phantom{(W \circ (-) )^*}} \cat{PSh}(\Hor(\P)(x,i)) \\
	&\xrightarrow{(W \circ (-) )^*} \cat{PSh}(\Hor(\P)(x,j)),
	\end{align*}
	where the first arrow comes from the conjoint embedding (see \cref{con.conjoint-embedding}), and the second arrow is the Yoneda embedding.
	It carries an arrow $f : i \rightarrow x$ to the presheaf 
	$
	C(f) = \map_{\Hor(\P)(x,i)}(W \circ (-) , f^\circledast),
	$
	whose representative (if it exists) would be called a $W$-weighted cone under $f$.
	Let us consider the fully faithful functor
	$$
	i : \Vert(\P)(j,x)^\op \xrightarrow{(\gamma_\P)_{j,x}} \Hor(\P)(x,j) \rightarrow \PSh(\Hor(\P)(x,j)),
	$$
	where the second arrow is the Yoneda embedding. By assumption, the essential image of the restricted functor $C|A^\op$ is contained in the essential image of $i$. 
	Consequently, $C$ factors through $i$. After applying $(-)^\op$ to this factorization, we obtain a functor
	$$
	{\colim}^W: A \rightarrow \Vert(\P)(j,x),
	$$
	that carries an arrow $f \in A$ to its $W$-weighted colimit so that
	$$
	\map_{\Vert(\P)(j,x)}({\colim}^W f, g) \simeq \map_{\Hor(\P)(x,i)}(W g^\circledast , f^\circledast),
	$$
	naturally in $g \in \Vert(\P)(j,x)$.
\end{construction}

\begin{example}\label{ex.ccat-weighted-colimits-2}
	Consider the $\infty$-equipment $\widehat{\infty\CCAT}$ of large $\infty$-categories (see \cref{rem.ccat-size}). The proarrows that are contained in $\infty\CCAT \subset \widehat{\infty\CCAT}$ will be called 
	\textit{small proarrows}.
	If $\C$ is a cocomplete $\infty$-category, then it admits all colimits 
	weighted by these small proarrows. To wit, suppose that $f : I \rightarrow \C$ is a functor, and let $W : J \rightarrow I$ be a small proarrow. 
	We will view $W$ as a profunctor $W : I^\op \times J \rightarrow \S$. Note that $I$ and $J$ are small. Since $\C$ is cocomplete, 
	it is tensored over spaces. I.e.\ there exists a functor 
	$
	- \times - : \S \times \C \rightarrow \C 
	$
	so that $$\map_\C(K \times x, y) \simeq \map_\S(K, \map_\C(x,y)),$$ naturally for all $x,y \in \C$ and (small) spaces $K$. Now, one readily verifies 
	that the $W$-weighted colimit $g : J \rightarrow \C$ of $f$ exists, and is given by the coend 
	$$
	g(j) = \int^{i \in I}W(i,j) \times f(i)
	$$ 
	on account of \cref{ex.ccat-weighted-colimits}.
	Conversely, \cref{ex.ccat-colims} implies that any $\infty$-category that admits all colimit weighted by small proarrows, is cocomplete. There is an analogous result for complete $\infty$-categories and weighted limits.
\end{example}

\subsection{Point-wise Kan extensions}\label{ssection.pke}
We introduce the following familiar terminology for colimits weighted by conjoints and limits weighted by companions:

\begin{definition}\label{def.kan-extensions}
	Let $f : i \rightarrow x$ and $w : i \rightarrow j$ be arrows of $\P$. 
	Then the $w^\circledast$-weighted colimit of $f$ is called the \textit{left Kan extension} 
	of $f$ along $w$. We say that a lax triangle
	\[
		\begin{tikzcd}
			i \arrow[dr,"f"name=f] \arrow[d, "w"'] \\
			|[alias=t]|j \arrow[r, "g"'] \arrow[r] & x
			\arrow[from=f,to=t, Rightarrow, shorten <= 6pt, "\eta"]
		\end{tikzcd}
	\]
	in $\Vert(\P)$ \textit{exhibits $g$ as the left Kan extension of $f$} 
	if the map 
	$$
	\map_{\Hor(\P)(x,j)}(-, g^\circledast) \to \map_{\Hor(\P)(x,i)}(w^\circledast(-), f^\circledast)
	$$
	that is induced by the image $w^\circledast g^\circledast \to f^\circledast$ of $\eta$ under the conjoint embedding, is an equivalence. In this case, 
	$\eta$ is called the \textit{unit} of the left Kan extension.
		
 Dually, the $w_\circledast$-weighted limit of $f$ is called the \textit{right Kan extension} of $f$ along $w$. A triangle
 \[
	\begin{tikzcd}
		|[alias=t]|i \arrow[r,"f"] \arrow[d, "w"'] & x \\
		j \arrow[ur, "g"'name=f]
		\arrow[from=f,to=t, Rightarrow, shorten <= 6pt, "\epsilon"]
	\end{tikzcd}
 \]
 in $\Vert(\P)$ is said to \textit{exhibit $g$ as the right Kan extension of $f$} 
 if the map 
 $$
 \map_{\Hor(\P)(j,x)}(-, g_\circledast) \to \map_{\Hor(\P)(i,x)}((-)w_\circledast, f_\circledast)
 $$
 that is induced by the image $g_\circledast w_\circledast \to f_\circledast$ of $\epsilon$ under the companion embedding, is an equivalence. In this case, 
 $\epsilon$ is called the \textit{unit} of the left Kan extension.
\end{definition}

We will focus on the study of left Kan extensions throughout this section, but all dual results for 
right Kan extensions hold as well and are demonstrated similarly.

\begin{example}
	Suppose that $\P = \infty\CCAT$. Let $f : I \to \C$ and $w : I \to J$ be functors. Unraveling the definitions, 
	the left Kan extension in the sense of \cref{def.kan-extensions} exists if and only if the proarrow $C : \C \to J$ of  
	$w^\circledast$-weighted cones under $f$ is a conjoint. We will show how this can be studied by viewing the proarrows of $\infty\CCAT$ as two-sided discrete fibrations. 
	
	As will be explained in \cref{ex.ccat-strongly-pointed}, 
	$C$ is a conjoint if and only if $j^\circledast C$ is a conjoint for every object $j : [0] \to J$. 
	A simple computation using \cref{rem.proarrow-interchange} shows that $j^\circledast C$ is the proarrow of cones under $f$ weighted 
	by $w^\circledast j_\circledast : [0] \to I$. Using \cref{ex.ccat-cart-lifts}, one readily computes 
	that $w^\circledast j_\circledast$ is classified by the right fibration $I/j := I \times_J J/j \to J$. 
	Thus \cref{ex.ccat-weighted-colimits} implies that $j^\circledast C$ is a conjoint if and only if 
	the colimit of $I/j \to I \to \C$ exists. We conclude that the left Kan extension $g$ of $f$ along $w$ exists (in the sense of \cref{def.kan-extensions})
if and only if the colimits $I/j \to I \to \C$ exist for every $j \in J$. In this case, the above shows that $g$ is computed by $g(j) = \colim(I/j \to I \to \C)$.
	We thus recover the usual notion of left Kan extensions for $\infty$-categories.

A similar strategy to compute left Kan extensions may be performed in any $\infty$-equip\-ment with a suitable notion of fibrations. This is 
discussed in the sequel paper \cite[Section 7]{EquipII}.

	Taking the viewpoint of profunctors, the left Kan extension $g$ (if it exists) of $f$ along $w$ is computed by the familiar coend formula 
	$$
	g(j) = \int^{i \in I} \map_J(w(i), j) \times f(i),
	$$
	as explained in \cref{ex.ccat-weighted-colimits-2}.
\end{example}

\begin{proposition}\label{cor.unit-pke-ff}
	Suppose that $w : i \rightarrow j$ is an injection in $\P$. Then the unit 2-cells
	associated to left Kan extension along $w$ are invertible.
\end{proposition}
\begin{proof}
	Suppose that $g : j \rightarrow x$ is the right Kan extension of an arrow $f : i \rightarrow x$ along $w$. 
	Then the unit 2-cell $\eta : f \rightarrow gw$ is invertible if and only if its image $\bar{\eta}$ under the fully faithful functor 
	$(\gamma_\P)_{i,x} : \Vert(\P)(i,x) \rightarrow \Hor(\P)(x,i)^\op$ is invertible. 
	Let us consider the following commutative diagram
	\[
	\begin{tikzcd}
		\map_{\Hor(\P)(x,j)}(w_\circledast(-), g^\circledast)\arrow[d, "w^\circledast"']\arrow[dr]\\ 
		 \map_{\Hor(\P)(x,i)}(w^\circledast w_\circledast(-), w^\circledast g^\circledast) \arrow[r]\arrow[d] & \map_{\Hor(\P)(x,i)}(w^\circledast w_\circledast(-), f^\circledast) \arrow[d] \\
		 \map_{\Hor(\P)(x,i)}(-, w^\circledast g^\circledast) \arrow[r] & \map_{\Hor(\P)(x,i)}(-, f^\circledast).
	\end{tikzcd}
	\]
	Here, the vertical arrows in the bottom square are induced by composing with the canonical 2-cell $\id_i \rightarrow w^\circledast w_\circledast$, 
	and hence equivalences (see \cref{rem.injection}).
	The horizontal arrows are induced by composition with $\bar{\eta}$. The total composite of the two left vertical arrows are an equivalence on account of \cref{prop.conj-comp-adjunction},
	and  the drawn diagonal 
	arrow must be an equivalence as well as $\eta$ exhibits a Kan extension. It follows that the bottom horizontal arrow is an equivalence by 2-out-of-3. Thus $\bar{\eta}$ is an equivalence.
\end{proof}

\begin{notation}\label{not.external-pke}
	Let $x \in \P$ and $w : i \rightarrow j$ be an arrow of $\P$. Then we can apply the construction of \cref{con.colim-functor} to the specific case that $W = w^\circledast$. In this case, 
	$A \subset \Vert(\P)(i,x)$ is the full subcategory spanned by the arrows that admit a left Kan extension along $w$. We will write 
	$$
	w_! := {\colim}^{w^\circledast} : A \rightarrow \Vert(\P)(j,x)
	$$
	for the functor that carries each $f \in A$ to its left Kan extension along $w$.
\end{notation}

\begin{proposition}\label{prop.external-pke}
	In the context of \cref{not.external-pke}, the functor $w_!$ is a partial left adjoint to the restriction functor 
	$$
	w^* : \Vert(\P)(j,x)\rightarrow \Vert(\P)(i,x).
	$$
	Moreover, the unit of this partial adjunction 
	is component-wise given by the 2-cells that exhibit left Kan extensions.
\end{proposition}
\begin{proof}
	There is a natural equivalence 
	$$
	\map_{\Vert(\P)(j,x)}(w_!f, g) \simeq \map_{\Hor(\P)(x,i)}(w^\circledast g^\circledast , f^\circledast)
	$$
	for $f\in A$ and $g \in \Vert(\P)(j,x)$ by construction. The conjoint embedding gives rise to a commutative square 
	\[
		\begin{tikzcd}
			\Vert(\P)(j,x) \arrow[r, "(-) \circ w"] \arrow[d] & \Vert(\P)(i,x) \arrow[d] \\
			\Hor(\P)(x,j)^\op \arrow[r, "w^\circledast \circ (-)"] & \Hor(\P)(x,i)^\op,
		\end{tikzcd}
	\]
	where the vertical arrows are fully faithful. Hence, we deduce that
	$$
	\map_{\Vert(\P)(j,x)}(w_!f, g) \simeq \map_{\Vert(\P)(i,x)}(f, gw)
	$$
	naturally, as desired.
	
	To identify the unit
	$
	f \rightarrow (w_!f)w
	$
	at $f\in A$,
	we compute its image under the conjoint embedding 
	$
	\eta_f : w^\circledast (w_!f)^\circledast \simeq ((w_!f)w)^\circledast \rightarrow f^\circledast,
	$
	which is obtained as the image of the identity $w_!f \rightarrow w_!f$ under the natural equivalence above. In light of \cref{con.colim-functor}, 
	$\eta_f$ may be computed as the image of the horizontal identity 2-cell $(w_!f)^\circledast \rightarrow (w_!f)^\circledast$ under the component 
	of the identification
	$$
	\map_{\Hor(\P)(x,j)}(-, (w_!f)^\circledast) \simeq \map_{\Hor(\P)(x,i)}(w^\circledast (-), f^\circledast)
	$$
	at $(w_!f)^\circledast$. By construction, 
	this is precisely the image of the 2-cell under the conjoint embedding that exhibits the left Kan extension.
\end{proof}

\begin{remark}\label{rem.wke}
	Any $(\infty,2)$-category $\C$ supports a notion of extensions. Namely, if $f : i \rightarrow x$ and $w : i \rightarrow j$ are 
	arrows of  $\C$, then the \textit{left extension} of $f$ along $w$ 
	is the arrow (if it exists) that represents the presheaf 
	$$
	\map_{\C(i,x)}(f, w^*(-)) : \C(j,x) \rightarrow  \S.
	$$
	In light of \cref{prop.external-pke}, left Kan extensions in $\P$ are left extensions in the $(\infty,2)$-category $\Vert(\P)$. However, 
	the reverse implication fails in general \cite{CampionCounterEx}.
\end{remark}

\subsection{Exact squares}\label{ssection.exact-squares}
We may now study so-called \textit{Beck-Chevalley conditions} for point-wise Kan extensions. 
This about the following situation. Suppose that 
we have a lax square 
\[
	\begin{tikzcd}
		i \arrow[r, "p"]\arrow[d,"w"'] &|[alias=f]| k \arrow[d, "v"] \\ 
		|[alias=t]|j \arrow[r, "q"'] & l 
		\arrow[from=f,to=t, Rightarrow, "\alpha"]
	\end{tikzcd}
\]
 and a left Kan extension diagram
\[
	\begin{tikzcd}
		k \arrow[dr, "f"name=f]\arrow[d,"v"']   \\
		|[alias=t]|l \arrow[r, "g"'] & x
		\arrow[from=f, to=t, Rightarrow, shorten <= 6pt, "\eta"]
	\end{tikzcd}
\]
in $\Vert(\P)$.
Under suitable conditions, the restriction $gq$ is the left Kan extension 
of the restriction $fp$ along $w$, and the associated comparison 2-cell $fp \rightarrow gqw$ 
is computed by whiskering and composing the comparison 2-cells above
$$
	fp \xrightarrow{\eta p} gvp \xrightarrow{g\alpha} gqw.
$$
To formulate this, we introduce the following terminology in our context of $\infty$-equipments (cf.\ \cite[Definition 9.2.2]{RiehlVerity}):

\begin{definition}\label{def.exact-squares}
	Suppose that we have a lax square
	\[
		\sigma = \begin{tikzcd}
			i \arrow[r, "p"]\arrow[d,"w"'] &|[alias=f]| k \arrow[d, "v"] \\ 
			|[alias=t]|j \arrow[r, "q"'] & l 
			\arrow[from=f,to=t, Rightarrow, "\alpha"]
		\end{tikzcd}
	\]
	in $\Vert(\P)$. Then we obtain an induced comparison 2-cell 
	$$
	\phi : p_\circledast w^\circledast \rightarrow v^\circledast q_\circledast
	$$
	that fits in the unique factorization 
	\[
		\begin{tikzcd}
			i \arrow[r, equal, ""name=f]\arrow[d,"w"'] & i \arrow[d, "p"]\\
			j\arrow[d,"q"'] & k\arrow[d, "v"] \\
			l \arrow[r, equal, ""name=t] & l
			\arrow[from=f, to=t, "\scriptstyle\alpha", phantom]
		\end{tikzcd}
		\simeq 
		\begin{tikzcd}
			i \arrow[r, equal, ""name=h1]\arrow[d,"w"'] & i \arrow[d, "p"]\\
			j \arrow[r, "p_\circledast w^\circledast"name=h2]\arrow[d,equal] & k\arrow[d,equal]\\
			j \arrow[r, "v^\circledast q_\circledast"'name=h3]\arrow[d,"q"'] & k \arrow[d,"v"]\\
			l \arrow[r, equal, ""name=h4] & l
			\arrow[from=h1,to=h2,Rightarrow, shorten <= 6pt]
			\arrow[from=h2,to=h3, "\scriptstyle\phi", phantom]
			\arrow[from=h3,to=h4,Rightarrow]
		\end{tikzcd}		
	\]
	in $\P$
	so that the top and bottom 2-cells on the right are cocartesian and cartesian respectively. 
	The lax square $\sigma$ is said to be \textit{exact} if $\phi$
	is an equivalence in $\Hor(\P)(j,k)$.
\end{definition}

\begin{remark}	
The terminology of exact squares dates back to Hilton \cite{Hilton}, and was originally introduced in the context 
of abelian categories. It was later 
further categorified by Guitart \cite[Section 4]{Guitart}, where similar phenomena  are studied in the Yoneda structures of Street and Walters \cite{StreetWalters}. 
\end{remark}

\begin{example}
	A map $f : x \rightarrow y$ in $\P$ is injective if and only if the following commutative square is exact:
	\[
		\begin{tikzcd}
			x \arrow[r ,equal]\arrow[d,equal] & x \arrow[d,"f"]\\
			x \arrow[r, "f"] & y,
		\end{tikzcd}
	\]
	see \cref{rem.injection}.
\end{example}

\begin{example}
	Suppose that 
	$u : x \rightleftarrows y :v$
	is an adjunction in $\Vert(\P)$ with unit $\eta$ and counit $\epsilon$. Then one may readily verify that
	the following two lax squares are exact:
	\[
		\begin{tikzcd}
			x \arrow[r, equal ]\arrow[d,"u"'] &|[alias=f]| x \arrow[d, equal] \\ 
			|[alias=t]|y \arrow[r, "v"'] & x,
			\arrow[from=f,to=t, Rightarrow, "\eta"]
		\end{tikzcd}
		\quad 
		\begin{tikzcd}
			y \arrow[r, "v"]\arrow[d,equal] &|[alias=f]| x \arrow[d, "u"] \\ 
			|[alias=t]|y \arrow[r, equal] & y.
			\arrow[from=f,to=t, Rightarrow, "\epsilon"]
		\end{tikzcd}
	\]
\end{example}

\begin{proposition}\label{prop.exact-squares}
	Suppose that we have an exact square
	\[
		\begin{tikzcd}
			i \arrow[r, "p"]\arrow[d,"w"'] &|[alias=f]| k \arrow[d, "v"] \\ 
			|[alias=t]|j \arrow[r, "q"'] & l, 
			\arrow[from=f,to=t, Rightarrow, "\alpha"]
		\end{tikzcd}
	\]
 	in $\Vert(\P)$. Then the following is true: 
	\begin{enumerate}
		\item If a diagram \[
			\begin{tikzcd}
				k \arrow[dr, "f"name=f]\arrow[d,"v"']  \\
				|[alias=t]|l \arrow[r, "g"'] & x
				\arrow[from=f, to=t, Rightarrow, shorten <= 6pt, "\eta"]
			\end{tikzcd}
		\] exhibits a left Kan extension, then the pasting of $\alpha$ and $\eta$ defines a triangle
		\[
			\begin{tikzcd}
				i \arrow[dr, "fp"name=f]\arrow[d,"w"']  \\
				|[alias=t]|j \arrow[r, "gq"'] & x
				\arrow[from=f, to=t, Rightarrow, shorten <= 6pt, "\omega"]
			\end{tikzcd}
			:= 
			\begin{tikzcd}
				i \arrow[r,"p"]\arrow[d,"w"'] &|[alias=f0]| k \arrow[dr, "f"name=f]\arrow[d,"v"']  \\
				|[alias=t0]| j \arrow[r, "q"'] & |[alias=t]|l \arrow[r, "g"'] & x
				\arrow[from=f, to=t, Rightarrow, shorten <= 6pt, "\eta"]
				\arrow[from=f0,to=t0, Rightarrow, "\alpha"]
			\end{tikzcd}
		\]
		that exhibits a left Kan extension as well.
		\item Conversely, suppose that $q$ is an equivalence. If the left Kan extension of $fp$ along $w$ exists, then 
		the left Kan extension of $f$ along $v$ exists.
	\end{enumerate}
\end{proposition}
\begin{proof}
	Let us write $\phi : p_\circledast w^\circledast \to v^\circledast q_\circledast$ for the 2-cell of \cref{def.exact-squares}.
	We will start by showing (1).
	Under the conjoint embedding, $\eta$ corresponds to a horizontal 2-cell $v^\circledast g^\circledast \to f^\circledast$. 
	This represents a map $\bar{\eta} : \map_{\Hor(\P)(x,l)}(-, g^\circledast) \to \map_{\Hor(\P)(x,k)}(v^\circledast (-), f^\circledast)$
	which is an equivalence if and only $\eta$ exhibits a left Kan extension. We may now consider the composite
	\begin{align*}
	\map(-,q^\circledast g^\circledast) \xrightarrow{\simeq} \map(q_\circledast(-), g^\circledast) \xrightarrow{\bar{\eta}q_\circledast^*} \map(v^\circledast q_\circledast (-), f^\circledast) &\xrightarrow{\phi^*} \map(p_\circledast w^\circledast(-), f^\circledast) \\ 
	& \xrightarrow{\simeq} \map(w^\circledast(-), p^\circledast f^\circledast).
	\end{align*}
	We claim that a straightforward diagram chase shows that this is represented by $\omega$ under the conjoint embedding. Since the above composite 
	is an equivalence, this would then imply that $\omega$ exhibits a left Kan extension as desired.
	To verify the remaining claim, 
	we note that the above composite is represented by the following composite of 2-cells $$\omega' : w^\circledast q^\circledast g^\circledast \to p^\circledast p_\circledast w^\circledast q^\circledast g^\circledast \to p^\circledast v^\circledast q_\circledast q^\circledast g^\circledast \to p^\circledast v^\circledast g^\circledast \to p^\circledast f^\circledast$$
	 where the first comes from the canonical 2-cell $\id \to p^\circledast p_\circledast$, the second is induced from $\phi$ by whiskering, the third  
	comes from the canonical 2-cell $q_\circledast q^\circledast \to \id$, and the fourth comes from whiskering $\eta$ and $p$ under the conjoint embedding. A straight-forward manipulation with companionship and conjunction 
	triangle identities will show that composing just the first three 2-cells, the resulting 2-cell $w^\circledast q^\circledast g^\circledast \to p^\circledast v^\circledast g^\circledast$ is 
	obtained under the conjoint embedding by whiskering $\alpha$ with $g$. Thus we conclude that $\omega'$ is the image of $\omega$ under the conjoint embedding.

	We will now show (2). Suppose that $g : j \to x$ is the left Kan extension of $fp$ along $w$. Then we obtain natural equivalences 
	$$
	\map(-, g^\circledast) \simeq \map(w^\circledast(-), p^\circledast f^\circledast) \simeq \map(p_\circledast w^\circledast (-), f^\circledast) \simeq 
	\map(v^\circledast q_\circledast(-),f^\circledast).
	$$
	Hence, if $r : l\to j$ is an inverse to $q$, then $gr$ must be the left Kan extension of $f$ along $v$.
\end{proof}

\begin{corollary}\label{cor.exact-square-mate}
	Suppose that
	\[
		\begin{tikzcd}
			i \arrow[r, "p"]\arrow[d,"w"'] &|[alias=f]| k \arrow[d, "v"] \\ 
			|[alias=t]|j \arrow[r, "q"'] & l 
			\arrow[from=f,to=t, Rightarrow, "\alpha"]
		\end{tikzcd}
	\]
	is an exact square in $\Vert(\P)$. Let $x\in \P$, and suppose 
	that all arrows $i \rightarrow x$ and $j \rightarrow x$ admit left Kan extensions 
	along $w$ and $v$ respectively. Then 
	the mate transformation 
	$$
	w_!p^* \rightarrow w_!p^*v^*v_! \rightarrow w_!w^*q^*v_! \rightarrow q^*v_!
	$$
	between functors $\Vert(\P)(k,x) \rightarrow \Vert(\P)(k,x)$
	is invertible.
\end{corollary}

\begin{example}\label{rem.proper-smooth}
	Consider a pullback square
	\[
		\begin{tikzcd}
			I \arrow[r, "p"]\arrow[d,"w"'] & K \arrow[d, "v"] \\ 
			J \arrow[r, "q"] & L
		\end{tikzcd}
	\]
	of $\infty$-categories. If $q$ is smooth or $v$ is proper as in Definitions 4.4.1 and 4.4.15 of \cite{Cisinski}, 
	then this square is exact in $\infty\CCAT$. This may be deduced from \cite[Proposition 6.4.3]{Cisinski}.
\end{example}

\subsection{Initiality and finality}\label{ssection.initial-final} In the context of $\infty$-categories, there is a notion 
of \textit{initial} and \textit{final} functors \cite[Section 4.1]{HTT}. Under an extra assumption 
on the $\infty$-equipment $\P$, we may develop a synthetic notion of initiality and finality in $\P$ as well.

\begin{definition}\label{def.pointed-equipment}
	An object $\ast \in \P$ is called a \textit{vertical terminal object} if 
	$\Vert(\P)(x,\ast) \simeq [0]$
	for all $x \in \P$.
\end{definition}

\begin{definition}\label{def.conical-colims}
	Suppose that $\P$ has a vertical terminal object. We define the following of special weights associated to $i \in \P$: 
	\begin{itemize}
		\item the \textit{coconical weight} $\Delta^i$ is the conjoint of the unique arrow $i \rightarrow \ast$,
		\item the \textit{conical weight} $\Delta_i$ is the companion of the unique arrow $i \rightarrow \ast$.
	\end{itemize} If the colimit (resp.\ limit) of a diagram $f : i \rightarrow x$ weighted by the coconical weight (resp.\ conical weight) exists, then 
	these are called \textit{conical}, and we denote them by $\colim f$ (resp.\ $\lim f$).
\end{definition}

\begin{example}
	The category $[0]$ is a vertical terminal object for $\infty\CCAT$. More generally, for every complete $\infty$-category $\E$ with universal geometric realizations, $\CCAT(\E)$ has a vertical terminal object
	given by the simplicial object that is constant at the terminal object of $\E$.
\end{example}

\begin{proposition}\label{prop.initiality-finality}
	Suppose that $\P$ has a vertical terminal object $\ast \in \P$.
	Let $f : i \rightarrow j$ be an arrow in $\P$. Then the following statements are equivalent:
	\begin{enumerate}
		\item there exists an equivalence $f_\circledast \Delta^i \simeq \Delta^j$ in  $\Hor(\P)(\ast, j)$,
		\item the commutative square
		\[
			\begin{tikzcd}
				i \arrow[r,"f"]\arrow[d] & j \arrow[d] \\ 
				\ast \arrow[r, equal] & \ast
			\end{tikzcd}
		\]
		is exact.
	\end{enumerate}
	Dually, the following statements are equivalent:
	\begin{enumerate}[label=(\roman*)]
		\item there exists an equivalence $\Delta_i f^\circledast  \simeq \Delta_j$ in $\Hor(\P)(j, \ast)$,
		\item the commutative square
		
		\[
			\begin{tikzcd}
				i \arrow[r]\arrow[d, "f"'] & \ast \arrow[d,equal] \\ 
				j \arrow[r] & \ast
			\end{tikzcd}
		\]
		is exact.
	\end{enumerate}
\end{proposition}
\begin{proof}
	We will show that assertions (1) and (2) are equivalent. 
	The other assertions are handled similarly.
	Suppose that (1) holds. Unfolding the definition of exactness, we must show that 
	the canonical 2-cell
	$f_\circledast \Delta^i \rightarrow \Delta^j$ is an equivalence. By assumption, 
	the proarrows on either side are in the essential image of
	the fully faithful conjoint embedding $[0] \simeq \Vert(\P)(j, \ast)^\op \rightarrow \Hor(\P)(\ast, j)$, 
	and hence the canonical 2-cell must be an equivalence. It follows immediately that (2) implies (1) as well.
\end{proof}

\begin{definition}\label{def.initiality-finality}
	Suppose that $\P$ has a vertical terminal object $\ast \in \P$. Let $f : i \rightarrow j$ be an arrow in $\P$.
	Then $f$ is called \textit{final} if one of the equivalent conditions (1) or (2) of \cref{prop.initiality-finality} are met. Dually, 
	$f$ is called \textit{initial} if conditions (i) or (ii) of \cref{prop.initiality-finality} hold.
\end{definition}

We will study final arrows in the remainder of this section, but the dual results for initial arrows can be obtained analogously.

\begin{proposition}
	Suppose that $\P$ has a vertical terminal object $\ast \in \P$. Let $f : i \to j$ be a final arrow in $\P$.
	Then for every arrow $g :j \to x$, the colimit of $g$ exists if and only if the  colimit of $gf$ exists, and in this case, the 
		canonical map 
		$$
		\colim gf \rightarrow \colim g
		$$
		is an equivalence.
\end{proposition}
\begin{proof}
	Suppose that the colimit of $g$ exists. Then this comes with a left Kan extension diagram
	\[
		\begin{tikzcd}
			i \arrow[dr, "f"name=f]\arrow[d] \\
			|[alias=t]|\ast \arrow[r, "\colim g"'] & x.
			\arrow[from=f,to=t,Rightarrow,shorten <= 6pt]
		\end{tikzcd}
	\]
	If we paste this diagram with the exact square that appears in (2) of \cref{prop.initiality-finality}, 
	then part (1) of \cref{prop.exact-squares} implies that we obtain a triangle that exhibits $\colim g$ to be the colimit of $gf$.
	Conversely, if the colimit of $gf$ exists then we may apply part (2) of 
	\cref{prop.exact-squares} to the exact square appearing in (2) of \cref{prop.initiality-finality}.
\end{proof}

As in the case of final arrows between $\infty$-categories \cite[Proposition 4.1.1.3]{HTT}, 
final arrows in an $\infty$-equipment have a cancellation property:

\begin{proposition}
	Suppose that $f : i \rightarrow j$ and $g : j \rightarrow k$ are arrows in a $\infty$-equipment $\P$ with a vertical terminal object.
	If $f$ is final, then the composite $gf$ is final if and only if $g$ is final.
\end{proposition}
\begin{proof}
	If $f$ is final, we have
	$
	(gf)_\circledast \Delta^i \simeq g_\circledast \Delta^j.
	$
	From this computation, the desired conclusion follows immediately.
\end{proof}

We can also give an abstract analogue of \textit{Quillen's theorem A} \cite[Subsection 4.1.3]{HTT}.

\begin{definition}\label{def.strongly-pointed}
	A \textit{strong vertical terminal object} $\ast \in \P$ is a vertical terminal object 
	so that
	 for any proarrow $F:x\rightarrow y$, the following is true:
	\begin{enumerate}
		\item $F$ is a companion if and only if for every arrow $p : \ast \rightarrow x$, the proarrow $Fp_\circledast$ is a companion,
		\item $F$ is a conjoint if and only if for every arrow $p : \ast \rightarrow x$, the proarrow $p^\circledast F$ is a conjoint.
	\end{enumerate}
\end{definition}

\begin{example}\label{ex.ccat-strongly-pointed}
	The vertical terminal object of $\infty\CCAT$ is strong. This follows from the Yoneda lemma. 
	Indeed, from the perspective of profunctors, a proarrow $F : \C \rightarrow \D$ is a companion 
	if and only if the associated functor $\C \rightarrow \PSh(\D)$ factors through the Yoneda embedding $\D \rightarrow \PSh(\D)$. Since the Yoneda embedding 
	is fully faithful, one just has to check that for every object $p \in \C$, the evaluation in $p$, computed by $Fp_\circledast$, is representable. 
	A similar argument shows that condition (2) also holds. 
	
	As a warning, we not that this is not true in general: if $\E$ is an arbitrary $\infty$-topos, then
	the vertical terminal object of $\CCAT(\E)$ is \textit{virtually never} strong.
\end{example}

\begin{proposition}\label{prop.quillen-A}
	Suppose that $\P$ has a strong vertical terminal object $\ast \in \P$.
	Then an arrow $f: i \rightarrow j$ in $\P$ is final
	if and only if for any $p : \ast \rightarrow j$, we have
	$$
	p^\circledast f_\circledast \Delta^i \simeq \Delta^\ast 
	$$
	in $\Hor(\P)(\ast, \ast)$.
\end{proposition}
\begin{proof}
	Note that $f_\circledast \Delta^i$ is equivalent to $\Delta^\ast$ if and only if it is a conjoint. Hence, the result 
	follows from the fact that $\ast$ is strong.
\end{proof}

\begin{example}
	Take $\P = \infty\CCAT$.
	Suppose that $f : I \rightarrow J$ is a functor between $\infty$-categories.
	Let $p : [0] \rightarrow J$ be a functor that corresponds to an object $j \in J$. From the perspective of two-sided discrete fibrations, 
	$p^\circledast f_\circledast : I \rightarrow [0]$ is given by the left fibration
	$$
	I \times_J j/J \rightarrow I.
	$$
	In light of \cref{prop.quillen-A}, $f$ is final if and only if 
	$$
	p^\circledast f_\circledast \Delta^I \simeq (I \times_J j/J)[(I \times_J j/J)^{-1}] \simeq \ast
	$$
	in $\Hor(\infty\CCAT)([0], [0]) \simeq \DFib([0],[0]) \simeq \S$ for every $j \in J$. Here 
	we used \cref{prop.ccat-hor-composition}. This is precisely the content of Quillen's theorem A for $\infty$-categories, and this 
	shows that \cref{def.initiality-finality} recovers the usual notion of final functors between $\infty$-categories.
\end{example}

\nocite{*}
\bibliographystyle{amsalpha}
\bibliography{EquipmentsI}

\newcommand{\etalchar}[1]{$^{#1}$}
\providecommand{\bysame}{\leavevmode\hbox to3em{\hrulefill}\thinspace}
\providecommand{\MR}{\relax\ifhmode\unskip\space\fi MR }
\providecommand{\MRhref}[2]{%
  \href{http://www.ams.org/mathscinet-getitem?mr=#1}{#2}
}
\providecommand{\href}[2]{#2}
\begin{thebibliography}{BDG{\etalchar{+}}16b}

\bibitem[AF20]{AyalaFrancis}
D.~Ayala and J.~Francis, \emph{Fibrations of $\infty$-categories}, Higher Structures \textbf{4} (2020), no.~1, 168--265.

\bibitem[AFR18]{AyalaFrancisRozenblyum}
D.~Ayala, J.~Francis, and N.~Rozenblyum, \emph{Factorization homology {I}: {H}igher categories}, Advances in Mathematics \textbf{333} (2018), 1042--1177.

\bibitem[Bar05]{BarwickPhD}
C.~Barwick, \emph{$(\infty,n)$-$\mathrm{Cat}$ as a closed model category}, Ph.D. thesis, University of Pennsylvania, 2005.

\bibitem[BDG{\etalchar{+}}16a]{Expose0}
C.~Barwick, E.~Dotto, S.~Glasman, D.~Nardin, and J.~Shah, \emph{Parametrized higher category theory and higher algebra: {A} general introduction}, arXiv:1608.03654, 2016.

\bibitem[BDG{\etalchar{+}}16b]{ExposeI}
\bysame, \emph{Parametrized higher category theory and higher algebra: {E}xpos\'e {I}}, arXiv:1608.03657, 2016.

\bibitem[BEH22]{BachmannElmantoHeller}
T.~Bachmann, E.~Elmanto, and J.~Heller, \emph{Motivic colimits and extended powers}, arXiv:2104.01057, 2022.

\bibitem[BGH20]{BarwickGlasmanHaine}
C.~Barwick, S.~Glasman, and P.~Haine, \emph{Exodromy}, arXiv:1807.03281, 2020.

\bibitem[BK75]{BorceuxKelly}
F.~Borceux and G.~M. Kelly, \emph{A notion of limit for enriched categories}, Bull. Aust. Math. Soc. \textbf{12} (1975), no.~1, 49--72.

\bibitem[Cam15]{CampionCounterEx}
T.~Campion, \emph{What is the point of pointwise {K}an extensions?}, MathOverflow question, 2015, \url{https://mathoverflow.net/questions/220246/what-is-the-point-of-pointwise-kan-extensions}.

\bibitem[Cis19]{Cisinski}
D.~Cisinski, \emph{Higher categories and homotopical algebra}, Cambridge Studies in Advanced Mathematics, vol. 180, Cambridge University Press, 2019.

\bibitem[GH15]{GepnerHaugseng}
D.~Gepner and R.~Haugseng, \emph{Enriched {$\infty$}-categories via non-symmetric {$\infty$}-operads}, Adv. Math. \textbf{279} (2015), 575--716.

\bibitem[GHN17]{GepnerHaugsengNikolaus}
D.~Gepner, R.~Haugseng, and T.~Nikolaus, \emph{Lax colimits and free fibrations in $\infty$-categories}, Documenta Mathematica \textbf{22} (2017), 1225--1266.

\bibitem[GP04]{GrandisPare}
M.~Grandis and R.~Par\'e, \emph{Adjoint for double categories}, Cah. Topol. G\'eom. Diff\'er. Cat\'eg. \textbf{45} (2004), no.~3, 193--240.

\bibitem[Gra74]{Gray}
J.~Gray, \emph{Formal category theory: {A}djointness for 2-categories}, Lecture Notes in Mathematics, vol. 391, Springer-Verlag, 1974.

\bibitem[Gui80]{Guitart}
R.~Guitart, \emph{Relations et carr\'es exacts}, Ann. Sci. Math. Qu\'ebec \textbf{4} (1980), no.~2, 103--125.

\bibitem[Hau13]{HaugsengPhD}
R.~Haugseng, \emph{Weakly enriched higher categories}, Ph.D. thesis, Massachusetts Institute of Technology, 2013.

\bibitem[Hau15]{HaugsengRect}
\bysame, \emph{Rectification of enriched {$\infty$}-categories}, Algebr. Geom. Topol. \textbf{15} (2015), no.~4, 1931--1982.

\bibitem[Hau16]{HaugsengEnriched}
\bysame, \emph{Bimodules and natural transformations for enriched $\infty$-categories}, Homology, Homotopy and Applications \textbf{18} (2016), no.~1, 71--98.

\bibitem[Hau17]{HaugsengMorita}
\bysame, \emph{The higher {M}orita category of {$E_n$}-algebras}, Geometry \& Topology \textbf{21} (2017), no.~3, 1631--1730.

\bibitem[Hau18]{HaugsengSpans}
\bysame, \emph{Iterated spans and classical topological field theories}, Mathematische Zeitschrift \textbf{289} (2018), no.~3, 1427--1488.

\bibitem[Hau22]{HaugsengEnds}
\bysame, \emph{On (co)ends in $\infty$-categories}, J. Pure Appl. Algebra \textbf{226} (2022), no.~2, 106819.

\bibitem[Hei24]{Heine}
H.~Heine, \emph{The higher algebra of weighted colimits}, arXiv:2406.08925, 2024.

\bibitem[Hei25]{HeineCatHom}
\bysame, \emph{On the categorification of homology}, arXiv:2505.22640, 2025.

\bibitem[HHK{\etalchar{+}}20]{EquivThom}
J.~Hahn, A.~Horev, I.~Klang, D.~Wilson, and F.~Zou, \emph{Equivariant nonabelian {P}oincar\'e duality and equivariant factorization homology of {T}hom spectra}, arXiv:2006.13348, 2020.

\bibitem[HHLN21]{HHLN}
R.~Haugseng, F.~Hebestreit, S.~Linskens, and J.~Nuiten, \emph{Lax monoidal adjunctions, two-variable fibrations and the calculus of mates}, arXiv:2011.08808, 2021.

\bibitem[HHR16]{HHR}
M.~A. Hill, M.~J. Hopkins, and D.~C. Ravenel, \emph{On the nonexistence of elements of {K}ervaire invariant one}, Ann. of Math. \textbf{184} (2016), no.~1, 1--262.

\bibitem[Hil66]{Hilton}
P.~Hilton, \emph{Correspondences and exact squares}, Proc. Conf. Categorical Algebra, La Jolla, Springer-Verlag, 1966, pp.~254--271.

\bibitem[Hil24]{Hilman}
K.~Hilman, \emph{An equivariant generalisation of {M}c{D}uff-{S}egal's group-completion theorem}, Int. Math. Res. Not. (2024), no.~9, 7552--7570.

\bibitem[Hin21]{Hinich}
V.~Hinich, \emph{Colimits in enriched $\infty$-categories and {Day} convolution}, arXiv:2101.09538, 2021.

\bibitem[JT07]{JoyalTierney}
A.~Joyal and M.~Tierney, \emph{Quasi-categories vs {Segal} spaces}, Categories in algebra, geometry and mathematical physics, Contemp. Math., vol. 431, Amer. Math. Soc., 2007, pp.~277--326.

\bibitem[Lac02]{Lack}
S.~Lack, \emph{A {Q}uillen model structure for 2-categories}, K-Theory (2002), no.~2, 171--205.

\bibitem[Lor18]{Loregian}
F.~Loregian, \emph{A {F}ubini rule for $\infty$-coends}, MPIM Preprints, preprint no. 68, 2018.

\bibitem[LR25]{GRconj}
F.~Loubaton and J.~Ruit, \emph{On the squares functor and the {G}aitsgory--{R}ozenblyum conjectures}, arXiv:arXiv:2507.07807, 2025.

\bibitem[Lur09a]{HTT}
J.~Lurie, \emph{Higher topos theory}, Annals of Mathematics Studies, vol. 170, Princeton University Press, 2009.

\bibitem[Lur09b]{LurieInfty2}
\bysame, \emph{$(\infty,2)$-{C}ategories and the {G}oodwillie calculus {I}}, arXiv:0905.0462, 2009.

\bibitem[Lur17]{HA}
\bysame, \emph{Higher algebra}, 2017, \url{https://www.math.ias.edu/~lurie/papers/HA.pdf}.

\bibitem[Mar21]{Martini}
L.~Martini, \emph{Yoneda's lemma for internal higher categories}, arXiv:2103.17141, 2021.

\bibitem[Mos20]{Moser}
L.~Moser, \emph{A double $(\infty,1)$-categorical nerve for double categories}, arXiv:2007.01848, 2020.

\bibitem[MW24]{MartiniWolf}
L.~Martini and S.~Wolf, \emph{Colimits and cocompletions in internal higher category theory}, High. Struct. \textbf{8} (2024), no.~1, 97--192.

\bibitem[Nar16]{Nardin}
D.~Nardin, \emph{Parametrized higher category theory and higher algebra: {E}xposé {I}{V}}, arXiv:1608.07704, 2016.

\bibitem[Rez01]{RezkSeg}
C.~Rezk, \emph{A model for the homotopy theory of homotopy theory}, Trans. Amer. Math. Soc. \textbf{353} (2001), no.~3, 973--1007.

\bibitem[Rez10]{RezkTheta}
\bysame, \emph{A cartesian presentation of weak {$n$}-categories}, Geom. Topol. \textbf{14} (2010), no.~1, 521--571.

\bibitem[Rie23]{RiehlUndergrads}
E.~Riehl, \emph{Could {$\infty$}-category theory be taught to undergraduates?}, Notices Amer. Math. Soc. \textbf{70} (2023), no.~5, 727--736.

\bibitem[Rov21]{Rovelli}
M.~Rovelli, \emph{Weighted limits in an $(\infty,1)$-category}, Appl. Categ. Structures \textbf{29} (2021), no.~6, 1019--1062.

\bibitem[RS17]{RiehlShulman}
E.~Riehl and M.~Shulman, \emph{A type theory for synthetic {$\infty$}-categories}, High. Struct. \textbf{1} (2017), no.~1, 147--224.

\bibitem[Rui24a]{EquipII}
J.~Ruit, \emph{Formal category theory in $\infty$-equipments {II}}, arXiv:2408.15190, 2024.

\bibitem[Rui24b]{Pasting}
\bysame, \emph{A pasting theorem for iterated {S}egal spaces}, J. Pure Appl. Algebra \textbf{228} (2024), no.~11, 107712.

\bibitem[Rui25]{Comp}
\bysame, \emph{Homotopy coherent companionships and conjunctions}, arXiv:2408.14335, 2025.

\bibitem[RV16]{RiehlVerityAdj}
E.~Riehl and D.~Verity, \emph{Homotopy coherent adjunctions and the formal theory of monads}, Advances in Mathematics \textbf{286} (2016), 802--888.

\bibitem[RV17]{RiehlVerityFormal}
\bysame, \emph{Kan extensions and the calculus of modules for {{\(\infty\)}}-categories}, Algebr. Geom. Topol. \textbf{17} (2017), no.~1, 189--271 (English).

\bibitem[RV22]{RiehlVerity}
\bysame, \emph{Elements of $\infty$-category theory}, Cambridge Studies in Advanced Mathematics, vol. 194, Cambridge University Press, 2022.

\bibitem[Sha23]{Shah}
J.~Shah, \emph{Parametrized higher category theory}, Algebr. Geom. Topol. \textbf{23} (2023), no.~2, 509--644.

\bibitem[Shu08]{ShulmanFramedBicats}
M.~Shulman, \emph{Framed bicategories and monoidal fibrations}, Theory Appl. Categ. \textbf{20} (2008), no.~18, 650--738.

\bibitem[Shu13]{ShulmanEnInCats}
\bysame, \emph{Enriched indexed categories}, Theory Appl. Categ. \textbf{28} (2013), 616--696.

\bibitem[Ste18]{Stevenson}
D.~Stevenson, \emph{Model structures for correspondences and bifibrations}, arXiv:1807.08226, 2018.

\bibitem[SW78]{StreetWalters}
R.~Street and R.~Walters, \emph{Yoneda structures on 2-categories}, Journal of Algebra \textbf{50} (1978), no.~2, 350--379.

\bibitem[{Uni}13]{hottbook}
The {Univalent Foundations Program}, \emph{Homotopy type theory: Univalent foundations of mathematics}, \url{https://homotopytypetheory.org/book}, Institute for Advanced Study, 2013.

\bibitem[Ver92]{Verity}
D.~Verity, \emph{Enriched categories, internal categories and change of base}, Ph.D. thesis, University of Cambridge, 1992.

\bibitem[Wol22]{Wolf}
S.~Wolf, \emph{The pro-\'etale topos as a category of pyknotic presheaves}, Doc. Math. \textbf{27} (2022), 2067--2106.

\bibitem[Woo82]{Wood}
R.J. Wood, \emph{Abstract proarrows {I}}, Cah. Topol. G\'eom. Diff\'er. Cat\'eg. \textbf{23} (1982), no.~3, 279--290.

\end{thebibliography}

\end{document}